\pdfoutput=1
\documentclass[12pt]{amsart}
\usepackage[usenames,dvipsnames]{xcolor}
\usepackage{amsmath,amssymb,amsthm,tikz,tikz-cd,ytableau,mathdots,amscd,tensor,stmaryrd}
\usepackage[all]{xy}
\usepackage{fullpage}
\usepackage[protrusion=true,expansion=true]{microtype}
\usepackage{float}
\usepackage{hyperref}
\usepackage{placeins}
\usepackage[T1]{fontenc}  
\usetikzlibrary{shapes,arrows,svg.path}
\usepackage{lmodern}
\usepackage{makecell}
\usepackage{xspace}
\usepackage{mathtools}
\usepackage{enumitem}
\usepackage{booktabs}
\usepackage{makecell}
\usepackage{caption}
\newtheorem{maintheorem}{Theorem}

\newtheorem{theorem}{\textbf{Theorem}}[section]
\newtheorem{proposition}[theorem]{\textbf{Proposition}}
\newtheorem{corollary}[theorem]{\textbf{Corollary}}
\newtheorem{lemma}[theorem]{\textbf{Lemma}}
\newtheorem{assumption}[theorem]{\textbf{Assumption}}
\theoremstyle{definition}
\newtheorem{remark}[theorem]{Remark}
\newtheorem{definition}[theorem]{Definition}
\newtheorem{convention}[theorem]{Convention}

\newtheorem{property}{Property}

\newcounter{Todo}

\newcommand{\ii}{\mathbf{i}}

\hypersetup{
    pdftitle={ Colored vertex models and Iwahori Whittaker functions },
    pdfauthor={Ben Brubaker, Valentin Buciumas, Daniel Bump, Henrik P. A. Gustafsson},
    linktocpage=true,           
    colorlinks=true,            
    linkcolor=blue!75!black,    
    citecolor=green!50!black,   
    filecolor=magenta,          
    urlcolor=blue!75!black      
}

\definecolor{darkgreen}{RGB}{0,180,0}
\definecolor{purple}{RGB}{138,43,226}
\definecolor{oceanmist}{RGB}{43,138,226}

\newcommand{\abs}[1]{\left|{#1}\right|}
\newcommand{\lie}[1]{\mathfrak{#1}}
\newcommand{\iso}{\cong}

\newcommand{\z}{\mathbf{z}}
\newcommand{\J}{\mathbf{J}}
\newcommand{\I}{\mathbf{I}}
\newcommand{\CC}{\mathbb{C}}
\newcommand{\RR}{\mathbb{R}}
\newcommand{\ZZ}{\mathbb{Z}}
\newcommand{\FF}{\mathbb{F}}

\newcommand{\wt}{\operatorname{wt}}
\newcommand{\GL}{\operatorname{GL}}
\newcommand{\Sp}{\operatorname{Sp}}
\newcommand{\End}{\operatorname{End}}

\DeclareRobustCommand{\plus}{\tikz[anchor=base, baseline=-1pt]{\node[scale=0.8, draw, thick, circle, inner sep=0pt] {$+$};}\xspace}
\DeclareRobustCommand{\minus}{\tikz[anchor=base, baseline=-1pt]{\node[scale=0.8, draw, thick, circle, inner sep=0pt] {$-$};}\xspace}

\tikzstyle{invisiblepath}=[draw=none, line width=1.5pt, spins]
\tikzstyle{path}=[line width=1.5pt, spins]
\tikzstyle{spins}=[every node/.style={spin}]
\tikzstyle{spin}=[circle, draw, fill=white, minimum size=14pt, inner sep=0pt, text=black]
\tikzstyle{halo}=[circle, fill=white, inner sep=0pt]
\tikzstyle{rectanglehalo}=[rectangle, fill=white, inner sep=0pt]

\usetikzlibrary{decorations.pathreplacing}

\newcommand{\icegrid}[2]{

    \foreach \i in {0,...,#1}{
        \pgfmathtruncatemacro\col{int(#1-\i)}
        \draw (2*\i+1,0) -- (2*\i+1,2*#2) node[spin, label=above:$\col$] {$+$};
    }

    \foreach \i in {1,...,#2}{
        \pgfmathtruncatemacro\row{int(#2-\i+1)}
        \draw (0,2*\i-1) node[spin, label=left:$\row$] {$+$} -- (2*#1+2, 2*\i-1);
    }

    \foreach \i in {0,...,#1}{
        \foreach \j in {1,...,#2}{
            \node[spin] at (2*\i+1, 2*\j-2) {$+$};
            \node[spin] at (2*\i+2, 2*\j-1) {$+$};
        }
    }
}

\newcommand{\icegridnorowlabels}[2]{

    \foreach \i in {0,...,#1}{
        \pgfmathtruncatemacro\col{int(#1-\i)}
        \draw (2*\i+1,0) -- (2*\i+1,2*#2) node[spin, label=above:$\col$] {$+$};
    }

    \foreach \i in {1,...,#2}{
        \pgfmathtruncatemacro\row{int(#2-\i+1)}
        \draw (0,2*\i-1) node[spin] {$+$} -- (2*#1+2, 2*\i-1);
    }

    \foreach \i in {0,...,#1}{
        \foreach \j in {1,...,#2}{
            \node[spin] at (2*\i+1, 2*\j-2) {$+$};
            \node[spin] at (2*\i+2, 2*\j-1) {$+$};
        }
    }
}

\newcommand{\icegridnocolumnlabels}[2]{

    \foreach \i in {0,...,#1}{
        \draw (2*\i+1,0) -- (2*\i+1,2*#2) node[spin] {$+$};
    }

    \foreach \i in {1,...,#2}{
        \pgfmathtruncatemacro\row{int(#2-\i+1)}
        \draw (0,2*\i-1) node[spin, label=left:$\row$] {$+$} -- (2*#1+2, 2*\i-1);
    }

    \foreach \i in {0,...,#1}{
        \foreach \j in {1,...,#2}{
            \node[spin] at (2*\i+1, 2*\j-2) {$+$};
            \node[spin] at (2*\i+2, 2*\j-1) {$+$};
        }
    }
}

\newcommand{\gammaunco}[5]{
    \draw[fill=white] (-.75,0) circle (.25);
    \draw[fill=white] (0,.75) circle (.25);
    \draw[fill=white] (.75,0) circle (.25);
    \draw[fill=white] (0,.-.75) circle (.25);
    \node at (-.75,0) {$#1$};
    \node at (0,.75) {$#2$};
    \node at (.75,0) {$#3$};
    \node at (0,.-.75) {$#4$};
    \path[fill=white] (0,0) circle (.25);
    \node at (0,0) {$#5$};
    \node at (0,1) { };}

\newcommand{\cross}{
      \draw (-.75,0)--(.75,0);
      \draw (0,-.75)--(0,.75);
}

\newcommand{\craa}{\draw (-.75,0)--(-.25,0);}
\newcommand{\crab}{\draw (0,.75)--(0,.25);}
\newcommand{\crac}{\draw (.75,0)--(.25,0);}
\newcommand{\crad}{\draw (0,-.75)--(0,-.25);}

\newcommand{\gamgamunco}[5]{
    \draw (-.75,-.75)--(.75,.75);
    \draw (-.75,.75)--(.75,-.75);
    \draw[fill=white] (-.75,-.75) circle (.28);
    \draw[fill=white] (-.75,.75) circle (.28);
    \draw[fill=white] (.75,.75) circle (.28);
    \draw[fill=white] (.75,.-.75) circle (.28);
    \node at (-.75,-.75) {$#1$};
    \node at (-.75,.75) {$#2$};
    \node at (.75,.75) {$#3$};
    \node at (.75,.-.75) {$#4$};
    \path[fill=white] (0,0) circle (.4);
    \node at (0,0) {$#5$};
    \node at (0,1) { };}

\newcommand{\spokea}[2]{
  \draw[line width=0.5mm, #1] (-.75,-.75)--(-.25,-.25);
  \draw[line width=0.5mm, #1,fill=white] (-.75,-.75) circle (.28);
  \node at (-.75,-.75) {$#2$};
}

\newcommand{\spokeb}[2]{
  \draw[line width=0.5mm, #1] (-.75,.75)--(-.25,.25);
  \draw[line width=0.5mm, #1,fill=white] (-.75,.75) circle (.28);
  \node at (-.75,.75) {$#2$};
}

\newcommand{\spokec}[2]{
  \draw[line width=0.5mm, #1] (.75,.75)--(.25,.25);
  \draw[line width=0.5mm, #1,fill=white] (.75,.75) circle (.28);
  \node at (.75,.75) {$#2$};
}

\newcommand{\spoked}[2]{
  \draw[line width=0.5mm, #1] (.75,-.75)--(.25,-.25);
  \draw[line width=0.5mm, #1,fill=white] (.75,.-.75) circle (.28);
  \node at (.75,-.75) {$#2$};
}

\newcommand{\bspokea}[2]{
  \draw[line width=0.5mm, #1] (-.75,-.75)--(-.25,-.25);
  \draw[line width=0.5mm, #1,fill=white] (-.75,-.75) circle (.28);
  \node [label=below:$\scriptstyle#2$] at (-.75,-.75) {};
}

\newcommand{\bspokeb}[2]{
  \draw[line width=0.5mm, #1] (-.75,.75)--(-.25,.25);
  \draw[line width=0.5mm, #1,fill=white] (-.75,.75) circle (.28);
  \node [label=above:$\scriptstyle#2$] at (-.75,.75) {};
}

\newcommand{\bspokec}[2]{
  \draw[line width=0.5mm, #1] (.75,.75)--(.25,.25);
  \draw[line width=0.5mm, #1,fill=white] (.75,.75) circle (.28);
  \node [label=above:$\scriptstyle#2$] at (.75,.75) {};
}

\newcommand{\bspoked}[2]{
  \draw[line width=0.5mm, #1] (.75,-.75)--(.25,-.25);
  \draw[line width=0.5mm, #1,fill=white] (.75,.-.75) circle (.28);
  \node [label=below:$\scriptstyle#2$] at (.75,-.75) {};
}

\newcommand{\spokepa}{
  \node at (-.75,-.75) {$+$};
}

\newcommand{\spokepb}{
  \node at (-.75,.75) {$+$};
}

\newcommand{\spokepc}{
  \node at (.75,.75) {$+$};
}

\newcommand{\spokepd}{
  \node at (.75,-.75) {$+$};
}

\newcommand{\vspokea}[2]{
  \draw[line width=0.5mm, #1] (-.75,0)--(-.25,0);
  \draw[line width=0.5mm, #1,fill=white] (-.75,0) circle (.25);
  \node at (-.75,0) {$#2$};
}

\newcommand{\vspokeb}[2]{
  \draw[line width=0.5mm, #1] (0,.75)--(0,.25);
  \draw[line width=0.5mm, #1,fill=white] (0,.75) circle (.25);
  \node at (0,.75) {$#2$};
}

\newcommand{\vspokec}[2]{
  \draw[line width=0.5mm, #1] (.75,0)--(.25,0);
  \draw[line width=0.5mm, #1,fill=white] (.75,0) circle (.25);
  \node at (.75,0) {$#2$};
}

\newcommand{\vspoked}[2]{
  \draw[line width=0.5mm, #1] (0,-.75)--(0,-.25);
  \draw[line width=0.5mm, #1,fill=white] (0,-.75) circle (.25);
  \node at (0,.-.75) {$#2$};
}

\newcommand{\pspokeb}{
  \draw[line width=0.5mm, blue] (-.075,.75)--(-.075,.25);
  \draw[line width=0.5mm, red] (.075,.75)--(.075,.25);
  \draw[line width=0.5mm, purple,fill=white] (0,.75) circle (.25);
  \node at (0,.75) {$\scriptstyle c,d$};
}

\newcommand{\pspoked}{
  \draw[line width=0.5mm, blue] (-.075,-.75)--(-.075,-.25);
  \draw[line width=0.5mm, red] (.075,-.75)--(.075,-.25);
  \draw[line width=0.5mm, purple,fill=white] (0,-.75) circle (.25);
  \node at (0,.-.75) {$\scriptstyle c,d$};
}

\newcommand{\cv}[2]{
  \draw[line width=0.5mm, #1, fill=white] (0,0) circle (.25);
  \node at (0,0) {$#2$};
}

\newcommand{\ybelhs}[9]{
\begin{tikzpicture}[baseline=(current bounding box.center)]
  \draw (0,1) to [out = 0, in = 180] (2,3) to (4,3);
  \draw (0,3) to [out = 0, in = 180] (2,1) to (4,1);
  \draw (3,0) to (3,4);
  \draw[fill=white] (0,1) circle (.3);
  \draw[fill=white] (0,3) circle (.3);
  \draw[fill=white] (3,4) circle (.3);
  \draw[fill=white] (4,3) circle (.3);
  \draw[fill=white] (4,1) circle (.3);
  \draw[fill=white] (3,0) circle (.3);
  \node at (0,1) {$#1$};
  \node at (0,3) {$#2$};
  \node at (3,4) {$#3$};
  \node at (4,3) {$#4$};
  \node at (4,1) {$#5$};
  \node at (3,0) {$#6$};
\path[fill=white] (1,2) circle (.3);
\node at (1,2) {$#7$};
\path[fill=white] (3,3) circle (.3);
\node at (3,3) {$#8$};
\path[fill=white] (3,1) circle (.3);
\node at (3,1) {$#9$};
\end{tikzpicture}}

\newcommand{\yberhs}[9]{
\begin{tikzpicture}[baseline=(current bounding box.center)]
  \draw (0,1) to (2,1) to [out = 0, in = 180] (4,3);
  \draw (0,3) to (2,3) to [out = 0, in = 180] (4,1);
  \draw (1,0) to (1,4);
  \draw[fill=white] (0,1) circle (.3);
  \draw[fill=white] (0,3) circle (.3);
  \draw[fill=white] (1,4) circle (.3);
  \draw[fill=white] (4,3) circle (.3);
  \draw[fill=white] (4,1) circle (.3);
  \draw[fill=white] (1,0) circle (.3);
  \node at (0,1) {$#1$};
  \node at (0,3) {$#2$};
  \node at (1,4) {$#3$};
  \node at (4,3) {$#4$};
  \node at (4,1) {$#5$};
  \node at (1,0) {$#6$};
\path[fill=white] (3,2) circle (.3);
\node at (3,2) {$#7$};
\path[fill=white] (1,1) circle (.3);
\node at (1,1) {$#8$};
\path[fill=white] (1,3) circle (.3);
\node at (1,3) {$#9$};
\end{tikzpicture}}

\begin{document}
\title{Colored vertex models and Iwahori Whittaker functions}
\author{Ben Brubaker}
\address{School of Mathematics, University of Minnesota, Minneapolis, MN 55455}
\email{brubaker@math.umn.edu}
\author{Valentin Buciumas}
\address{School of Mathematics and Physics, 
The University of Queensland, 
St. Lucia, QLD, 4072, 
Australia}
\email{valentin.buciumas@gmail.com}
\author{Daniel Bump}
\address{Department of Mathematics, Stanford University, Stanford, CA 94305-2125}
\email{bump@math.stanford.edu}
\author{Henrik P. A. Gustafsson}
\address{\hspace{-\parindent}\textnormal{Until September 15, 2019:} \newline
  \indent Department of Mathematics, Stanford University, Stanford, CA 94305-2125. \newline 
  \textnormal{Since September 16, 2019:} \newline
  \indent School of Mathematics, Institute for Advanced Study, Princeton, NJ~08540. \newline 
  \indent Department of Mathematics, Rutgers University, Piscataway, NJ~08854. \newline 
  \indent Department of Mathematical Sciences, University of Gothenburg and Chalmers University of Technology, SE-412~96 Gothenburg, Sweden.}
\email{gustafsson@ias.edu}


\begin{abstract}
  We give a recursive method for computing \textit{all} values of a basis of
  Whittaker functions for unramified principal series invariant under an Iwahori or parahoric
  subgroup of a split reductive group $G$ over a nonarchimedean local field
  $F$. Structures in the proof have surprising analogies to features
  of certain solvable lattice models. In the case $G=\mathrm{GL}_r$ we show
  that there exist solvable lattice models whose partition functions give
  precisely all of these values.  Here `solvable' means that the models have a
  family of Yang-Baxter equations which imply, among other things, that their
  partition functions satisfy the same recursions as those for Iwahori or
  parahoric Whittaker functions. The R-matrices for these Yang-Baxter
  equations come from a Drinfeld twist of the quantum group
  $U_q(\widehat{\mathfrak{gl}}(r|1))$, which we then connect to the standard
  intertwining operators on the unramified principal series.
  We use our results to connect Iwahori and parahoric Whittaker functions to
  variations of Macdonald polynomials.
\end{abstract}
\maketitle

\tableofcontents

\section{Introduction}

Solvable lattice models {\cite{Baxter,JimboMiwaAlgebraic}} are
statistical-mechanical systems (usually two-dimensional) that are amenable to
analysis using Yang-Baxter equations, highly constrained identities whose
mysterious nature led to the discovery of quantum groups
{\cite{Drinfeld,JimboHecke}}. Beyond their origins in
statistical mechanics, lattice models are also closely connected to quantum field theory
{\cite{BaxterHeisenberg,BPZStatistical,Costello}}, knot
invariants~{\cite{TuraevReshetikhin,JonesBaxterization,AkutsuDeguchiOhtsuki}}
and integrable probability~{\cite{BorodinPetrovLectures}}. Most importantly for us, they 
also give a fruitful method for studying symmetric function theory and its generalizations by representing
polynomials as partition functions of solvable lattice models, as for example
in~{\cite{FominKirillov,LLT-YBE,KuperbergASM,WheelerZinn-JustinAIM,WheelerZinn-JustinCrelle,KnutsonZinnJustinSchubert,KorffVerlinde}}.
In this context, the Yang-Baxter equation becomes a powerful tool for demonstrating identities
among partition functions.

Variants of these methods have been used by the authors and their collaborators {\cite{hkice,ivanov,mice,BBB,BBBGu,GrayThesis}}
to study the representation theory of algebraic groups and their covers over a $p$-adic field $F$. These papers use
families of solvable six-vertex models and their generalizations to produce partition functions giving 
special values of Whittaker functions for unramified principal series, particularly for the groups $\GL_r (F)$ and $\Sp (2 r, F)$ and
their metaplectic covers. The results shed a lot of light on the nature of these Whittaker functions.
However, these prior results have treated only \textit{spherical} Whittaker
functions, those invariant under a maximal compact subgroup. It is very
desirable to have lattice model interpretations for Whittaker functions
invariant under smaller compact subgroups.

\textit{Iwahori Whittaker functions} for an unramified principal series
representation are fixed by an Iwahori subgroup $J$ that is smaller than the maximal compact
subgroup $K$. Even if one is mainly interested in spherical ($K$-fixed)
Whittaker functions, the Iwahori Whittaker functions are needed. The
Iwahori Hecke algebra that acts on them has a richer structure and
the Iwahori Whittaker functions play a crucial role in evaluating the spherical Whittaker function in the work of Casselman and Shalika~\cite{CasselmanShalika}.
Moreover, the Iwahori Whittaker functions interestingly mirror the geometry of
the Schubert varieties in the flag variety~{\cite{ReederCompositio,BBL,MihalceaSu}}.
So understanding Iwahori Whittaker functions by bringing them into the
lattice model framework is an important goal. Once this connection is made,
we find that Whittaker functions invariant under certain \emph{parahoric} subgroups
(intermediate between $J$ and $K$) also fit naturally into these models, and
they too should be included in order to have a complete~story.

A recent breakthrough by Borodin and Wheeler
{\cite{BorodinWheelerColored,BorodinWheelerNonsymmetric}} showed how to refine
lattice models using an additional attribute they
called `color.' This led us to wonder whether a similar refinement of the
six-vertex model in {\cite{hkice}} might produce values of Iwahori Whittaker
functions for the general linear group.  The surprisingly complete
answer to this question for $\GL_r$, and related results about Whittaker
functions on {\it all} split reductive groups, are the subject of this paper.

Let us highlight two of our main results and their implications 
before a precise accounting of our results in the next section. 
Given a split, reductive group $G$, let $\mathbf{J}$ be a subset of the index set for simple 
reflections of the Weyl group $W$ of $G(F)$ and $K_{\mathbf{J}}$ the associated standard parahoric subgroup 
(as in Definition~\ref{definitionparahoric}). 
Note that $\mathbf{J} = \varnothing$ reduces to the Iwahori case $K_{\varnothing} = J$.
There exists a \emph{standard basis} $\{ \psi_w \}$ of $K_{\mathbf{J}}$-fixed
Whittaker functions for any irreducible unramified principal series, given explicitly in Section~\ref{sec:parahoric}.
In this paper we construct what we will call the \emph{parahoric lattice model} for which we prove the following theorem:

\begin{maintheorem}
  \label{thm:exists-lattice-model}
  For every $g$ in $\GL_r(F)$, any subset
  $\mathbf{J} \subseteq \{1,2,\ldots,r-1\}$, and for every $\psi_w$ in
a basis of $K_\mathbf{J}$-fixed Whittaker functions, there exists a choice of boundary conditions for the parahoric lattice model such that its partition function equals $\psi_w(g)$.
\end{maintheorem}

Conversely, we find that {\it every} admissible choice of boundary
conditions (in particular every assignment of colors to the boundary) has an important representation theoretic meaning.
If the colors used along the top boundary are distinct, then each choice of boundary data corresponds to
an element $g$ in $G(F)$ and a $w$ in the Weyl group such that the partition function
of this solvable lattice model is the value of the Iwahori Whittaker function $\phi_w(g)$. (See Theorem~\ref{coloredwhittaker},
Lemma~\ref{almostdominantlemma} and Remark~\ref{iwahoripg}.) If some 
colors are allowed to repeat, the resulting
partition functions correspond precisely to values of Whittaker functions for vectors fixed by
a parahoric subgroup (Theorem~\ref{thm:parahoric-lattice-model}).
The restriction of the parahoric lattice model to top boundary conditions with distinct colors will therefore be called the \emph{Iwahori lattice model}.
If instead the top boundary colors are all the same, we obtain a one-colored lattice model for the spherical Whittaker function, which is equivalent to the uncolored lattice model in~\cite{hkice} that we will here call the \emph{Tokuyama model}.

In order to prove Theorem~\ref{thm:exists-lattice-model}, we needed to extend
known results expressing Whittaker functions recursively using Demazure-like
operators, and these very general results (in
Sections~\ref{Iwahoripreliminaries} and~\ref{sec:parahoric}) are valid for any
split, reductive group. Thus we were led to prove the following result:

\begin{maintheorem} Let $G:=G(F)$ be any split reductive group defined over a
local field $F$. For a basis $\{\phi_w\}$ of the space of Iwahori Whittaker functions
for any irreducible unramified principal series representation of $G$ and any $g \in G$, there is a recursive algorithm using Demazure-like
operators to compute $\phi_w(g)$.
\end{maintheorem}

We have similar algorithmic results for parahoric Whittaker functions. (See
Theorem~\ref{thm:paracs} and Remark~\ref{rem:parcaveat}.)
Prior to this work, even a conjectural description of all values of the
Iwahori and parahoric Whittaker functions was unknown. Reeder 
\cite{ReederCompositio} and Brubaker, Bump, and Licata \cite{BBL} computed
Iwahori Whittaker functions if $g$ is a torus element. In the
introduction to \cite{ReederCompositio}, Reeder describes the determination of
the remaining values
as a difficult problem. Lansky~\cite{LanskyParahoric} determined the
dimensions of the spaces of standard parahoric fixed vectors in unramified
principal series.

In addition to the above theorems, the development of these results led us to:
\begin{itemize}[label=$\boldsymbol{\cdot}$, leftmargin=2\parindent]  
  \item Interpret the Demazure recurrence relations in our algorithmic
  description in terms of the R-matrix for the quantum affine superalgebra
  $U_{\sqrt{v}} (\hat{\mathfrak{g}\mathfrak{l}} (r| 1))$;

  \item Prove a Casselman-Shalika formula for certain
    parahoric Whittaker functions in Theorem~\ref{thm:paracs};
  
  \item Relate several classes of Whittaker functions to
    special functions that arise in algebraic combinatorics summarized in Table~\ref{tab:mini-nsmac} below and further detailed in Table~\ref{tab:nsmac} of Section~\ref{sec:nsmac}; 
  
  \item Interpret the action of intertwiners of Iwahori or parahoric fixed
  vectors in the principal series representation with $R$-matrices acting on
  evaluation representations of quantum loop groups
  in Section~\ref{sec:intertwiningR}.
\end{itemize}

Potential future applications include new Cauchy identities for these Whittaker
functions, which may then find application in local Rankin-Selberg computations. 
It may also be possible to relate our lattice models to interacting particle processes 
such as ASEP and TASEP in the spirit of similar results by Borodin and Wheeler \cite[Section~12]{BorodinWheelerColored}.
Furthermore in~{\cite{BBBGu}} we
found connections between the metaplectic models in~{\cite{BBB}} and LLT
polynomials (ribbon symmetric functions) and we expect that there may be
similar connections for this work.

Another exciting open question is whether 
Whittaker functions for even smaller compact subgroups might
similarly have representations as partition functions of
solvable lattice models. In other words, if $L$ is an
arbitrary compact open subgroup, does there exist a
basis of $L$-invariant Whittaker functions whose values
may be represented by solvable lattice models? If the answer
is affirmative one could also look for representations of
Whittaker of other admissible representations, particularly
supercuspidals. To what extent are solvable lattice models and their associated
quantum groups a tool for studying \emph{all} representations of $p$-adic groups?
These possibilities are purely speculative at this time, but
our success in treating Iwahori and parahoric Whittaker
functions and the work of Ju-Lee Kim~\cite{JuLeeKimCS} which gives combinatorial descriptions of Whittaker functions for other admissible representations is suggestive in this regard.

Finally, the models in this paper may also be generalized to refine the
models representing metaplectic Whittaker functions introduced
in~{\cite{BBB}}, called \textit{metaplectic ice}. Whereas for metaplectic ice
certain edges are enhanced by adding an attribute called charge, in
this paper we enhance the system by decorating certain edges with color. These
attributes of charge and color are handled differently, and at first glance
the schemes seem different. However we will show in a later paper that
they may be placed in a unified framework. The R-matrix for these solvable models is
related to $U_{\sqrt{v}}(\widehat{\mathfrak{gl}}(r|n))$.

\medbreak\noindent
\textbf{Acknowledgements:} We thank Amol Aggarwal, Alexei Borodin, Siddhartha Sahi and
Michael Wheeler for helpful conversations and communications. This work was
supported by NSF grants DMS-1801527 (Brubaker) and DMS-1601026 (Bump), ARC
grant DP180103150 (Buciumas) and the Knut and Alice Wallenberg Foundation
(Gustafsson).
During later stages of this work, Gustafsson was supported by 
by the Swedish Research Council (Vetenskapsr\aa det), grant 2018-06774.
Bump and Gustafsson also gratefully acknowledge the hospitality
of the Simons Center for Geometry and Physics during parts of this project.

\section{Outline of the paper}
\label{sec:outline}

Let us now explain the results of this paper in more detail. We begin with a
brief description of spherical Whittaker functions for $\GL_r (F)$ and their
associated six-vertex model, which we will call the \textit{Tokuyama model}.
Let $\mathfrak{o}$ be the ring of integers of the nonarchimedean local field
$F$ and let $v^{- 1}$ be the cardinality of the residue field. Construct an
unramified principal series representation of $\GL_r (F)$ from a character of
$T (F) / T (\mathfrak{o})$ where $T$ is the maximal split torus (see
Section~\ref{Iwahoripreliminaries} for full details). These representations
have unique Whittaker functionals and a unique-up-to-constant vector which is
right invariant under $K = \GL_r (\mathfrak{o})$. The spherical Whittaker
function is the image of this vector in the Whittaker functional and it is
completely determined by its values on $T (F) / T (\mathfrak{o})$, which we
identify with the weight lattice $\Lambda$ of the Langlands dual group. It is
easily seen that the spherical Whittaker function vanishes unless the
associated weight is dominant. The remaining values for dominant weights are
given by the Shintani-Casselman-Shalika formula in terms of Schur polynomials
in the Langlands parameters of the principal series. By Tokuyama's theorem,
described in {\cite{Tokuyama,HamelKing,hkice}} and recalled below as
(\ref{tokform}), there exists a solvable
six-vertex model with boundary conditions indexed by dominant weights whose partition function give
the Shintani-Casselman-Shalika formula.

There are multiple ways to describe this six-vertex model. A state of the system
attaches a spin \plus or \minus to each edge of a grid so that each vertex has
adjacent edges in one of six possible configurations seen for example in \cite[Table~1]{hkice}. Alternatively, the state
may be described by connecting the edges labeled with \minus into lines or
paths as in Section~8.1 of {\cite{Baxter}}. The models needed in this paper
refine the six vertex model by coloring these paths.

For the Iwahori lattice model in this paper the grid has $r$ rows and boundary
conditions that specify the particular $\GL_r$ Iwahori Whittaker function $\phi_w$ and group element
$g$.
Then a state of the system representing $\phi_w (g)$ will consist of $r$
paths with distinct colors beginning at the top boundary and traveling
downward and rightward until each path exits on a distinct row along the right
boundary. The partition function is the sum of the Boltzmann weights for all
such configurations of paths.
The Boltzmann weights can, for example, be found in the row labeled $S_\Gamma(i)$ of \cite[Table~2]{hkice}.

In Section~\ref{sec:some-YBEs}, we exhibit a set of colored Boltzmann weights
and prove Yang-Baxter equations for them in Theorem~\ref{thm:coloredybe}. In
Theorem~\ref{coloredwhittaker} we use these Yang-Baxter equations to
demonstrate that the partition function of the Iwahori lattice model for different boundary conditions is equal to values of
Iwahori Whittaker functions in the standard basis.

The Iwahori lattice model never has more than
one path of a given color, because the boundary conditions only have two edges
of a given color: one on the top boundary and one on the right. In a state of
the system, these are connected by a single path, so there is only one path of
that color. However, the Yang-Baxter equations, which are highly constrained
and overdetermined relations the Boltzmann weights must
satisfy, \textit{force} us to assign nonzero values to local vertex or path
configurations where two paths of the same color cross, even though these
configurations can never appear in any state of the Iwahori
model. Remarkably, these seemingly unneeded weights that we are forced to
include \textit{do} turn out to be needed when we turn to the more general
parahoric Whittaker functions discussed in Section~\ref{sec:parahoric}, and the corresponding parahoric lattice model of Section~\ref{sec:parahoric_systems} allowing for multiple paths of the same colors.

It is precisely these latter, parahoric weights that distinguish our models from models of Borodin and Wheeler
{\cite{BorodinWheelerColored}}, which also have configurations of multicolored paths
crossing in a grid. In both types of colored models, the horizontal edges can carry only one color, but the
vertical edges can carry more than one. In this paper, the vertical edges are
`fermionic' meaning that they satisfy an exclusion principle: no vertical edge
may carry more than one instance of a particular color. The models
in~{\cite{BorodinWheelerColored}} are `bosonic' and the vertical edges may
carry their colors with a multiplicity.  The parahoric lattice model makes clear the need for these fermionic weights in order to produce the correct
partition function describing a parahoric Whittaker function.

This distinction may also be observed from the quantum groups that underlie the
models. One may check that the R-matrix in Figure~\ref{coloredrmat} is a
Drinfeld twist of the R-matrix for the quantum affine Lie superalgebra
$U_{\sqrt{v}}(\widehat{\mathfrak{gl}}(r|1))$ from
\cite{Kojima}. By contrast the R-matrix for the systems in~\cite{BorodinWheelerColored}
is a twist of a quantum group $U_q(\widehat{\mathfrak{sl}}_{r+1})$. See also Remark~\ref{bwcomparisonremark} for
further details on the relationship between our models and those
of~{\cite{BorodinWheelerColored}}.

While our R-matrix is associated to a quantum superalgebra, we have no such
module interpretation for the Boltzmann weights of our colored models.
To demonstrate the Yang-Baxter equation, we must instead introduce an equivalent 
version of the systems in which
each vertex is replaced by $r$ vertices, and each vertical edge by
$r$ vertical edges. We refer these expanded systems as \textit{monochrome}
because each column is assigned a color, and
each vertical edge in that column can carry only that color and no other.
See Figure~\ref{oceanmisty} for an example of a colored model and
its equivalent expanded monochrome system. The relationship between
the Boltzmann weights for the regular colored systems and the expanded
monochrome systems is demonstrated in Figure~\ref{coloredfusion}.
This relationship is reminiscent of the \textit{fusion} construction for tensor products
of quantum group modules (see {\cite{KulishReshetikhinSklyanin,DAndecyFusion}} 
and Appendix~B of {\cite{BorodinWheelerColored}}) and our results in Section~\ref{sec:lattice-models} may be viewed as a
combinatorial substitute for fusion in the absence of a quantum group module interpretation.  
This link is the key to the solvability of the colored models, for it is with the monochrome vertices and weights that we prove the
Yang-Baxter~equations.

Apart from our results on solvable lattice models, our independent results on Iwahori and
parahoric Whittaker functions (i.e.\@ which are not depending on their relations to lattice model partition functions) are more precise than what is found in the
literature. Let us explain this point.

With notation as in Section~\ref{Iwahoripreliminaries}, our task is to study so-called `standard basis'
Iwahori Whittaker functions $\phi_{w_1} (g)$ for $w_1\in W$ and $g\in G$. Using left and right
translation properties of $\phi_{w_1}$, we may assume that
$g=\varpi^{- \lambda} w_2$ where $w_2$ is again a Weyl group element,
$\lambda$ is a weight and $\varpi$ is a uniformizer of $F$. 
Not every pair $\lambda,w_2$ needs to be considered,
since $\phi_{w_1}(g)$ vanishes for many values of $g$.
For example if $w_2=1$
then $\lambda$ must be dominant, meaning that
$\langle \alpha_i^\vee, \lambda \rangle \geqslant 0$ for all simple roots $\alpha_i$. If $w_2$ is a general permutation, then we
allow $\langle \alpha_i^\vee, \lambda \rangle\geqslant-1$ when $i$ is a descent of
$w_2^{-1}$ (see Definition~\ref{def:almost}).

In fact, for $G = \GL_r$, it turns out that the triples $w_1$, $w_2$, $\lambda$ with $w_2$ and $\lambda$
as above correspond exactly to the boundary conditions with distinct colors in our lattice model scheme! 
The permutations $w_1$ and $w_2$ describe the
order of the colors on the right and top boundaries; $\lambda$ describes the
columns on the top boundary where a colored line has its terminus; and the
cases where $\langle \alpha_i^\vee, \lambda \rangle = - 1$ is allowed correspond to the
possibility that a top vertical edge carries more than one color. This
correspondence between the data describing the values of Whittaker functions
and available systems in our scheme becomes even more striking when we
consider the parahoric systems in Section~\ref{sec:parahoric_systems}.

In order to prove all this we need new results on the values of
$\phi_{w_1}(g)$ with $g=\varpi^{- \lambda} w_2$ when $w_2 \neq 1$. In prior
work such as~{\cite{BBL,ReederCompositio,MihalceaSu}} the group element $g$
has usually been taken to be diagonal. We are able to give a recursive
method of computing the values $\phi_{w_1} (\varpi^{- \lambda} w_2)$ based on
Proposition~\ref{basecaseprop} and Theorem~\ref{mainrecursion}.
This then is used in Theorem~\ref{coloredwhittaker} to prove that $\phi_{w_1} (\varpi^{- \lambda} w_2)$ equals the partition function of the Iwahori lattice model with certain boundary conditions determined by $\lambda$, $w_1$ and $w_2$ as explained above.
The above methods to compute the Iwahori Whittaker functions, both the recursive algorithm (for any reductive group) and the partition function (for $\GL_r$), are easy to implement on a computer.
In the parahoric case, our results are also applicable to all group elements.

In Section~\ref{sec:nsmac}, we use our descriptions for Whittaker
functions in terms of Demazure-like divided difference operators, which arise
from our Yang-Baxter equations, to relate them to variations of Macdonald
polynomials. Indeed, we provide interpretations for spherical, Iwahori, and
parahoric Whittaker functions in terms of a specialization or generalization
of a Macdonald polynomial. In particular, the parahoric Whittaker functions
are expressed as Macdonald polynomials with prescribed symmetry studied
in \cite{BakerDunklForrester,Marshall,Baratta}.
Moreover, Jian-Shu
Li~\cite{JianShuLiUnramified} introduced a certain Iwahori Whittaker
function to study the unique genuine subquotient of the unramified principal series,
and we will show that its values are Hall-Littlewood polynomials.
See Table~\ref{tab:mini-nsmac} for a summary of these connections and relevant references.
Details of the exact relations are shown in Table~\ref{tab:nsmac} of Section~\ref{sec:nsmac}.

\begin{table}[htb]
  \centering
  \caption{Relations between Whittaker functions and special polynomials as detailed in Section~\ref{sec:nsmac}.
  In the third column we list references for the algebraic viewpoint on the Whittaker function and in the fourth column we give references for a lattice model interpretation.}
  \vspace{-0.25em}
  \label{tab:mini-nsmac}
    \scalebox{0.8}{
    \begin{tabular}{lclll}
    \toprule
    \thead{Whittaker function} &  & \thead{Special polynomial} & \thead{Algebraic} & \thead{Lattice models} \\[0.25em]
    \midrule
    \\[-0.5em]
    Spherical Whittaker function & = & Schur polynomial & \cite{Shintani, CasselmanShalika} & \cite{hkice, HamelKing} \\[0.25em] 
    Li's Whittaker function & = & Hall-Littlewood polynomial & \cite{JianShuLiUnramified} & \S\ref{sec:nsmac}, \cite{Tsilevich} \\[0.25em]
    Iwahori Whittaker function & = & {Non-symmetric  Macdonald polynomial} & \cite{BBL}, \S\ref{Iwahoripreliminaries} & \S\ref{coloredsystems} \\[0.25em]
    Parahoric Whittaker function & = & \makecell[lt]{Macdonald polynomial with \\ prescribed symmetry} & \S\ref{sec:parahoric}& \S\ref{sec:parahoric_systems}, \S\ref{sec:nsmac} \\
    \bottomrule
  \end{tabular}
  }
  \vspace{-0.25em}
\end{table}

Finally, in Section~\ref{sec:intertwiningR}, we explain how lattice models shed some further light on $p$-adic
representation theory.
Not only are the outputs of both the lattice models
and the $p$-adic representation theory the same, but each tool or
technique has a counterpart in this dictionary as we shall examplify now. Standard intertwining operators
on principal series are a basic tool in the representation theory of $p$-adic
groups. Their action on Iwahori fixed vectors and how they interact with the
Whittaker functional are the two principal ingredients in the
Casselman-Shalika formula {\cite{CasselmanSpherical,CasselmanShalika}} and are
also the key to Theorem~\ref{mainrecursion}. Roughly, we show that these
two actions of intertwining operators, on Iwahori fixed vectors and for the
Whittaker functional, correspond to restrictions of the quantum superalgebra
$U_{\sqrt{v}} (\widehat{\mathfrak{gl}} (r| 1))$ to its $U_{\sqrt{v}}
(\widehat{\mathfrak{gl}} (r))$ and $U_{\sqrt{v}} (\widehat{\mathfrak{gl}}
(1))$ pieces, respectively.

In Theorem~\ref{RmatrixintertwiningIwahori} we show that the action of the
intertwining integral on the space of Iwahori fixed vectors is the same as the
action of the affine R-matrix on a subspace of the tensor product of
evaluation representations $V_r (z_1) \otimes \cdots \otimes V_r (z_r)$ of
$U_{\sqrt{v}} (\widehat{\mathfrak{gl}} (r))$. This result is independent of
the Whittaker functional and only the smaller quantum group $U_{\sqrt{v}}
(\widehat{\mathfrak{gl}} (r)) \subset U_{\sqrt{v}} (\widehat{\mathfrak{gl}}
(r| 1))$ appears due to the fact that the right boundary conditions of our
model contain only colored edges (which span a subspace that can be thought of
as the tensor product of evaluation representations of $U_{\sqrt{v}}
(\widehat{\mathfrak{gl}} (r))$). This result can be easily generalized to the
parahoric setting.

A result similar to Theorem~\ref{RmatrixintertwiningIwahori} was proved in the
case of spherical Whittaker functions on the metaplectic $n$-cover of $\GL_r$
in {\cite[Theorem 1.1]{BBB}}, where the first three authors relate the
Kazhdan-Patterson scattering matrix to the $U_{\sqrt{v}}
(\widehat{\mathfrak{g}\mathfrak{l}} (n))$ R-matrix. The relation was used
in~{\cite{BBBF}} to build finite dimensional representations of the affine
Hecke algebra starting from metaplectic Whittaker functionals.
Theorem~\ref{RmatrixintertwiningIwahori} now allows for a similar construction
starting from Iwahori fixed vectors in an unramified principal series
representation.

\section{Iwahori Whittaker functions}\label{Iwahoripreliminaries}

We will review the constructions of Iwahori Whittaker functions following~\cite{BBL}.
There are several differences between choices made here and in~\cite{BBL} with those
in Casselman-Shalika~\cite{CasselmanShalika}. Let us summarize these choices, with
notations to be defined more precisely below.
\begin{itemize}[label=$\boldsymbol{\cdot}$, leftmargin=*]
  \item As in~\cite{CasselmanShalika}, principal series representations are induced from
the standard Borel subgroup $B$. But in contrast with~\cite{CasselmanShalika},
we will take Whittaker functions with respect to the unipotent radical $N_-$ of the
 opposite Borel subgroup $B_-$.
 \item We will take our Iwahori subgroup $J$ to be the
preimage in the maximal compact subgroup $K$ of $B_-$ modulo $\mathfrak{p}$.
 \item We will apply our construction to the contragredient
representation of the principal series with Langlands
parameters $\mathbf{z}$.
  \item When restricting to the maximal torus we will evaluate our Whittaker functions at
values $\varpi^{-\lambda}$ where $-\lambda$ is
antidominant. 
\end{itemize}

The advantage of these unconventional choices is that it keeps
the long Weyl group element $w_0$ out of the formulas. Thus whereas for Casselman
and Shalika the simplest Whittaker function is that supported on the double
coset $Bw_0J$, and its value at $\varpi^{\lambda}$ is (up to normalization),
$\mathbf{z}^{w_0\lambda}$, with our conventions the simplest Whittaker
function is supported on $B\cdot 1_WJ$, and its value is (up to normalization)
$\mathbf{z}^{\lambda}$.

In more detail,
let $F$ be a non-archimedean local field with ring of integers $\mathfrak{o}$. Let $\mathfrak{p}$ be the maximal ideal of $\mathfrak{o}$ with generator $\varpi \in \mathfrak{p}$. Then, $\varpi$ is a prime element, or uniformizer, of $F$. We will denote by $q$ the cardinality $q = \abs{\mathfrak{o}/\mathfrak{p}}$ and the residue field itself by $\mathbb{F}_q = \mathfrak{o}/\mathfrak{p}$.

Let $G$ be a split reductive Chevalley group, that is, an affine algebraic group scheme over $\ZZ$ with a fixed Chevalley basis for its Lie algebra $\lie g_\ZZ$.
Let $T$ be the standard maximal split torus of $G$ obtained from our choice of
Chevalley basis, and similarly let $N$ be the standard maximal unipotent
subgroup whose Lie algebra is the union of the positive root spaces. Together
they form the standard Borel subgroup $B = TN$ and the Weyl group $W$ is
defined by $N_G(T) / T$ where $N_G(T)$ is the normalizer of $T$ in $G$.

\begin{remark}\label{weylrep}
For each Weyl group element $w$ we will always choose a representative in
$ K = G(\mathfrak{o})$ that is the maximal compact subgroup of $G(F)$.
Strictly speaking this representative is only determined modulo $T(\mathfrak{o})$.
However because we are considering representations induced from unramified
data, none of the functions we compute ever depend on this choice, nor
on the choice of $\varpi$.
\end{remark}

Let $B_-$ be the opposite Borel subgroup and $N_-$ be its unipotent radical
generated by the negative root spaces. In the later sections of this paper we
will mainly consider $G = \GL_r$ for which $B$ is the subgroup of upper
triangular matrices, $T$ the diagonal matrices and $B_-$ the lower triangular
matrices.

Let $\hat G$ be the Langlands dual group of $G$. We will denote the root
system of $\hat G$ by $\Delta$ and the simple roots of $\hat G$ by $\alpha_1,
\ldots, \alpha_r$. The root system of $G$ is the dual root system $\Delta^\vee$.
We prefer this notation instead of making $\Delta$ the root system of $G$,
because the weight lattice $\Lambda$ of $\hat{G}$ appears
frequently in the sequel.

We will consider an unramified character $\tau$ of $T(F)$, that is, a character that is trivial on $T(\mathfrak{o})$. The group of such characters is isomorphic to $\hat T(\CC) \iso (\CC^\times)^r$, where $\hat T$ is the standard split maximal torus of $\hat G$. To define the unramified character $\tau_\z$ for $\z \in \hat T(\CC)$ we will use the following isomorphisms. 

The group $X_*(T)$ of rational cocharacters of $T$ is isomorphic to the
weight lattice $\Lambda=X^*(\widehat{T})$ of rational characters of the dual torus,
and we will identify these two groups. But $X_*(T)$ is
also isomorphic to the quotient $T(F)/T(\mathfrak{o})$. Indeed,
if $\lambda$ is a cocharacter let $\varpi^\lambda$ be the image
of the uniformizer $\varpi$ in $T$ under $\lambda$; then
we associate with $\lambda$ the coset $\varpi^\lambda T(\mathfrak{o})$
in $T(F)/T(\mathfrak{o})$. On the other hand 
we may regard $\lambda$ as a rational character 
and, with $\mathbf{z}\in \hat T(\mathbb{C})$, let $\mathbf{z}^\lambda \in F^\times$ be the application of this character
to $\mathbf{z}$. Then we define the unramified character $\tau_\z$ of $T(F)$
by $\tau_\z(t) = \z^\lambda$ when $t\in\varpi^\lambda T(\mathfrak{o})$.

In particular, for $G = \GL_r$ with $\lambda = (\lambda_1, \ldots, \lambda_r) \in \ZZ^r \iso \Lambda$ and $\z = (z_1, \ldots, z_r) \in (\CC^\times)^r \iso \hat T(\CC)$ we let
\begin{equation*}
  \varpi^\lambda = 
  \begin{psmallmatrix}
    \varpi^{\lambda_1} & & & \\
    & \varpi^{\lambda_2} & & \\
    & & \ddots & \\
    & & & \varpi^{\lambda_r}
  \end{psmallmatrix} \in \GL_r(F) \text{ and } \tau_\z(\varpi^\lambda) = \z^\lambda = \prod_{i=1}^r z_i^{\lambda_i}\, .
\end{equation*}

The Iwahori subgroup $J$ of $G(F)$ is the subgroup of $K = G(\mathfrak{o})$ defined as the preimage of $B_-(\FF_q)$ under the mod $\mathfrak{p}$ reduction map $K \to G(\FF_q)$. For $G = \GL_r$ the Iwahori subgroup consists of elements in $\GL_r(\mathfrak{o})$ which are lower triangular mod $\mathfrak{p}$.

We trivially extend an unramified character $\tau_\z$ of $T(F)$ to $B(F)$ and
let $(\pi, I(\z))$ denote the induced representation $I(\z)
= \operatorname{Ind}_B^G(\delta^{1/2}\tau_\z)$ under the right-regular action
$\pi$ of $G(F)$ where $\delta : B(F) \to \RR^\times$ is the modular
quasicharacter.

\begin{remark}
  \label{rem:deltaex}
  It is convenient to extend $\delta$ to a function on $G(F)$ as
  follows using the Iwasawa decomposition. If $g\in G(F)$ we
  may write $g=bk$ where $b\in B(F)$, $k\in K$ and we define
  $\delta(g)=\delta(b)$. This is well-defined since $\delta$
  is trivial on $B(F)\cap K$.
\end{remark}

Consider the space $I(\z)^J$ of Iwahori fixed vectors in $I(\z)$ which is of dimension $\dim I(\z)^J = \abs{W}$. We will now describe a basis for $I(\z)^J$ which will be used throughout the paper. By combining the Bruhat decomposition $G(F) = \bigsqcup_{w \in W} B(F) w B(F)$ and the Iwahori factorization $J = N(\mathfrak{p}) T(\mathfrak{o}) N_-(\mathfrak{o})$ one can show that $G = \bigsqcup_{w \in W} B(F) w J$ \cite{CasselmanSpherical}. Then, the elements $\Phi_w^\z \in I(\z)^J$ for $w \in W$ defined by
\begin{equation*}
  \Phi_w^\z(bw'k) :=
  \begin{cases}
    \delta^{1/2} \tau_\z(b) & \text{if } w' = w \\
    0 & \text{otherwise}
  \end{cases}
  \qquad b \in B(F), w' \in W, k \in J
\end{equation*}
form a basis of $I(\z)^J$, commonly referred to as the `standard basis.'

For $\alpha \in \Delta$, let $x_{\alpha} : \mathbb{G}_a \to G$ be the
one-parameter subgroup of $G$ corresponding to $\alpha^\vee$. (We
recall that $\Delta$ is the root system of the dual group $\widehat G$,
and it is the coroot $\alpha^\vee$ that is a root of $G$.)
Thus $x_\alpha(t)=\exp(tX_\alpha)$ where $X_\alpha$ is the corresponding
Chevalley basis element of the Lie algebra. The group $K=G(\mathfrak{o})$
is generated by the unipotent groups $x_\alpha(\mathfrak{o})$.
Fix a unitary character
$\psi$ on $N_-(F)$ such that, for any simple root $\alpha$, $\psi \circ x_{-\alpha} :
F \to \CC^\times$, is a character on $F$ trivial on $\mathfrak{o}$ but no
larger fractional ideal. The space of Whittaker
functionals, which are linear maps $\Omega_\z : I(\z) \to \CC$ satisfying
$\Omega_\z(\pi(n_-)f) = \psi(n_-)
\Omega_\z(f)$ for $n_- \in N_-(F)$, is one dimensional \cite{Rodier}. We need
therefore only consider the following explicit Whittaker functional
\begin{equation}
  \label{eq:whittaker-functional}
  \Omega_\z(f) := \int\limits_{N_-(F)} f(n) \psi(n)^{-1} \, dn \qquad f \in I(\z) \, .
\end{equation}
The integral is convergent if $|\mathbf{z}^\alpha|<1$ for positive roots
$\alpha$, and can be extended to all $\mathbf{z}$ by analytic continuation.

The objects of study in this paper are the \textit{Iwahori Whittaker
functions} obtained by applying the Whittaker
functional \eqref{eq:whittaker-functional} to right-translates of standard
basis elements $\Phi_w^\z$.

\begin{remark}\label{iwahoripg}
  Any $g\in G(F)$ may be written as $g=n\varpi^{-\lambda} w_2 k$ with
  $n\in N_-(F)$, $\lambda\in\Lambda$, $w_2$ a Weyl group representative
  chosen in $K$ by Remark~\ref{weylrep}, and $k\in J$. 
  Using the left $N_-(F)$ equivariance and
  the right $J$ invariance, any Iwahori Whittaker function $\mathcal{W}(g)$
  then satisfies $\mathcal{W}(g)=\psi(n)\,\mathcal{W}(\varpi^{-\lambda}w_2)$.
  Thus we reduce to computing our Whittaker functions at values~$\varpi^{-\lambda} w_2$.
\end{remark}

We will use the following conventions and normalizations for Whittaker
functions of the contragredient $I(\z^{-1})$ of $I(\z)$ at these values. For $w_1 \in W$ we consider the Iwahori Whittaker function
\begin{equation}
  \label{eq:iwahori-whittaker-function}
   \phi_{w_1}(\z; g) := 
\delta^{1/2}(g)\Omega_{\z^{-1}}\bigl(\pi(g) \Phi_{w_1}^{\z^{-1}}\bigr)
\end{equation}
which is determined by its values on $g = \varpi^{-\lambda} w_2$ for a
weight $\lambda \in \Lambda$ and $w_2 \in W$.
See Remark~\ref{rem:deltaex} for the extension of $\delta$ to $G(F)$. 
In \cite{BBL} these Whittaker functions were considered only for torus elements $g = \varpi^{-\lambda}$ and were there denoted as $\mathcal{W}_{\lambda, w_1}(\z)$. In this paper
we treat the general case, not only $w_2 =1$, and will therefore need to
compute more generally $\phi_{w_1}(\z; \varpi^{-\lambda} w_2)$.

First, we determine the pairs $\lambda$ and $w_2$ for which this $J$-invariant Whittaker function is non-vanishing.

\begin{definition}\label{def:almost}
We say that $\lambda$ is \textit{$w$-almost dominant} if for
all simple roots $\alpha_i$
\begin{equation}
\label{waldomdef}
\left\{\begin{array}{ll}\langle\alpha_i^\vee,\lambda\rangle\geqslant 0 &\text{if $w^{-1}\alpha_i\in\Delta^+$,}\\
\langle\alpha_i^\vee,\lambda\rangle\geqslant -1 &\text{if
  $w^{-1}\alpha_i\in\Delta^-$.}\end{array}\right.
\end{equation}
\end{definition}

\begin{lemma} \label{almostdominantlemma} 
  Let $\mathcal{W}$ be any $J$-invariant Whittaker function. Then
  \[ \mathcal{W}(\varpi^{-\lambda} w_2) = 0\]
  unless $\lambda$ is $w_2$-almost dominant.
\end{lemma}

\begin{proof} 
This is similar to Lemma 5.1 of \cite{CasselmanShalika}.
Let $\alpha_i$ be a simple root such that (\ref{waldomdef}) fails.
We may find $t\in\mathfrak{p}^{-1}$ such that $\psi(u)\neq1$
where $u=x_{-\alpha_i}(t)$. Now
\begin{equation}
  \label{eq:Whittaker-translation}
  \psi(u)\mathcal{W}(\varpi^{-\lambda} w_2)=\mathcal{W}(u\varpi^{-\lambda} w_2)=\mathcal{W}(\varpi^{-\lambda} w_2j)
\end{equation}
where
\[j=w_2^{-1}\varpi^\lambda u\varpi^{-\lambda}w_2=x_{-w_2^{-1}(\alpha_i)}(\varpi^{-\langle\alpha_i^\vee,\lambda\rangle}t).\]
Our assumption that (\ref{waldomdef}) fails implies that 
$\varpi^{-\langle\alpha_i^\vee,\lambda\rangle}t\in\mathfrak{o}$ if $w_2^{-1}(\alpha)\in\Delta^+$
and $\varpi^{-\langle\alpha_i^\vee,\lambda\rangle}t\in\mathfrak{p}$ if $w_2^{-1}(\alpha)\in\Delta^-$
and in either case $j\in J$, so
$\mathcal{W}(\varpi^{-\lambda} w_2j)$ equals $\mathcal{W}(\varpi^{-\lambda} w_2)$ which
must therefore vanish by~\eqref{eq:Whittaker-translation}.
\end{proof}

Next we analyze the special case $w_1 = w_2$.
To any $w \in W$, let $\Delta^+_w$ denote the set of
positive roots $\{ \alpha \in \Delta^+ \; | \; w(\alpha) \in \Delta^- \}$.

\begin{proposition} \label{basecaseprop} Let $w \in W$ and $\lambda \in \Lambda$ a $w$-almost dominant weight. Then
$$  \phi_{w}(\z; \varpi^{-\lambda} w) = q^{-\ell(w)} \mathbf{z}^\lambda, $$
where $\ell(w)$ denotes the length of a reduced expression for $w$.
\end{proposition}

\begin{proof}
  By definition
  \[ \phi_w (\mathbf{z}; \varpi^{- \lambda} w) = \delta^{- 1 / 2}
     (\varpi^{\lambda}) \int_{N_- (F)} \Phi^{\mathbf{z}^{- 1}}_w (n
     \varpi^{- \lambda} w) \psi (n)^{- 1} d n. \]
  We make the variable change $n \mapsto \varpi^{- \lambda} n
  \varpi^{\lambda}$. This multiplies the measure by $\delta
  (\varpi^{\lambda})$ and using $\Phi_w^{\mathbf{z}^{- 1}} (\varpi^{-
  \lambda} g) = \delta^{1 / 2} (\varpi^{- \lambda}) \mathbf{z}^{\lambda}
  \Phi_w^{\mathbf{z}^{- 1}} (g)$ we get
  \[ \mathbf{z}^{\lambda} \int_{N_- (F)} \Phi^{\mathbf{z}^{- 1}}_w (n w)
     \psi (\varpi^{- \lambda} n \varpi^{\lambda})^{- 1} d n. \]
  Let $J_w = w J w^{- 1}$. This has the Iwahori factorization $J_w = N_w^- T
  (\mathfrak{o}) N_w^+$ where $N_w^- = J_w \cap N_- (F)$ and similarly for
  $N_w^+$. In particular
  \begin{equation}
    \label{nmindef} N_w^- = \prod_{\alpha \in \Delta^+} \left\{
    \begin{array}{ll}
      x_{- \alpha} (\mathfrak{o}) & \text{if $w^{- 1} \alpha \in
      \Delta^+$\;,}\\
      x_{- \alpha} (\mathfrak{p}) & \text{if $w^{- 1} \alpha \in \Delta^-$\;.}
    \end{array} \right.
  \end{equation}
  The integrand is nonzero only if $n w \in B w J$. We will show that this is
  true if and only if $n \in N_w^-$. Indeed, write $n w = b w j$ where $j \in
  J$. Then $n = b j_w$ where $j_w = w j w^{- 1} \in J_w$. Using the Iwahori
  factorization, $j_w = \beta n_w^-$ where $\beta \in B$ and $n_w^- \in N_w^-$.
  Because $B \cap N_- = \{1\}$, $b = \beta = 1$ and $n = n_w^-$. Therefore the
  integral equals
  \[ \mathbf{z}^{\lambda} \int_{N_w^-} \Phi^{\mathbf{z}^{- 1}}_w (n w)
     \psi (\varpi^{- \lambda} n \varpi^{\lambda})^{- 1} d n. \]
  Now we will show that the value of the integrand is $1$ so this is just
  $\mathbf{z}^{\lambda}$ times the volume of $N_w^-$. We have
  $\Phi_w^{\mathbf{z}^{- 1}} (n w) = 1$ since the argument is in $w J$. We
  must show that $\varpi^{- \lambda} n \varpi^{\lambda}$ is in the kernel of
  $\psi$. For this it is sufficient to show that if $\alpha = \alpha_i$ is a
  simple positive root then
  \[ \varpi^{- \lambda} x_{- \alpha_i} (t) \varpi^{\lambda} \in N_-
     (\mathfrak{o}) \]
  where using (\ref{nmindef}) we may assume that $t \in \mathfrak{o}$ if $w^{-
  1} (\alpha_i) \in \Delta^+$ and $t \in \mathfrak{p}$ otherwise. Now
  \[ \varpi^{- \lambda} x_{- \alpha_i} (t) \varpi^{\lambda} = x_{- \alpha_i}
     (\varpi^{\langle \lambda, \alpha_i^{\vee} \rangle} t) . \]
  Because $\lambda$ is $w$-almost dominant $\varpi^{\langle \lambda,
  \alpha_i^{\vee} \rangle} t$ is indeed in $\mathfrak{o}$.
  
  Hence $\phi_w (\mathbf{z}; \varpi^{- \lambda} w)$ equals
  $\mathbf{z}^{\lambda}$ times the volume of $N_w^-$, which is $q^{-\ell
  (w)}$. 
\end{proof}

In order to determine the values of the Iwahori Whittaker function
$\phi_{w_1}(\z; \varpi^{-\lambda} w_2)$ in full generality, we mimic the
methods of \cite{BBL}, which used ingredients from earlier papers of Casselman
and Shalika \cite{CasselmanSpherical, CasselmanShalika}. In brief, we will
develop a recursion using the Bruhat order in the Weyl group in the $w_1$
variable above, whose base case is given by
Proposition~\ref{basecaseprop}. The recursion results from computing the
function $\Omega_{\mathbf{z}}(\mathcal{A}_{s_i} \cdot \Phi_w)$ in two ways,
where $\mathcal{A}_w$ denotes the standard intertwining operator on principal
series corresponding to the Weyl group element $w \in W$ and $s_i$ is a simple
reflection. Comparing the two methods of computation will give the
values of the Whittaker function. We begin by briefly reviewing the basics of
intertwining operators. These facts will also be needed in
Section~\ref{sec:intertwiningR}. To avoid technical problems
with the poles and zeros of the intertwining integrals,
we will assume that $\mathbf{z}^{\alpha^\vee}\neq 1,q^{\pm1}$
for all $\alpha$.

The standard intertwining integral $\mathcal{A}^\z_{w}:I(\z) \to I(w\z)$ is given by
\begin{equation}
  \label{eq:intertwiner}
  \mathcal{A}_w^\z \Phi(g) = \int_{N(F) \cap wN_-(F)w^{-1}} \Phi(w^{-1}ng)dn.
\end{equation}
The integral converges when $|\z^\alpha|<1$ for $\alpha \in \Delta^+$ and can be extended to arbitrary $\z$ by meromorphic continuation. The intertwining integral induces a map $\mathcal{A}_{w}^\mathbf{z} : I(\z)^J \to I(w \z)^J$ and an explicit expression for $\mathcal{A}_{s_i}^\z$ on $I(\z)^J$ 
is given by the following formula. See Proposition~3 in \cite{BBL} for a proof
of this fact, which is equivalent to Theorem~3.4 of~\cite{CasselmanSpherical}:
\begin{equation}\label{intertwiningaction}
\mathcal{A}_{s_i}^\z (\Phi^\z_{w}) =      \begin{cases}
       (1-c_{\alpha_i}(s_i \z) )\Phi_w^{s_i \z} + \Phi^{s_i \z}_{s_i w} &\quad\text{if } \ell(s_i w) > \ell(w),\\
       (q^{-1} - c_{\alpha_i}(s_i \z)) \Phi^{s_i \z}_w+ q^{-1} \Phi^{s_i \z}_{s_i w} &\quad\text{if } \ell(s_i w) < \ell(w), \\
     \end{cases}
\end{equation}
where
\begin{equation}\label{calphai}
  c_{\alpha_i}(\z) = \frac{1-q^{-1} \z^{\alpha_i}}{1-\z^{\alpha_i}}.
\end{equation}

Substituting in the definition of $c_{\alpha_i}$ and using that $(s_i \z)^{\alpha_i} = \z^{-\alpha_i}$ we get that equation~\eqref{intertwiningaction} is equivalent to  
\begin{equation}\label{intertwiningactionexplicit}
\bar{\mathcal{A}}_{s_i}^\z  (\Phi^\z_{w}) =     \begin{cases}
  \frac{1-q^{-1}}{1-q^{-1}\z^{\alpha_i}} \Phi_w^{s_i \z} + \frac{1-\z^{\alpha_i}}{1-q^{-1}\z^{\alpha_i}} \Phi^{s_i \z}_{s_i w} &\quad\text{if } \ell(s_i w) > \ell(w),\\
  \z^{\alpha_i} \frac{1-q^{-1}}{1-q^{-1}\z^{\alpha_i}} \Phi^{s_i \z}_w+ q^{-1}\frac{1-\z^{\alpha_i}}{1-q^{-1}\z^{\alpha_i}}\Phi^{s_i \z}_{s_i w} &\quad\text{if } \ell(s_i w) < \ell(w), \\
     \end{cases}
\end{equation}
where $\bar{\mathcal{A}}_{s_i}^\z := \frac{1-\z^{\alpha_i}}{1-q^{-1}\z^{\alpha_i}}\mathcal{A}_{s_i}^\z$.
The normalized version of the intertwiner $\bar{\mathcal{A}}_{s_i}^\z$ is sometimes preferred because it is a cocycle, i.e. 
\[ \bar{\mathcal{A}}_{s_j s_i}^\z = \bar{\mathcal{A}}_{s_j}^{s_i \z} \circ \bar{\mathcal{A}}_{s_i}^\z . \]
In Proposition~2 of \cite{BBL}, following from Proposition 4.3 of
\cite{CasselmanShalika}, the following result is proven. 

\begin{proposition}\label{csinterident} For any $w \in W$,
\begin{equation}\label{intertwiningactionexplicitspherical}
\Omega_{w \z} \circ \mathcal{A}_w^\z = \Bigl( \prod_{\alpha \in \Delta^+_w} \frac{1- q^{-1} \z^{-\alpha}}{1-\z^\alpha} \Bigr) \Omega_\z. 
\end{equation}
\end{proposition}

In Theorem~\ref{mainrecursion} below we will combine the above results to obtain a recursion relation for $\phi_w(\z; g)$ using Demazure type operators that we will define now. 

Let $\mathcal{O}(\hat T)$ be the ring of regular (polynomial) functions on
$\hat{T}(\mathbb{C}) \iso (\CC^{\times})^r$. This ring is isomorphic to the group algebra of
$\Lambda =\mathbb{Z}^r$ as follows. If $\z= (z_1, \cdots, z_r) \in
(\CC^{\times})^r$ and $\lambda \in \Lambda$ let $\z^{\lambda} = \prod
z_i^{\lambda_i}$. Then $\mathcal{O}(\hat T)$ is spanned by the functions
$\z^{\lambda}$. 
We may now define operators $\mathfrak{T}_i$ on $\mathcal{O}(\hat T)$ as follows.
Let $v$ be a nonzero complex number and, for $f \in \mathcal{O}(\hat T)$, let
\begin{equation}
  \label{demwhitop} \mathfrak{T}_{i,v} f (\z) = \mathfrak{T}_i f (\z) = \frac{f (\z) -
  f (s_i \z)}{\z^{\alpha_i} - 1} - v \frac{f
  (\z) -\z^{- \alpha_i} f (s_i
  \z)}{\z^{\alpha_i} - 1} .
\end{equation}

These operators satisfy the braid relations (see for example Proposition~5 of~\cite{BBL}):
\[ \mathfrak{T}_i \mathfrak{T}_{i + 1} \mathfrak{T}_i =\mathfrak{T}_{i + 1}
   \mathfrak{T}_i \mathfrak{T}_{i + 1}, \]
while $\mathfrak{T}_i$ and $\mathfrak{T}_j$ commute if $| i - j | > 1$. They
also satisfy the quadratic relation
\begin{equation}
\label{heckequadratic}
\mathfrak{T}_i^2 = (v - 1) \mathfrak{T}_i + v.
\end{equation}
This quadratic relation implies that $\mathfrak{T}_i$ is invertible. Indeed its inverse is
\begin{equation}
\label{demwhitopshort}
\mathfrak{T}_{i}^{-1} f (\z) = \frac{\z^{-\alpha_i} f (s_i\z) -
  \z^{\alpha_i} f (\z)}{\z^{\alpha_i} - 1} - \frac{f
  (s_i \z) -\z^{\alpha_i} f (
  \z)}{v(\z^{\alpha_i} - 1)}.
\end{equation}
See {\cite{BBL}} Propositions~5 and~6 for proofs of these facts.\footnote{All
references to {\cite{BBL}} are to the published version; the operators in the
arXiv version are slightly different.} The operators $\mathfrak{T}_i$ thus
generate a finite Iwahori Hecke algebra. They are similar to the well-known
\textit{Demazure-Lusztig operators}~\cite{LusztigEquivariant}, which by comparison send $f$ to
\begin{equation}
  \label{demazurelusztig}
  \mathfrak{L}_{i,v} f(\z)=\mathfrak{L}_i f(\z)=\frac{f (\z) - f (s_i \z)}{\z^{\alpha_i} -
   1} - v \frac{f (\z) -\z^{\alpha_i} f (s_i
    \z)}{\z^{\alpha_i} - 1} .
\end{equation}
As we will discuss in Section~\ref{sec:nsmac}, the difference is slight but significant. 
We will refer to the $\mathfrak{T_i}$ operators in (\ref{demwhitop}) as {\textit{Demazure-Whittaker operators}}.

The following result generalizes Theorem~2 of~\cite{BBL}.

\begin{theorem} \label{mainrecursion} For any $w \in W$, a simple reflection $s_i$, and with $v = q^{-1}$,
\begin{equation}\label{eq:mainrecursion} 
\phi_{s_i w}(\z; g) = 
\begin{cases} 
  \mathfrak{T}_{i} \cdot \phi_{w}(\z; g) & \text{if $\ell(s_i w) > \ell(w)$}, \\
  \mathfrak{T}_{i}^{-1} \cdot \phi_{w}(\z; g) & \text{if $\ell(s_i w) < \ell(w),$} 
\end{cases} 
\end{equation}
where the $\mathfrak{T}_i$ and their inverses are as in \eqref{demwhitop} and \eqref{demwhitopshort}.
\end{theorem}
\begin{proof}
The result follows from combining the relation~(\ref{intertwiningactionexplicit}) with Proposition~\ref{csinterident}, recalling from~(\ref{eq:iwahori-whittaker-function}) that the Whittaker functions $\phi_w(\z; g)$ are made with respect to principal series with the Langlands parameter $\z^{-1}$. 
\end{proof}
Proposition~\ref{mainrecursion} gives a recursion on Iwahori fixed vectors $\phi_w$ which is independent of the word used to represent $w$.

\begin{corollary}
 \label{phialgorithm}
 Given any $w_1, w_2$ and a $w_2$-almost dominant weight
 $\lambda$, let $(s_{i_1}, \ldots, s_{i_k})$ be any path in the Weyl group
 from $w_2$ to $w_1$, so $w_2 \rightarrow s_{i_1} w_2 \rightarrow \cdots
 \rightarrow s_{i_k} \cdots s_{i_1} w_2 = w_1$. Set $e_{i_j}$ to be $+1$ or
 $-1$ depending on whether $s_{i_j}$ is an ascent or descent, respectively, in
 Bruhat order. Then, with $v = q^{-1}$,
 \[ \phi_{w_1}(\z; \varpi^{-\lambda} w_2) =
 v^{\ell(w_2)} \mathfrak{T}_{i_k}^{e_{i_k}} \cdots \mathfrak{T}_{i_1}^{e_{i_1}} \z^\lambda. \]
\end{corollary}

\begin{proof}
This corollary follows immediately from Proposition~\ref{basecaseprop} and Theorem~\ref{mainrecursion}. It generalizes Theorem 1 of \cite{BBL}, which evaluates the special case $w_2=1$.
\end{proof}

\begin{corollary}
\label{whitregular}
For any $g\in G(F)$, the function $\phi_w(\mathbf{z};g)$ is
regular as a function of $\mathbf{z}$
on~$\widehat{T}(\mathbb{C})$.
\end{corollary}

\begin{proof}
This is known on other grounds from Proposition~2.1 of~\cite{CasselmanShalika},
but let us show how it follows from our results.
By Remark~\ref{iwahoripg} we may assume $g=\varpi^{-\lambda} w_2$ and if $w=w_2$, regularity follows from Proposition~\ref{basecaseprop}.
Then for more general $w=w_1$ as in Corollary~\ref{phialgorithm},
$\phi_w(\mathbf{z};g)$ may be obtained by applying the $\mathfrak{T}_i$
and $\mathfrak{T}_i^{-1}$, and these preserve regularity.
\end{proof}

\begin{corollary}
\label{whitspecialcase}
Suppose that $w_2\leqslant w_1$ in the Bruhat order.
Then
\begin{equation}
  \phi_{w_1}(\z; \varpi^{-\lambda}w_2) =
v^{\ell(w_2)}\phi_{w_1w_2^{-1}}(\z; \varpi^{-\lambda}).\end{equation}
\end{corollary}

\begin{proof}
Taking a reduced expression $s_{i_1}\cdots s_{i_k}$ for $w_1w_2^{-1}$,
by Corollary~\ref{phialgorithm} both sides equal
\[v^{\ell(w_2)}\mathfrak{T}_{i_1}\cdots\mathfrak{T}_{i_k}\z^{\lambda}. \qedhere \]
\end{proof}

\section{\label{sec:parahoric}Parahoric Whittaker functions}

We now extend the results of the last section to Whittaker functions that are
invariant under so-called `parahoric subgroups' which are intermediate
between the Iwahori subgroup $J$ and the maximal compact subgroup $K$. 

Let $W$ be a Coxeter group with generators $s_i$ ($i \in \mathbf{I}$), which
we will call {\textit{simple reflections}}.
Let $\mathbf{J}$ be a subset of $\mathbf{I}$, let $W_{\mathbf{J}}$
be the subgroup generated by the $s_j$ with $j \in \mathbf{J}$, and let
\begin{equation}\label{WtotheJ} 
W^{\mathbf{J}} = \left\{ w \in W \mid \text{$w s_j > w$ for all $j \in \mathbf{J}$}
   \right\} .
 \end{equation}  
By Proposition~2.4.4 of {\cite{BjornerBrenti}} every element of $W$ has a
unique factorization $w^{\mathbf{J}} w_{\mathbf{J}}$ with
$w^{\mathbf{J}} \in W^{\mathbf{J}}$ and $w_{\mathbf{J}} \in
W_{\mathbf{J}}$. Moreover, by the corollary to that proposition, every coset
$w W_{\mathbf{J}}$ has a unique representative of shortest length, and $w
\in W^{\mathbf{J}}$ if and only if $w$ is that generator.

\begin{lemma}
  \label{lem:bruhatfacts}Let $w \in W^{\mathbf{J}}$. Suppose that $s_i$ is a
  simple reflection such that $w^{- 1} s_i w \notin W_{\mathbf{J}}$. Then
  $s_i w \in W^{\mathbf{J}}$. Assume furthermore that $s_i w > w$. Then $s_i
  w y > w y$ for any $y \in W_{\mathbf{J}}$.
\end{lemma}

\begin{proof}
  Let us show that $s_i w \in W^{\mathbf{J}}$. If not, then for some $j \in
  \mathbf{J}$ we have $s_i w s_j < s_i w$. Since $w \in W^{\mathbf{J}}$ we
  have $w s_j > w$. Now there are two cases. First assume that $s_i w < w$.
  Then $w s_j > w > s_i w
  > s_i w s_j$, and $\ell (w s_j) = \ell (s_i w s_j) + 3$. This is a
  contradiction since $\ell (s_i y) = \ell (y) \pm 1$ for any $y \in W$. On
  the other hand suppose that $s_i w > w$. By the Lifting Property of the Bruhat order
  (Proposition~2.2.7 of~{\cite{BjornerBrenti}}) the inequalities
  $s_i w > w$, $s_i w s_j < s_i w$ and
  $w s_j > w$ imply that $w \leqslant s_i w s_j < s_i w$. 
  Since $\ell (s_i w) = \ell (s_i) + 1$ this implies that $w = s_i w
  s_j$ and therefore $w^{- 1} s_i w = s_j \in W_{\mathbf{J}}$, contradicting
  one of our assumptions.
  Thus, $s_i w \in W_{\mathbf{J}}$.
  
  For the rest of the proof, we assume that $s_i w > w$, and we will prove
  that $s_i w y > w y$ for any $y \in W_{\mathbf{J}}$. Arguing by contradiction,
  assume that $s_i w y < w y$. We claim that either
  $s_i w y = w' y$ for $w' < w$ or $s_i w y = w y'$ for $y' < y$. Indeed, let
  $s_{i_1} \cdots s_{i_k}$ be a reduced word for $w$ and $s_{i_{k + 1}} \cdots
  s_{i_l}$ be a reduced word for $y$. By \cite{BjornerBrenti} Proposition~2.4.4
  $\ell(wy)=\ell(w)+\ell(y)$, so $s_{i_1}\cdots s_{i_l}$ is a reduced word for
  $wy$. By the Exchange Property for Coxeter Groups 
  (Theorem~1.4.3 of {\cite{BjornerBrenti}}) it follows that $s_i w y =
  s_{i_1} \cdots \hat{s}_{i_m} \cdots s_{i_l}$ and if $m \leqslant k$ we may
  take $w' = s_{i_1} \cdots \hat{s}_{i_m} \cdots s_{i_k}$, otherwise, we may
  take $y' = s_{i_{k + 1}} \cdots \hat{s}_{i_m} \cdots s_{i_l}$.
  
  First suppose that $s_i w y = w' y$ with $w' < w$. Then $s_i w = w' < w$,
  contradicting one of our assumptions. On the other hand, suppose that $s_i w
  y = w y'$ with $y' < y$. Then $y' \in W_{\mathbf{J}}$ since $y \in
  W_{\mathbf{J}}$ and $y' < y$.
  Hence $w^{- 1} s_i w = y' y^{- 1} \in
  W_{\mathbf{J}}$, also contradicting one of our assumptions. This proves
  $s_i w y > w y$.
\end{proof}

\begin{definition}\label{definitionparahoric}
By a {\textit{standard parahoric subgroup}} we mean a subgroup
of $K=G(\mathfrak{o})$ that arises as the preimage of a standard parabolic under
the canonical map $K \longrightarrow G(\mathbb{F}_q)$. In particular, any such group 
contains the Iwahori subgroup $J$. We will denote these groups by $K_{\mathbf{J}}$ where
$\mathbf{J}$ is the index set of simple roots in the corresponding standard parabolic subgroup $P$ of $G$.
We will denote the Levi subgroup of $P$ by $M$ which has Weyl group $W_\mathbf{J}$.
With this notation $K_{\varnothing}=J$ (the Iwahori subgroup)
and $K_{\mathbf{I}}=K$ (the maximal compact subgroup).
\end{definition}

For example, if $G = \GL_r$, these may be constructed as follows. Let $\mathbf{r}= (r_1, \cdots, r_k)$ be a set of
positive integers such that $\sum_i r_i = r$. Let $P := P_{\mathbf{r}}$ be the parabolic
subgroup of $\GL_r$ containing the opposite Borel $B_-$ with Levi subgroup $M =
\GL_{r_1} \times \cdots \times \GL_{r_k}$ embedded diagonally in
$\GL_r$. The Weyl group $W_{\mathbf{J}} = S_{r_1} \times \cdots \times S_{r_k}$ of
$M$ is a parabolic subgroup of $W = S_r$ with $\mathbf{J}$ consisting of
integers $1 \leqslant j \leqslant r - 1$ such that $j$ is not of the form
$r_1$, $r_1 + r_2$, $r_1 + r_2 + r_3$, etc. Then the standard parahoric
subgroup $K_{\mathbf{J}}$ is the preimage of $P_{\mathbf{r}}(\mathbb{F}_q)$ under the map
$K \longrightarrow \GL_r (\mathbb{F}_q)$.

If $w \in W^{\mathbf{J}}$ let
\begin{equation}
 \label{eq:psidef}
 \psi^{\mathbf{J}}_w := \sum_{y \in W_{\mathbf{J}}} \phi_{wy},
\end{equation}
where the $\phi_w$ denote the standard basis Iwahori Whittaker functions
defined in~(\ref{eq:iwahori-whittaker-function}).
\begin{proposition}\label{proposition:basisparahoricfixedW}
  The $\psi^{\mathbf{J}}_w$ with $w \in W^{\mathbf{J}}$ are a basis of the $K_{\mathbf{J}}$-fixed
  vectors in the Whittaker model of a principal series $I(\mathbf{z})$.
\end{proposition}

\begin{proof}
  Let $k =\mathbb{F}_q$ be the residue field. We have
  \[ G (k) = \bigsqcup_{w \in W^{\mathbf{J}}} B (k) w P (k) \qquad (\text{disjoint})
  \]
  by {\cite{BourbakiLie456}}, Remark 2 in Section IV.2.5 and the fact that
  $W^{\mathbf{J}}$ are a set of coset representatives for $W /
  W_{\mathbf{J}}$. (The parabolic $P$ actually contains the opposite Borel
  subgroup to $B$ but the decomposition is still valid.) Pulling this back to
  $K = G (\mathfrak{o})$ we have
\begin{equation}
  \label{kjjwj} K = \bigsqcup_{w \in W^{\mathbf{J}}} J_+ wK_{\mathbf{J}}
\end{equation}
where $J_+$ is the upper Iwahori subgroup. We have bijections of coset spaces
\[ B (F) \setminus G (F) \; \longleftrightarrow \; (B (F) \cap K) \setminus K
   \; \longleftrightarrow \; \bigsqcup_{w \in W^{\mathbf{J}}} N_- (\mathfrak{p})
   wK_{\mathbf{J}} \]
where the first step follows from the Iwasawa decomposition $G (F) = B (F) K$
and for the second step we have used (\ref{kjjwj}) and the Iwahori
factorization $J_+ = (B (F) \cap K) N_- (\mathfrak{p})$. Therefore
\[ G (F) = \bigsqcup_{w \in W^{\mathbf{J}}} B (F) N_- (\mathfrak{p}) wK_{\mathbf{J}} =
   \bigsqcup_{w \in W^{\mathbf{J}}} B (F) wK_{\mathbf{J}},
\]
where for the last step we have used the fact that $w^{- 1} N_- (\mathfrak{p})
w \subset K_{\mathbf{J}}$.
  Hence $I (\mathbf{z})^{K_{\mathbf{J}}}$ has a basis of functions
  $\Psi^{\mathbf{J}}_w$ ($w \in W^{\mathbf{J}}$) defined by
  \[ \Psi^{\mathbf{J}}_w (b w' k) = \left\{\begin{array}{ll}
       \delta^{1 / 2} \tau_{\mathbf{z}} (b) & \text{if $w' \in w
       W_{\mathbf{J}}$},\\
       0 & \text{otherwise,}
     \end{array}\right. \]
  for $w' \in W$ and $k \in K_{\mathbf{J}}$. Decomposing the support of $\Psi^{\mathbf{J}}_w$ into a
  union of $J$-cosets, we see that $\Psi^{\mathbf{J}}_w = \sum_{y \in W_{\mathbf{J}}}
  \Phi_{w y}$, and projecting this identity into the Whittaker model, the
  statement follows.
\end{proof}

\begin{remark}
  One can deduce from Proposition~\ref{proposition:basisparahoricfixedW} the dimension of the space of
parahoric fixed vectors in the unramified principal series. This is in
accordance with the work of Lansky~\cite{LanskyParahoric} (see Theorem 1.1), where these
dimensions were first computed.
\end{remark}

\begin{proposition}
  \label{parahoric_firstcase}
  Suppose that $w \in W^{\mathbf{J}}$ and that $s_i$ is a simple reflection
  such that $w^{- 1} s_i w \in W_{\mathbf{J}}$. Then $\psi^{\mathbf{J}}_w (\mathbf{z};
  g)$ is divisible by $1 - v\mathbf{z}^{- \alpha_i}$ as a Laurent
  polynomial, and $(1 - v\mathbf{z}^{- \alpha_i})^{- 1} \psi^{\mathbf{J}}_w (\mathbf{z}; g)$ is
  invariant under the reflection~$s_i$. Moreover $s_iw>w$.
\end{proposition}

\begin{proof}
  Let $t = w^{- 1} s_i w$ be the reflection that is in $W_{\mathbf{J}}$.
  Assume $y \in W_{\mathbf{J}}$ and $t y > y$. Then $s_i w y = w t y$ and by
  Theorem~\ref{mainrecursion} we have $\phi_{w t y} =\mathfrak{T}_i
  \phi_{w y}$. Thus
  \[ \psi^{\mathbf{J}}_w = \sum_{y \in W_\mathbf{J}} \phi_{w y} = \sum_{\substack{
       y \in W_\mathbf{J}\\
       t y > y}}(\phi_{w y} + \phi_{w t y}) = (1 +\mathfrak{T}_i)
     \sum_{\substack{y \in W_\mathbf{J}\\ t y > y}}
     \phi_{w y} . \]
  Now using the fact that $(\mathfrak{T}_i - v) (\mathfrak{T}_i + 1) = 0$, we
  have $\mathfrak{T}_i \psi^{\mathbf{J}}_w = v \psi^{\mathbf{J}}_w$. Substituting the definition of
  $\mathfrak{T}_i$ a small amount of algebra gives
  \begin{equation}
    \label{ydsinvariant} (1 - v\mathbf{z}^{\alpha_i}) \psi^{\mathbf{J}}_w (\mathbf{z};
    g) = (1 - v\mathbf{z}^{- \alpha_i}) \psi^{\mathbf{J}}_w (s_i \mathbf{z}; g) .
\end{equation}

  The function $\psi^{\mathbf{J}}_w (\mathbf{z}; g)$ is a regular function on
  $\hat{T} (\mathbb{C})$ by Corollary~\ref{whitregular}.
  The ring $\mathcal{O} (\hat{T})$ of regular functions is a
  principal ideal domain; indeed it is a Laurent polynomial ring. The
  functions $1 - v\mathbf{z}^{\alpha_i}$ and $1 - v\mathbf{z}^{-
  \alpha_i}$ are coprime. From the right-hand side of (\ref{ydsinvariant}), $1
  - v\mathbf{z}^{- \alpha_i}$ divides the left-hand side, and it therefore
  divides $\psi^{\mathbf{J}}_w (\mathbf{z}; g)$. Remembering that $w (\alpha_i) = -
  \alpha_i$, we may rearrange (\ref{ydsinvariant}) in the form
  \[ \frac{\psi^{\mathbf{J}}_w (\mathbf{z}; g)}{1 - v\mathbf{z}^{- \alpha_i}} = s_i
     \left( \frac{\psi^{\mathbf{J}}_w (\mathbf{z}; g)}{1 -v \mathbf{z}^{- \alpha_i}}
     \right), \]
  and we have proved that this is an $s_i$-invariant regular function.

  The last assertion to be proved is that $s_iw>w$. For this we note
  that with $t=w^{-1}s_iw\in W_{\mathbf{J}}$, $wt$ is the unique
  factorization of $s_iw$ into a product of elements of $W^{\mathbf{J}}$
  and $W_{\mathbf{J}}$ by Proposition~2.4.4 of~\cite{BjornerBrenti} which was mentioned before, and
  by part~(iii) of that result, $\ell(s_iw)=\ell(w)+\ell(t)>\ell(w)$,
  as required.
\end{proof}

\begin{proposition}
  \label{parahoric_secondcase}
  Suppose that $w \in W^{\mathbf{J}}$ and that $s_i$ is a simple reflection
  such that $w^{- 1} s_i w \notin W_{\mathbf{J}}$. Then
  \[ \psi^{\mathbf{J}}_{s_i w} (\mathbf{z}; g) = \left\{\begin{array}{ll}
       \mathfrak{T}_i \psi^{\mathbf{J}}_w (\mathbf{z}; g) & \text{if $s_i w > w,$}\\
       \mathfrak{T}_i^{- 1} \psi^{\mathbf{J}}_w (\mathbf{z}; g) & \text{if $s_i w < w.$}
     \end{array}\right. \]
\end{proposition}

\begin{proof}
  By Lemma~\ref{lem:bruhatfacts}, $w \in W^{\mathbf{J}}$ if and only if $s_i
  w \in W^{\mathbf{J}}$. With this in mind, the two cases are equivalent
  since we may interchange the roles of $w$ and $s_i w$. Therefore we may
  assume that $s_i w > w$. Let $y \in W_{\mathbf{J}}$. By
  Lemma~\ref{lem:bruhatfacts} we have $s_i w y > w y$ and so by
  Theorem~\ref{mainrecursion} we have $\phi_{s_i w y} =\mathfrak{T}_i
  \phi_{w y}$. Therefore
  \[ \psi^{\mathbf{J}}_{s_i w} = \sum_{y \in W_{\mathbf{J}}} \phi_{s_i w y}
     =\mathfrak{T}_i \sum_{y \in W_{\mathbf{J}}} \phi_{w y} =\mathfrak{T}_i
     \psi^{\mathbf{J}}_w . \qedhere \]
\end{proof}

In the following theorem, we show that Proposition~\ref{parahoric_firstcase} implies a
Casselman-Shalika formula for $\psi^{\mathbf{J}}_1 = \sum_{w \in W_{\mathbf{J}}} \phi_w$ for any
fixed subset $\mathbf{J}$.
If $\lambda$ is a dominant weight for $\hat{G}(\CC)$, the complex, connected Langlands dual
group of $G$, then it is also a
dominant weight for the Levi subgroup
$\hat{M}(\CC)$, the dual group of $M$ having Weyl group $W_{\mathbf{J}}$. So we may consider the irreducible character
$\chi_{\lambda}^{\mathbf{J}}$ of $\hat{M}(\CC)$ with highest weight $\lambda$.
We will denote by $\chi_{\lambda}^{\mathbf{J}}(\mathbf{z})$ the
value of this character on the semisimple element with eigenvalues
$(z_1,\ldots,z_r)$. The root system
$\Delta_\mathbf{J}$ of $\hat{M}(\CC)$ is a subsystem of the root system $\Delta$ of $\hat{G}(\mathbb{C})$.

\begin{theorem}
  \label{thm:paracs}
  Suppose that $\lambda$ is a dominant weight for $G$. Then
  \begin{equation}
  \label{cshop}
  \psi^{\mathbf{J}}_1 (\z; \varpi^{- \lambda}) = \prod_{\alpha \in \Delta^+_{\mathbf{J}}}
     (1 - v\mathbf{z}^{- \alpha}) \chi_{\lambda}^{\mathbf{J}} (\mathbf{z}) .
  \end{equation}
  If $w\in W^\mathbf{J}$ then
  \begin{equation}
    \label{mmhop}
    \psi^{\mathbf{J}}_w(\mathbf{z};g)=\mathfrak{T}_w\psi^{\mathbf{J}}_1(\mathbf{z};g)
  \end{equation}
\end{theorem}

\begin{proof}
  We will argue that the expression
  \begin{equation}
    \label{psidivided} \frac{\psi^{\mathbf{J}}_1 (\z; \varpi^{- \lambda})}{\prod_{\alpha \in
    \Delta^+_{\mathbf{J}}} (1 - v\mathbf{z}^{- \alpha})} \,
  \end{equation}
   is regular for $\mathbf{z} \in \hat{T} (\mathbb{C})$,
  symmetric under the action of $W_{\mathbf{J}}$, and independent of $v$.
  
  If $s_i$ is a simple reflection in $W_{\mathbf{J}}$ then we may write
  (\ref{psidivided}) as $(1 - v\mathbf{z}^{- \alpha_i})^{- 1}
  \psi^{\mathbf{J}}_1 (\z; \varpi^{- \lambda})$ divided by the remaining factors, which are
  permuted by $s_i$. By Proposition~\ref{parahoric_firstcase} this shows that
  (\ref{psidivided}) is invariant under $s_i$ and so it is invariant under
  $W_{\mathbf{J}}$.
  
  Next let us show that (\ref{psidivided}) is regular on all of $\hat{T}
  (\mathbb{C})$. To see this, note that its potential poles are in the union
  of the hypersurfaces $\mathbf{z}^{\alpha} = v$ ($\alpha \in
  \Delta_{\mathbf{J}}^+$). But by Proposition~\ref{parahoric_firstcase} if
  $\alpha_i$ is a simple root the hyperplane $\mathbf{z}^{\alpha_i} = v$ is
  not among the poles. As the polar divisor is invariant under
  $W_{\mathbf{J}}$, neither are any of the other loci $\mathbf{z}^{\alpha}
  = v$. Hence (\ref{psidivided}) is regular.
  
  Now the numerator and the denominator of (\ref{psidivided}) are polynomials
  in $z_i$, $z_i^{- 1}$ and $v$, and the numerator is divisible by the
  denominator. We will argue that both have degree $| \Delta_{\mathbf{J}}^+
  | = \ell (w_0^\mathbf{J})$ as polynomials in $v$ where $w_0^\mathbf{J}$ is the long element of $W_\mathbf{J}$ which shows that the ratio is in fact independent of $v$. For the numerator, we may write
  $\psi^{\mathbf{J}}_1 = \sum_{w \in W_{\mathbf{J}}} \phi_w$ and $\phi_w (\varpi^{-
  \lambda}) =\mathfrak{T}_w \mathbf{z}^{\lambda}$; it follows from the
  definition of $\mathfrak{T}_w$ that its degree in $v$ is $\ell (w)$, which
  takes its maximum at the long element $w_0^\mathbf{J}$, proving that the degree is
  $\ell (w_0^\mathbf{J})$. And clearly the denominator has the same degree.
  
  Then (\ref{psidivided}) is independent of $v$ and we may therefore take the limiting case $v
  \rightarrow 0$ to evaluate it. Let
  \[ \partial_i^{\circ} f (\mathbf{z}) = \frac{f (\mathbf{z}) - f (s_i
     \mathbf{z})}{\mathbf{z}^{\alpha_i} - 1}, \qquad \partial_i f
     (\mathbf{z}) = \frac{f (\mathbf{z}) -\mathbf{z}^{- \alpha_i} f (s_i
     \mathbf{z})}{1 -\mathbf{z}^{- \alpha_i}} . \]
  We have $\partial_i^{\circ} = \partial_i - 1$. These Demazure operators are
  known to satisfy the braid relations; see for example \cite{BumpLie}
  Propositions~25.1 and~25.3 but note that in Proposition~25.3 there is a typo
  and $D_i$ should be $\partial_i$. If $w = s_{i_1} \cdots s_{i_k}$
  is a reduced expression define $\partial_w = \partial_{i_1} \cdots
  \partial_{i_k}$ and similarly for $\partial^\circ_w$.
  
  When $v \rightarrow 0$ the operator $\mathfrak{T}_i$ reduces to $\partial_i^\circ$
  and so
  \[ \lim_{v \rightarrow 0} \; \phi_w (\varpi^{- \lambda}) =
     \partial_w^{\circ} \mathbf{z}^{\lambda} . \]
  By Theorem~2.1 of {\cite{demice}}, it follows that (\ref{psidivided}) equals
  $\sum_{w \in W_{\mathbf{J}}} \partial_w^{\circ} \mathbf{z}^{\lambda} =
  \partial_{w_0^{\mathbf{J}}} \mathbf{z}^{\lambda}$ and by the Demazure
  character formula (\cite{BumpLie} Theorem~25.3), this is
  $\chi_{\lambda}^{\mathbf{J}} (\mathbf{z})$. This proves (\ref{cshop}).

  Now let us prove (\ref{mmhop}) by induction on $\ell(w)$.
  If $w=1$, there is nothing to prove. Thus suppose that
  $w=s_iw'$ where $s_i$ is a simple reflection and $\ell(w')<\ell(w)$. Then
  $w^{-1}s_iw\notin W_{\mathbf{J}}$, since otherwise the last assertion of
  Proposition~\ref{parahoric_firstcase} would be contradicted.
  Therefore by Lemma~\ref{lem:bruhatfacts} $w'\in W^{\mathbf{J}}$
  and by induction $\psi^{\mathbf{J}}_{w'}=\mathfrak{T}_{w'}\psi^{\mathbf{J}}_1$. Now
  Proposition~\ref{parahoric_secondcase} gives
  \[\psi^{\mathbf{J}}_w=\mathfrak{T}_i\mathfrak{T}_{w'}\psi^{\mathbf{J}}_1=\mathfrak{T}_w\psi^{\mathbf{J}}_1,\]
  proving (\ref{mmhop}).
\end{proof}

\begin{remark}
\label{rem:parcaveat}
For Iwahori Whittaker functions Corollary~\ref{phialgorithm}
gives an algorithm to compute $\phi_{w_1}$ at any
value of $g$ (which we can take to be on the form $\varpi^{-\lambda} w_2$). This depends on Proposition~\ref{basecaseprop},
which gives the base case $\phi_{w}(\varpi^{-\lambda} w)$, i.e.\ $w_1 = w_2$, for a recursive algorithm using Demazure-Whittaker operators,
suitable for implementation on a computer. Our parahoric
results are not as general because we only have a direct expression (without Demazure-Whittaker operators) for the base case (\ref{cshop}) which is only for $w_1=w_2=1$.
For an arbitrary base case $w_1 = w_2 \neq 1$ we need to fall back on (\ref{eq:psidef}) which expresses $\psi^\mathbf{J}_{w_1}(\varpi^{-\lambda} w_2)$ as a sum of $\phi_{w_1y}(\varpi^{-\lambda} w_2)$.
However since the $\phi_w$ are computable, this is not an obstacle
to an explicit computation of $\psi_w^{\mathbf{J}}$.
In both cases we may then compute all $\psi_{w_1}^{\mathbf{J}}(\varpi^{-\lambda} w_2)$
from $\psi_{w_2}^{\mathbf{J}}(\varpi^{-\lambda} w_2)$ recursively, taking
Proposition~\ref{parahoric_firstcase} and~\ref{parahoric_secondcase} into
account.
Alternatively, knowing all $\phi_w$ we may compute $\psi_{w_1}^{\mathbf{J}}(\varpi^{-\lambda} w_2)$ directly by the sum~\eqref{eq:psidef}.
\end{remark}

\section{Yang-Baxter equations from fusion}
\label{sec:lattice-models}
The models that we will be concerned with take place on planar
graphs. In using the term \textit{graph} to describe these
arrays we are deviating from usual terminology, where edges
have always two vertices, for we will allow open edges with only
a single endpoint. Thus we mean a set of \textit{vertices} which are
points in the plane, together with \textit{edges} that are
arcs which either join two vertices, or which are attached
to only a single vertex. The edges which are only attached to
a single vertex are called \textit{boundary edges}.
The edges attached to two vertices are called \textit{interior edges}.
Every vertex is adjacent to four edges. Edges can only
cross at a vertex.

For every edge $A$ in the graph there is a finite set $\Sigma_A$ of
values called \textit{spins} that may be assigned to the
edge in a state of the system.

\begin{assumption}
\label{diagsamev}
At each vertex, let $A$, $B$, $C$, $D$ be the four adjacent edges,
arranged so that $A$ and $C$ are opposite edges, as are $B$ and $D$.
Then $\Sigma_A=\Sigma_C$ and $\Sigma_B=\Sigma_D$.
\end{assumption}

Each vertex has a \textit{label} $\xi$, and an associated set
of \textit{Boltzmann weights} $\beta_\xi$. This is a rule which assigns a
complex number to every possible choice of spins at the four adjacent edges of
the vertex. Thus if $A,B,C,D$ are the four adjacent edges to a vertex with
label $\xi$, this data consists of a map
\[\beta_\xi:\Sigma_A\times\Sigma_B\times\Sigma_C\times\Sigma_D\to\CC.\]
We call the set of spins
$(a,b,c,d)\in\Sigma_A\times\Sigma_B\times\Sigma_C\times\Sigma_D$
a \textit{configuration} at the vertex. The configuration is
\textit{admissible} if $\beta_\xi(a,b,c,d)\neq 0$.

\begin{assumption}
  \label{configdet}
  At each vertex, if three out
  of four spins in $(a,b,c,d)$ in an admissible configuration
  are given, the fourth is uniquely determined.
\end{assumption}

In a system $\mathfrak{S}$, the data specifying the system
are the graph itself, the spinsets $\Sigma_A$, the Boltzmann
weight data $\beta_\xi$ for each label $\xi$, and for
every boundary edge $A$ a fixed boundary spin $\mathfrak{b}_A\in\Sigma_A$.
For example, the labels $\xi$ might 
be complex numbers and $\beta_\xi$ are uniformly described as a set of 
complex-valued functions of $\xi$ for each configuration.
Note that the spins of the boundary
edges are fixed, and are part of the data specifying
the system.

A \textit{state} $\mathfrak{s}$ of the system is an assignment of spins
to all edges. That is, for each edge $A$ there
is specified a spin $\mathfrak{s}_A\in\Sigma_A$.
For boundary
edges $\mathfrak{s}_A$ must be the fixed boundary spin $\mathfrak{b}_A$,
while the interior edge spins are allowed to vary.
We will use the notation $\mathfrak{s}\in\mathfrak{S}$
to mean that $\mathfrak{s}$ is a state of the system $\mathfrak{S}$.
The \textit{Boltzmann weight} $\beta(\mathfrak{s})$ of the state
is the product of the Boltzmann weights at the labelled vertices and the state is said to be \textit{admissible} if all of its vertices are admissible.
The \textit{partition function} $Z(\mathfrak{S})$ is the
sum of the Boltzmann weights of all the (admissible) states.

In the systems that we will consider, the edges may all be classified
as either \textit{horizontal} or \textit{vertical}. There will be
two types of vertices. In one type, the vertex intersects four
horizontal edges and will be called an \textit{R-vertex}. In the other, called \textit{ordinary}, it intersects two horizontal and
two vertical ones. See Figure~\ref{vertextypes}.

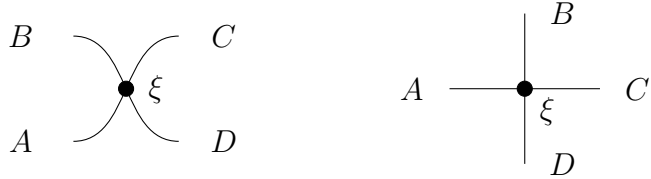
\begin{figure}
  \begin{center}
\begin{tikzpicture}
\draw (0,0) to[out=0, in=180] (1.4,1.4);
\draw (0,1.4) to[out=0, in=180] (1.4,0);
\draw[fill=black] (.7,.7) circle (.1);
\node at (1.1,.7) {$\xi$};
\node at (-.7,0) {$A$};
\node at (-.7,1.4) {$B$};
\node at (2,1.4) {$C$};
\node at (2,0) {$D$};
\begin{scope}[shift={(5,-0.3)}]
\draw (0,1)--(2,1);
\draw (1,0)--(1,2);
\draw[fill=black] (1,1) circle (.1);
\node at (1.3,.7) {$\xi$};
\node at (-.5,1) {$A$};
\node at (1.5,2) {$B$};
\node at (2.5,1) {$C$};
\node at (1.5,0) {$D$};
\end{scope}
\end{tikzpicture}
\end{center}
\caption{Left: a vertex adjoining four horizontal edges. Such a vertex will
  be called an \textit{R-vertex}. Right: A vertex adjoining two horizontal
  edges and two vertical ones. We will call such vertices \textit{ordinary}.
  Each vertex receives a label $\xi$ corresponding to its Boltzmann weights.}
\label{vertextypes}
\end{figure}

Next we explain a procedure we refer to as \textit{fusion} for producing new kinds of
edges and vertices from given ones. (This is partly inspired by a process of the  
same name described in Borodin and Wheeler~\cite{BorodinWheelerColored}, Appendix~B.)

Given a sequence of edges, $A_1,\cdots,A_m$
we may replace these with a single edge
$\mathbf{A}$ such that $\Sigma_{\mathbf{A}}=\prod_k \Sigma_{A_k}$.
This edge is called the \textit{fusion} of the
edges $\{ A_k \}$. Next assume that we have a sequence
of $m$ \textit{ordinary} vertices with labels $\xi_1,\cdots,\xi_m$ such that the vertex with label $\xi_k$ is 
adjoined to the vertex with label $\xi_{k+1}$
by an edge $E_k$ if $1\leqslant k\leqslant m-1$.
Let the remaining adjacent edges of the vertex with label $\xi_k$ be $B_k$ and $D_k$ and
$A$ (if $k=1$) and $C$ if $k=m$. Thus the configuration
is as in Figure~\ref{fusionfig} (left).

\begin{figure}[ht]
\[\begin{array}{|c|c|}\hline\begin{tikzpicture}
\draw (0,1) to (4,1);
\draw (1,0) to (1,2);
\draw (3,0) to (3,2);
\node at (5,2) {$\cdots$};
\node at (5,1) {$\cdots$};
\node at (5,0) {$\cdots$};
\draw (6,1) to (8,1);
\draw (7,0) to (7,2);
\draw[fill=black] (1,1) circle (.1);
\draw[fill=black] (3,1) circle (.1);
\draw[fill=black] (7,1) circle (.1);
\node at (-.5,1) {$A$};
\node at (1.3,0.7) {$\xi_1$};
\node at (3.3,0.7) {$\xi_2$};
\node at (7.3,0.7) {$\xi_m$};
\node at (1,2.25) {$B_1$};
\node at (3,2.25) {$B_2$};
\node at (7,2.25) {$B_m$};
\node at (1,-.25) {$D_1$};
\node at (3,-.25) {$D_2$};
\node at (7,-.25) {$D_m$};
\node at (8.25,1) {$C$};
\node at (2,1.25) {$E_1$};
\node at (3.5,1.25) {$E_2$};
\node at (6.25,1.25) {$E_{m-1}$};
\end{tikzpicture}&
\begin{tikzpicture}
\draw (0,1) to (2,1);
\draw (1,0) to (1,2);
\draw[fill=black] (1,1) circle (.1);
\node at (-.5,1) {$A$};
\node at (1.2,0.7) {$\xi$};
\node at (2.5,1) {$C$};
\node at (1,2.25) {$\mathbf{B}$};
\node at (1,-.25) {$\mathbf{D}$};
\end{tikzpicture}\\\hline\end{array}\]
\caption{Fusion. This procedure replaces a sequence of vertices by a single vertex.}
\label{fusionfig}
\end{figure}
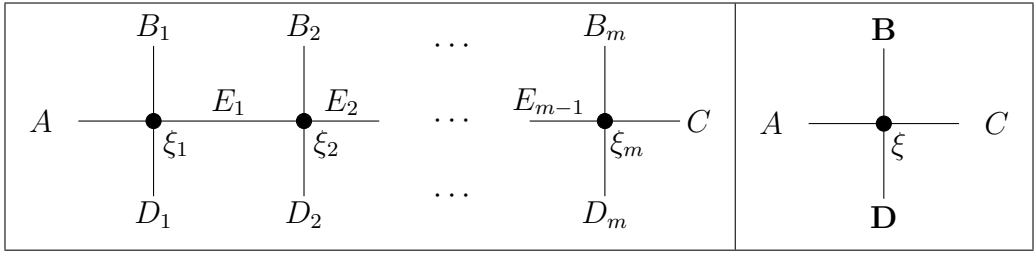

Now we may construct the fusion $\mathbf{B}$ of the edges $B_i$ as
above, as well as the fusion $\mathbf{D}$ of the edges $D_i$.
We may then fuse the vertices, replacing the sequence of $m$ vertices with labels
$\xi_1,\ldots,\xi_m$ by a single vertex labeled $\xi$ (as in Figure~\ref{fusionfig}).
It remains to discuss the Boltzmann weights.
Let spins $(a,b,c,d)\in\Sigma_A\times\Sigma_{\mathbf{B}}\times\Sigma_C\times\Sigma_{\mathbf{D}}$.
By definition $b$ and $d$ are sequences of spins
$b_k\in\Sigma_{B_k}$ and $d_k\in\Sigma_{D_k}$.
Fixing $(a,b,c,d)$, it follows from Assumption~\ref{configdet} that
the system in Figure \ref{fusionfig} (left) has at most one (admissible) state. We define
$\beta_\xi(a,b,c,d)$ to be its partition function.
It is clear that Assumption~\ref{configdet} remains
valid for this fused vertex.

At any vertex, it will be useful to choose a clockwise ordering
$(A,B,C,D)$ of the adjoining edges. In our illustrations, we will
always choose the ordering as in Figure~\ref{vertextypes}.
If $A$ is an edge, we will denote by $V_A$ the free vector space with basis $\Sigma_A$.
By Assumption~\ref{diagsamev}, we may identify $V_A=V_C$ and
$V_B=V_D$. Then the Boltzmann weights at a vertex with label $\xi$ define an element of $\End(V_A\otimes V_B)$ by
\begin{equation}
  \label{rmatrixbw}
  a\otimes b\mapsto\sum_{(c,d)\in\Sigma_C\times\Sigma_D}\beta_\xi(a,b,c,d)(c\otimes d).
\end{equation}
If the vertex is an R-vertex we will denote this endomorphism as $R_\xi$; 
this endomorphism is called an \textit{R-matrix}. 
For ordinary vertices, we will denote the endomorphism (\ref{rmatrixbw}), which is called a \textit{transfer matrix}, as $T_\xi$.

\begin{definition}
Suppose that, for ordinary vertices labeled $\xi$, $\eta$ and R-vertex labeled $\zeta$, there exists Boltzmann weights such that for every choice of boundary spins $(a,b,c,d,e,f)$ the partition
functions of the two systems in Figure~\ref{ybesfig} are equal. Then we say
we have a solution of the \textit{Yang-Baxter equation}.
\end{definition}

\begin{figure}[H]
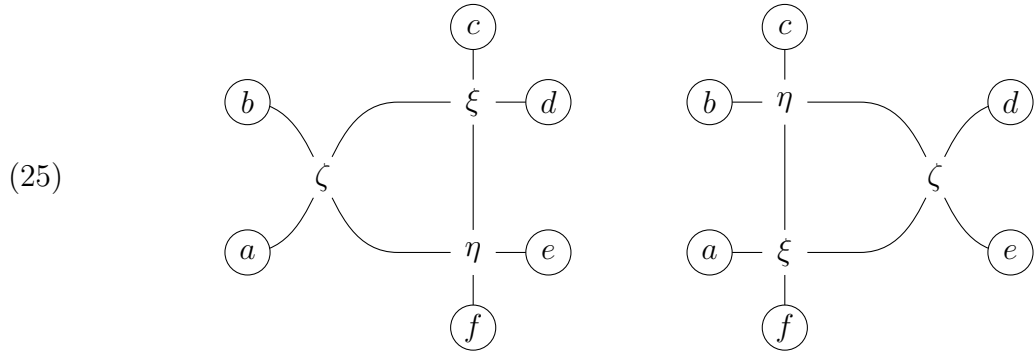

\begin{equation}
  \label{ybe}
  \begin{array}{ccc}\ybelhs{a}{b}{c}{d}{e}{f}{\zeta}{\xi}{\eta}&\qquad\qquad&
  \yberhs{a}{b}{c}{d}{e}{f}{\zeta}{\xi}{\eta}\end{array}
\end{equation}
\caption{The Yang-Baxter equation.}
\label{ybesfig}
\end{figure}

Let $A$, $B$, $C$, $D$, $E$, $F$ be the 
boundary edges of these configurations, so that $a\in\Sigma_A$,
etc. By Assumption~\ref{diagsamev} $V_A=V_D$, $V_B=V_E$
and $V_C=V_F$. Then $R_\zeta\in\End(V_A\otimes V_B)$, $T_\xi\in\End(V_A\otimes V_C)$
and $T_\eta\in\End(V_B\otimes V_C)$. The Yang-Baxter equation can be expressed in
the formula
\begin{equation}
  \label{rttisttr}
  (R_\zeta)_{12}(T_\xi)_{13}(T_\eta)_{23}=(T_\eta)_{23}(T_\xi)_{13}(R_\zeta)_{12},
\end{equation}
an identity in $\End(V_A\otimes V_B\otimes V_C)$,
where, in the notation common in quantum group theory, $(R_\zeta)_{12}$
denotes $R_\zeta$ acting on the first two components of
$V_A\otimes V_B\otimes V_C$ and so forth. We wish to consider
examples of (\ref{ybe}) where the ordinary vertices arise from the fusion process described above. Thus the left
configuration can be expanded as in Figure~\ref{ybefusionsetup}.

\begin{figure}[H]
\[\begin{tikzpicture}
  \draw (0,1) to [out = 0, in = 180] (2,3) to (4,3);
  \draw (0,3) to [out = 0, in = 180] (2,1) to (4,1);
  \draw (6,1) to (8,1);
  \draw (6,3) to (8,3);
  \draw (3,0) to (3,4);
  \draw (7,0) to (7,4);
  \draw[fill=white] (0,1) circle (.3);
  \draw[fill=white] (0,3) circle (.3);
  \draw[fill=white] (3,4) circle (.3);
  \draw[fill=white] (7,4) circle (.3);
  \draw[fill=white] (7,0) circle (.3);
  \draw[fill=white] (8,3) circle (.3);
  \draw[fill=white] (8,1) circle (.3);
  \draw[fill=white] (3,0) circle (.3);
  \node at (0,1) {$a$};
  \node at (0,3) {$b$};
  \node at (3,4) {$c_1$};
  \node at (5,4) {$\cdots$};
  \node at (5,3) {$\cdots$};
  \node at (5,1) {$\cdots$};
  \node at (5,0) {$\cdots$};
  \node at (7,4) {$c_m$};
  \node at (7,0) {$f_m$};
  \node at (8,3) {$d$};
  \node at (8,1) {$e$};
  \node at (3,0) {$f_1$};
  \node at (7,0) {$f_m$};
\path[fill=white] (1,2) circle (.3);
\node at (1,2) {$\zeta$};
\path[fill=white] (3,3) circle (.3);
\node at (3,3) {$\xi_1$};
\path[fill=white] (3,1) circle (.3);
\node at (3,1) {$\eta_1$};
\path[fill=white] (7,3) circle (.3);
\node at (7,3) {$\xi_m$};
\path[fill=white] (7,1) circle (.3);
\node at (7,1) {$\eta_m$};
\end{tikzpicture}\]
\caption{Setup for the Yang-Baxter equation with fused vertices $\xi$ and $\eta$.}
\label{ybefusionsetup}
\end{figure}
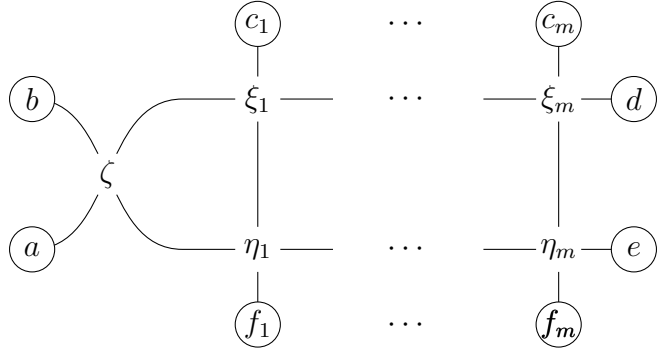

\begin{lemma}
  \label{lemfusedybe}
  Suppose there exists a sequence of R-vertices with labels $\zeta_1,\ldots, \zeta_{m+1}$
such that $\zeta_1=\zeta_{m+1}=\zeta$ and such that for each $1\leqslant k\leqslant m$,
the two partition functions in Figure~\ref{auxiliaryybe} are equal. (Note that the R-vertex of the left-hand side is $\zeta_{k}$ while the one on the right-hand side is $\zeta_{k+1}$.) Then the auxiliary Yang-Baxter equations in \eqref{ybepre} induce a solution to the Yang-Baxter equation in~\eqref{ybe} for the fused system.
\end{lemma}

\begin{figure}[ht]
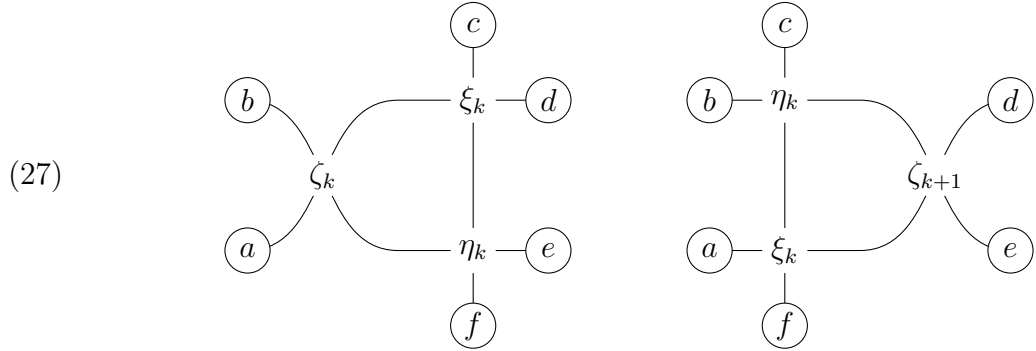

\begin{equation}
  \label{ybepre}
  \begin{array}{ccc}
  \ybelhs{a}{b}{c}{d}{e}{f}{\zeta_{k}}{\xi_{k}}{\eta_{k}}&\qquad\qquad&
  \yberhs{a}{b}{c}{d}{e}{f}{\zeta_{k+1}}{\xi_{k}}{\eta_{k}}\end{array}
\end{equation}
\caption{Auxiliary Yang-Baxter equations. These imply a Yang-Baxter equation
  for the fusion situation in Figure~\ref{ybefusionsetup}. In these equations, the
  R-matrix changes after moving past the vertical edges. After $m$
  such changes, it is back to its original form.}
\label{auxiliaryybe}
\end{figure}

\begin{proof}
  This follows from the usual train argument. Each time
  the R-matrix moves to the right, $\zeta_{k}$ is replaced by
  $\zeta_{k+1}$. Since $\zeta_1=\zeta_{m+1}=\zeta$, the statement follows.
\end{proof}

\begin{remark} We have chosen to call the method for producing new solutions to the Yang-Baxter equation outlined in this section `fusion,' despite some differences with the prior notion in the literature (see for example \cite{KulishReshetikhinSklyanin} and Appendix B of \cite{BorodinWheelerColored}). Both methods construct new solutions from old by forming new weights using one-row partition functions. The typical fusion construction features two steps: first summing over all one-row systems with given multiset of spins on its vertical edges (which is the graphical manifestation of the R-matrix of a tensor product of quantum group modules) and then taking a further weighted average (which manifests the resulting R-matrix for projection onto irreducible constituents of the tensor product; see for example (B.2.1) of \cite{BorodinWheelerColored}). However, our fusion prescribes a set of labels for each vertex in the one-row system, our weights are allowed to vary based on the label, and we do not require a second summation acting as a projection. Our example of weights for fusion in the next section (see Figure~\ref{cvmonochrome}) will have vertices labeled by colors and the weights depend critically on this color.
\end{remark}

\section{Yang-Baxter equations for colored models}
\label{sec:some-YBEs}

We shall describe Yang-Baxter equations for systems that generalize the Tokuyama model introduced in Section~\ref{sec:outline} by replacing its $\minus$ spins by a set
$\mathfrak{P}$ (called the \textit{palette}) of $r$ different
{\textit{colors}}. The set $\mathfrak{P}$ is ordered, and when convenient we
may take $\mathfrak{P}$ to be the set of integers $1 \leqslant c \leqslant r$.

If $A$ is a horizontal edge, the spinset $\Sigma_A$ is $\left\{ \plus \right\}
\cup \mathfrak{P}$. On the other hand if $A$ is a vertical edge, the spinset
$\Sigma_A$ is the power set of $\mathfrak{P}$ where it will be
convenient to identify the empty set with $\plus$ as before. Now in a state
$\mathfrak{s}$ of such a system, we say that a horizontal
edge $A$ {\textit{carries}} the color $c$ if $\mathfrak{s}_A = c$. If $A$ is a
vertical edge, we say that $A$ \emph{carries} the color $c$ if $c \in
\mathfrak{s}_A$, remembering that $\mathfrak{s}_A$ is a subset of
$\mathfrak{P}$. The vertical edges are thus allowed to carry more than one
color (with multiplicity at most $1$), while the horizontal edges may carry at
most one color.

\begin{figure}[b]
  \[\begin{array}{|c|c|c|c|}
  \hline
  \begin{tikzpicture}
    \gamgamunco{+}{+}{+}{+}{{z_i,z_j}}
  \end{tikzpicture} &
  \begin{tikzpicture}
      \gamgamunco{c}{c}{c}{c}{{z_i,z_j}}
      \spokea{red}{c}\spokeb{red}{c}\spokec{red}{c}\spoked{red}{c}
  \end{tikzpicture} &
  \begin{tikzpicture}
      \gamgamunco{c}{d}{d}{c}{{z_i,z_j}}
      \spokea{blue}{c}\spokeb{red}{d}\spokec{red}{d}\spoked{blue}{c}
  \end{tikzpicture} &  \begin{tikzpicture}
      \gamgamunco{c}{d}{c}{d}{{z_i,z_j}}
      \spokea{blue}{c}\spokeb{red}{d}\spokec{blue}{c}\spoked{red}{d}
  \end{tikzpicture} \\
  \hline
  z_j-vz_i & z_i-vz_j &
  \begin{array}{ll} (1-v)z_i &\text{if $c<d$}\\(1-v)z_j &\text{if
      $c>d$}\end{array} & \begin{array}{ll} z_i-z_j &\text{if $c>d$}\\v(z_i-z_j) &\text{if $c<d$}\end{array} \\
  \hline
   \begin{tikzpicture}
      \gamgamunco{c}{+}{+}{c}{{z_i,z_j}}
      \spokea{red}{c}\spokepb \spokepc \spoked{red}{c}
  \end{tikzpicture} & 
   \begin{tikzpicture}
      \gamgamunco{+}{c}{c}{+}{{z_i,z_j}}
      \spokepa \spokeb{red}{c} \spokec{red}{c} \spokepd
  \end{tikzpicture} &
  \begin{tikzpicture}
      \gamgamunco{+}{c}{+}{c}{{z_i,z_j}}
      \spokepa\spokeb{red}{c}\spokepc\spoked{red}{c}
  \end{tikzpicture} &
  \begin{tikzpicture}
      \gamgamunco{c}{+}{c}{+}{{z_i,z_j}}
      \spokea{red}{c}\spokepb\spokec{red}{c}\spokepd
  \end{tikzpicture} \\
  \hline
  (1-v) z_i & (1-v) z_j &
  v(z_i-z_j) & z_i-z_j\\
  \hline
  \end{array}\]
  \caption{The colored R-vertex weights. The colors $c$ and $d$ are an arbitrary choice of distinct colors in $\mathfrak{P}$.
  If a configuration does not appear in this table, the Boltzmann weight is
  zero. The associated R-matrix equals that of evaluation modules $V(z_j) \otimes V(z_i)$ of the quantum supergroup
  $U_{\sqrt{v}}(\widehat{\mathfrak{gl}}(r|1))$.}
  \label{coloredrmat}
\end{figure}

Having described the admissible configurations at each vertex, it remains to
describe the Boltzmann weights for both the ordinary and the R-vertices in
colored systems. The Boltzmann weights of the R-vertex are given in
Figure~\ref{coloredrmat}.

\begin{remark}
  It may be checked that with the Boltzmann weights in
  Figure~\ref{coloredrmat}, the R-matrix agrees with a Drinfeld twist of the
  (ungraded) R-matrix of evaluation modules $V (z_j) \otimes V (z_i)$ for the
  quantum affine Lie superalgebra $U_{\sqrt{v}} (\widehat{\mathfrak{gl}} (r|
  1))$ (cf.~{\cite[Definition 2.1]{Kojima}}). The $r$ colored spins span one
  graded piece in the super vector space, while the $\plus$ spin spans the
  remaining one-dimensional piece.
\end{remark}

\begin{figure}[tb]
  \[\begin{array}{|c|c|c|c|c|c|}
  \hline
  \begin{tikzpicture}\cross\gammaunco{+}{+}{+}{+}{}\cv{red}{c}\end{tikzpicture} &
  \begin{tikzpicture}\cross\gammaunco{d}{c}{d}{c}{}\cv{red}{c}\vspokea{blue}{d}\vspokec{blue}{d}\vspokeb{red}{c}\vspoked{red}{c}\end{tikzpicture} &
  \begin{tikzpicture}\cross\gammaunco{+}{c}{+}{c}{}\cv{red}{c}\vspokeb{red}{c}\vspoked{red}{c}\end{tikzpicture} &
  \begin{tikzpicture}\cross\gammaunco{d}{+}{d}{+}{}\cv{red}{c}\vspokea{blue}{d}\vspokec{blue}{d}\end{tikzpicture} &
  \begin{tikzpicture}\cross\gammaunco{d}{+}{+}{d}{}\cv{red}{c}\vspokea{red}{c}\vspoked{red}{c}\end{tikzpicture} &
  \begin{tikzpicture}\cross\gammaunco{+}{d}{d}{+}{}\cv{red}{c}\vspokeb{red}{c}\vspokec{red}{c}\end{tikzpicture} \\
  \hline
  1 & \begin{array}{ll}z_i&\text{if $d=c$}\\v&\text{if $c>d$}\\1&\text{if $c<d$}\end{array} &
  -v & \begin{array}{ll}z_i&\text{if $c=d$}\\1&\text{otherwise}\end{array} & (1-v)z_i & 1\\
  \hline  
  \end{array}\]
  \caption{Boltzmann weights for monochrome ordinary vertices. The weight depends on a pair of labels: a complex number $z_i$ (suppressed in pictures above) and a color (denoted $c$ above). Note that admissible vertical edges adjacent
    to the monochrome vertex may only carry the
    color $c$ of the vertex, while adjacent horizontal edges may
    carry any color. In particular, in the diagrams above, $c=d$ is allowed.}
  \label{cvmonochrome}
\end{figure}

The Boltzmann weights of the ordinary vertices, which adjoin two horizontal
edges and two vertical edges with many coloring possibilities, are harder to
describe. We will define these by means of fusion, starting with
simpler \textit{monochrome vertices}: vertices that adjoin
only \textit{monochrome edges} that are only allowed to carry at most one
particular color.
Each monochrome vertex is itself assigned a
color $c\in\mathfrak{P}$, and the spinset of a vertical edge attached to
that vertex is $\{\plus,c\}$. For horizontal edges, the spinset is
$\{\plus\}\cup\mathfrak{P}$ just as before.
The Boltzmann weights for the monochrome (ordinary) vertices are
given in Figure~\ref{cvmonochrome}.

\begin{convention}[Monochrome vertices]
\label{monochromedescriptions}
Now the admissible ordinary vertices and their weights may be described by
fusion of monochrome vertices as detailed in Section~\ref{sec:lattice-models}.
In a model with $r$ colors, we replace each
ordinary vertex by a single row of $r$ monochrome vertices with color labels
arranged in ascending order from left to right.
Recall that a vertical edge $A$ adjacent to an ordinary vertex is
decorated by a subset $\mathfrak{s}_A$ of $\mathfrak{P}$.
For the corresponding monochrome vertices we
color the $c$-th such edge (with color $c$) if and only if the color
$c$ appears in the set $\mathfrak{s}_A$. 
\end{convention}

\begin{remark}\label{expandedcollapsed}
Looking ahead to Section~\ref{coloredsystems}, we will consider
systems made from these ordinary, fused vertices. Regarding the vertex as a fusion, we may
replace the entire system by an \textit{expanded} or \emph{monochrome system} with
monochrome vertices; each column of vertices is replaced by $r$ different vertices. Then
we may refer to the system with fused edges as the \textit{fused}
system. See Figure~\ref{oceanmisty} for an example of
this procedure.
It follows from the definition of the fused weights that the
fused and expanded systems have the same partition function. Indeed, there
is a bijection between the states of the fused and expanded systems, and
corresponding states have the same Boltzmann weight, by definition.
\end{remark}

\begin{figure}[htb]
  \[
  \begin{array}{|c|c|c|}
    \hline
    \begin{tikzpicture}\crac\gammaunco{c}{d}{+}{}{z_i}\vspokea{red}{c}\vspokeb{blue}{d}\pspoked\end{tikzpicture} &
      \begin{tikzpicture}\gammaunco{c}{d}{c}{d}{}\vspokea{red}{c}\vspokeb{blue}{d}\vspokec{red}{c}\vspoked{blue}{d}\cv{blue}{d}
        \begin{scope}[shift={(1.50,0)}]
          \cross\gammaunco{c}{+}{+}{c}{}\vspokea{red}{c}\vspoked{red}{c}\cv{red}{c}
      \end{scope}\end{tikzpicture} &
      \begin{tikzpicture}\cross\gammaunco{c}{+}{+}{c}{}\vspokea{red}{c}\vspoked{red}{c}\cv{red}{c}
        \begin{scope}[shift={(1.50,0)}]
          \cross\gammaunco{+}{d}{+}{d}{}\vspokeb{blue}{d}\vspoked{blue}{d}\cv{blue}{d}
      \end{scope}\end{tikzpicture} \\
    \hline
    \begin{array}{ll} (1-v)z_i &\text{if $c>d$,}\\ (-v)(1-v)z_i &\text{if
        $c<d$.}\end{array} &
    \begin{array}{ll}\text{left:}&1\\\text{right:}&(1-v)z_i\end{array} &
    \begin{array}{ll}\text{left:}&(1-v)z_i\\\text{right:}&-v\end{array} \\
    \hline
  \end{array}
  \]
  \caption{Ordinary fused vertex constructed from monochrome vertices
    by fusion for $r=2$. Left: a fused vertex (compare with Figure~\ref{coloredweightsii}).
    Middle: The case $c>d$, using weights from
    Figure~\ref{cvmonochrome}. Right: The case $d>c$.}
  \label{coloredfusion}

  \hfill
  
  \[\begin{array}{|c|c|c|c|}
    \hline
    \begin{tikzpicture}\cross\gammaunco{+}{+}{+}{+}{z_i}\end{tikzpicture} &
    \begin{tikzpicture}\gammaunco{c}{c}{c}{c}{z_i}\vspokea{red}{c}\vspokeb{red}{c}\vspokec{red}{c}\vspoked{red}{c}\end{tikzpicture} &
    \begin{tikzpicture}\gammaunco{c}{d}{c}{d}{z_i}\vspokea{red}{c}\vspokeb{blue}{d}\vspokec{red}{c}\vspoked{blue}{d}\end{tikzpicture} &
    \begin{tikzpicture}\gammaunco{d}{c}{c}{d}{z_i}\vspokea{blue}{d}\vspokeb{red}{c}\vspokec{red}{c}\vspoked{blue}{d}\end{tikzpicture} \\
    \hline
    1 & z_i &
    \begin{array}{ll} z_i & \text{if $c>d$}\\ vz_i &\text{if $c<d$}\end{array} & 
    \begin{array}{ll} (1-v)z_i & \text{if $c>d$}\\ 0 &\text{if $c<d$}\end{array} \\
    \hline 
    \begin{tikzpicture}\craa\crac\gammaunco{+}{c}{+}{c}{z_i}\vspokeb{red}{c}\vspoked{red}{c}\end{tikzpicture} &
    \begin{tikzpicture}\crab\crad\gammaunco{c}{+}{c}{+}{z_i}\vspokea{red}{c}\vspokec{red}{c}\end{tikzpicture} &
    \begin{tikzpicture}\craa\crad\gammaunco{+}{c}{c}{+}{z_i}\vspokeb{red}{c}\vspokec{red}{c}\end{tikzpicture} &
    \begin{tikzpicture}\crab\crac\gammaunco{c}{+}{+}{c}{z_i}\vspokea{red}{c}\vspoked{red}{c}\end{tikzpicture} \\
    \hline
    -v & z_i & 1 & (1-v)z_i\\
    \hline
  \end{array}\]
  \caption{Fused weights (I). These are the Boltzmann weights in which the
    vertical edges carry no more than one color.  Since edges can carry more
    than one color, this is not a complete list of the possibilities. In this
    figure, $c\neq d$ except where explicitly allowed.}
  \label{coloredweightsi}
  
  \hfill
  
  \[\begin{array}{|c|c|c|c|}
  \hline
  \begin{tikzpicture}\crac\gammaunco{c}{d}{+}{}{z_i}\vspokea{red}{c}\vspokeb{blue}{d}\pspoked\end{tikzpicture} &
  \begin{tikzpicture}\craa\gammaunco{+}{}{d}{c}{z_i}\pspokeb\vspokec{blue}{d}\vspoked{red}{c}\end{tikzpicture} &
  \begin{tikzpicture}\craa\crac\gammaunco{+}{}{+}{}{z_i}\pspokeb\pspoked\end{tikzpicture} &
  \begin{tikzpicture}\gammaunco{-}{-}{+}{-}{z_i}\vspokea{red}{c}\vspokec{red}{c}\pspokeb\pspoked\end{tikzpicture} \\
    \hline
  \begin{array}{ll} (1-v)z_i &\text{if $c>d$,}\\ (-v)(1-v)z_i &\text{if $c<d$.}\end{array} &
  \begin{array}{ll} v &\text{if $c>d$,}\\ -v &\text{if $c<d$.}\end{array} &
  (-v)^2 & 
  \begin{array}{ll} z_i &\text{if $c>d$,}\\ vz_i &\text{if $c<d$.}\end{array} \\
  \hline
  \end{array}\]
  \caption{Fused weights (II). The vertical edges can carry more than one
    color, with multiplicity at most one (so $c\neq d$ in this figure). These are
    the extra possibilities when at most two colors appear.}
  \label{coloredweightsii}
\end{figure}
\begin{figure}[htb]
  \[
  \begin{array}{|c|c|c|c|}
    \hline
    \begin{tikzpicture}\gamgamunco{+}{+}{+}{+}{c}\spokepa\spokepb\spokepc\spokepd\cv{red}{c}\end{tikzpicture} &
    \begin{tikzpicture}\gamgamunco{d}{d}{d}{d}{c}\spokea{blue}{d}\spokeb{blue}{d}\spokec{blue}{d}\spoked{blue}{d}\cv{red}{c}\end{tikzpicture} &
    \begin{tikzpicture}\gamgamunco{d}{e}{d}{e}{c}\cv{red}{c}
      \spokea{blue}{d}\spokeb{darkgreen}{e}\spokec{blue}{d}\spoked{darkgreen}{e}\cv{red}{c}\end{tikzpicture} &
    \begin{tikzpicture}\gamgamunco{d}{e}{e}{d}{c}\cv{red}{c}
      \spokea{blue}{d}\spokeb{darkgreen}{e}\spokec{darkgreen}{e}\spoked{blue}{d}\cv{red}{c}\end{tikzpicture} \\
    \hline
    z_j-vz_i & \begin{array}{c}z_i-vz_j\\\text{$c=d$ allowed.}\end{array} &
      \begin{array}{c}\begin{array}{ll} v(z_i-z_j)&\text{if $e>d$}\\z_i-z_j&\text{if $d>e$}\\\end{array}\\\text{$c=d$ or $e$ allowed}\end{array}&
      \begin{array}{ll}
        (1-v)z_j&\text{if $e>c>d$}\\
                &\text{or $c>d>e$}\\
                &\text{or $d>e>c$}\\
        (1-v)z_i&\text{if $d>c>e$}\\
                &\text{or $c>e>d$}\\
                &\text{or $e>d>c$}\end{array}\\
    \hline
    \begin{tikzpicture}\gamgamunco{d}{c}{c}{d}{c}\cv{red}{c}\spokea{blue}{d}\spokeb{red}{c}\spokec{red}{c}\spoked{blue}{d}\cv{red}{c}\end{tikzpicture} &
    \begin{tikzpicture}\gamgamunco{c}{d}{d}{c}{c}\cv{red}{c}\spokea{red}{c}\spokeb{blue}{d}\spokec{blue}{d}\spoked{red}{c}\cv{red}{c}\end{tikzpicture} &
    \begin{tikzpicture}\gamgamunco{+}{d}{+}{d}{c}\spokepa\spokeb{blue}{d}\spokepc\spoked{blue}{d}\cv{red}{c}\end{tikzpicture}&
    \begin{tikzpicture}\gamgamunco{d}{+}{d}{+}{c}\spokea{blue}{d}\spokepb\spokec{blue}{d}\spokepd\cv{red}{c}\end{tikzpicture}\\
    \hline
    (1-v)z_j&(1-v)z_i&\begin{array}{c}v(z_i-z_j)\\\text{$c=d$ allowed}\end{array}&\begin{array}{c}z_i-z_j\\\text{$c=d$ allowed}\end{array}\\
    \hline
    &\begin{tikzpicture}\gamgamunco{d}{+}{+}{d}{c}\spokea{blue}{d}\spokepb\spokepc\spoked{blue}{d}\cv{red}{c}\end{tikzpicture}&
    \begin{tikzpicture}\gamgamunco{+}{d}{d}{+}{c}\spokepa\spokeb{blue}{d}\spokec{blue}{d}\spokepd\cv{red}{c}\end{tikzpicture}&\\
    \hline
    &\begin{array}{c}(1-v)z_i\\\text{$c=d$ allowed}\end{array}&
    \begin{array}{c}(1-v)z_j\\\text{$c=d$ allowed}\end{array}&\\
    \hline
  \end{array}\]
  \caption{R-vertices for auxiliary Yang-Baxter equations. These are labeled by a color $c$ and a pair of parameters $(z_i, z_j)$ (suppressed in the pictures above). If the
    color $c$ is minimal, that is if $c\leqslant d,e$ for all colors that
    appear in this figure, this agrees with the R-vertices in
    Figure~\ref{coloredrmat}. In this figure, the colors $c,d,e$ are distinct
    except when $c=d$ or $c=e$ is explicitly allowed.}
  \label{gencolrmat}
\end{figure}

As described in Section~\ref{sec:lattice-models}, the Boltzmann weight of the fused vertex is just the partition
function of the single row of these ordered $r$ monochrome vertices, which has
at most one admissible state. In Figure~\ref{coloredfusion}, we compute an
example of a fused Boltzmann weight when $r=2$ from the corresponding
monochrome vertices. In Figure~\ref{coloredweightsi} we give all the fused
Boltzmann weights (for any $r$) in which the vertical edges carry at most one
color. The possible cases in which vertical edges carry two colors are shown
in Figure~\ref{coloredweightsii}. For $r > 2$ one would have to complete these
with similar tables for vertical edges carrying more colors. 
At the end of this section we will give all the fused weights in a closed form in a notation similar to the one used in \cite{BorodinWheelerColored}. See Figure~\ref{coloredalabw}.

It remains to discuss the Yang-Baxter equation for fused vertices, which will result from 
auxiliary Yang-Baxter equations for the monochrome model according to Lemma~\ref{lemfusedybe}.
We first need to define monochrome R-vertices for use in~(\ref{ybepre}), generalizing the R-vertices in Figure~\ref{coloredrmat}. These will play
a role of the vertices labeled $\zeta_k$ in Lemma~\ref{lemfusedybe}, but now each such R-matrix
depends not only on a pair of complex parameters $z_i$, $z_j$, but also on a color $c$.
The Boltzmann weights for these are given in Figure~\ref{gencolrmat}.

Let $R(z_i,z_j)$ denote the R-matrix constructed with the weights in
Figure~\ref{coloredrmat} according to~(\ref{rmatrixbw}), and if $1\leqslant c\leqslant r$ is a color,
let $R^{(c)}(z_i,z_j)$ denote the colored R-matrix
constructed from the Boltzmann weights in Figure~\ref{gencolrmat} where the vertex is labeled by the color $c$. Note that $R^{(1)}=R$. Also, let $T^{(c)}(z_i)$ denote the
matrix associated with the monochrome (ordinary) vertices labeled by the color $c$ whose Boltzmann weights are described in Figure~\ref{cvmonochrome}.
We recall that the colors $c$ are identified with the integers $1\leqslant
c\leqslant r$, so there is a next color $c+1$ unless $c$ is the last color
$c=r$, in which case we define $R^{(r+1)} := R$. We may now describe auxiliary Yang-Baxter equations involving the
monochrome vertices.

\begin{proposition}
  \label{coloredaux}
  If $1 \leqslant c \leqslant r$, then
  \begin{equation}
  \label{ybemono}
  R^{(c)}(z_i,z_j)_{12}T^{(c)}(z_i)_{13}T^{(c)}(z_j)_{23}=
  T^{(c)}(z_j)_{23}T^{(c)}(z_i)_{13}R^{(c+1)}(z_i,z_j)_{12}.
  \end{equation}
\end{proposition}

\begin{proof}
  Note that since we are using monochrome edges, at most three different
  colors can appear in the boundary (and interior) spins in the
  equivalent description \eqref{ybepre} of \eqref{ybemono}. There are
  two more colors $c$ and $c+1$ which must also be compared with these
  three in the values for $T^{(c)}$, $R^{(c)}$, $R^{(c+1)}$ from
  Figures~\ref{coloredrmat} and Figure~\ref{cvmonochrome}.
  All possibilities are covered if we take 5 colors and we conclude
  that if the Yang-Baxter equation is checked for $r=5$ then it is
  true for all $r$. We checked the Yang-Baxter equation for $r=5$
  using a computer and (\ref{ybemono}) is proved.
\end{proof}

\begin{theorem} \label{thm:coloredybe}
  The Yang-Baxter equation for colored models is satisfied:
  \[R(z_i,z_j)_{12}T(z_i)_{13}T(z_j)_{23}=
  T(z_j)_{23}T(z_i)_{13}R(z_i,z_j)_{12}\]
\end{theorem}

\begin{proof}
   This follows from Proposition~\ref{coloredaux} and Lemma~\ref{lemfusedybe}.
\end{proof}

\begin{remark}
  Restricting the palette $\mathfrak{P}$ to a single color, the weights in Figures~\ref{coloredweightsi} recover the Tokyuama weights seen for example in the row labeled $S^\Gamma(i)$ of~\cite[Table~2]{hkice} if we replace this color with a $\minus$ spin. The R-matrices for the two models are also equal.
\end{remark}

We will now describe the fused weights in a closed form.
For comparison with \cite{BorodinWheelerColored}, we will
choose a notation close to theirs.
In~\cite{BorodinWheelerColored}, vertical edges are labeled by
tuples $\I=(I_1,\cdots,I_r)\in\mathbb{N}^r$ representing a
state in which the $k$-th color has multiplicity $I_k$.
The principal difference between their systems and ours
is that colors can only occur with multiplicity $0$ or $1$ in our systems.
In other words, if we imitate their setup, each $I_k\in\{0,1\}$.
Hence the same data can be specified by the subset $\Sigma=\{k \mid I_k=1\}$
of the palette $\mathfrak{P}$.

In~\cite{BorodinWheelerColored}, an operation adds (resp.~removes)
a color $a$ to the tuple $\I$, that is, increments (resp.~decrements)
$I_a$ and the resulting tuple is denoted $\I^+_a$ (resp.~$\I^-_a$).
We therefore introduce the corresponding operations on the set $\Sigma$
and denote $\Sigma^+_a=\Sigma\cup\{a\}$, to be used only if $a\notin\Sigma$,
and $\Sigma^-_a=\Sigma\setminus\{a\}$, to be used only if $a\in\Sigma$.
Finally if $a\in\Sigma$ and $b\notin\Sigma$, we will denote
$\tensor*{\Sigma}{*^+_a^-_b}=\Sigma\cup\{a\}\setminus\{b\}$, also corresponding to the
$\tensor*{\I}{*^+_a^-_b}$ in \cite{BorodinWheelerColored}.
If $1\leqslant a\leqslant b\leqslant r$, we will define
$\Sigma_{[a,b]}=\{c\in\Sigma \mid a\leqslant c\leqslant b\}$.

In Figure~\ref{coloredalabw} we give our Boltzmann weights in
closed form using these notations. It is easy to see that these
are the correct weights obtained from the monochrome weights by fusion.

\newcommand{\arrowdiagram}[4]{
\begin{tikzpicture}[anchor=base, baseline]
  \draw[line width=1.0pt ] (-1,0) node[left]{$#1$} -- (1,0) node[right]{$#3$};
  \draw[line width=5pt] (0,1) node[above]{$#2$} -- (0,-1) node[below]{$#4$};
\end{tikzpicture}
}

\begin{figure}[ht]
\begin{equation*}
  \begin{array}{|c|c|c|}\hline
    \arrowdiagram{+}{\Sigma}{+}{\Sigma} & \arrowdiagram{c}{\Sigma}{c}{\Sigma} & \arrowdiagram{+}{\Sigma}{c}{\Sigma^-_c} \\ &  & \\ 
    (-v)^{|\Sigma_{[1,r]}|} & z_i v^{|\Sigma_{[c+1,r]}|} & (-v)^{|\Sigma_{[1,c-1]}|} v^{|\Sigma_{[c+1,r]}|}
    \\\hline
    \arrowdiagram{c}{\Sigma}{+}{\Sigma^+_c} & \arrowdiagram{c}{\Sigma}{d}{\tensor*{\Sigma}{*^+_c^-_d}} & \arrowdiagram{d}{\Sigma}{c}{\tensor*{\Sigma}{*^+_d^-_c}} \\ & & \\
    (1-v) z_i (-v)^{|\Sigma_{[c+1,r]}|} & (1-v)z_i(-v)^{|\Sigma_{[c+1,d-1]}|} v^{|\Sigma_{[d+1,r]}|} & 0
    \\ \hline
  \end{array} 
\end{equation*}
\caption{Colored Boltzmann weights in a style
  resembling~\cite{BorodinWheelerColored} (2.2.2) or (2.2.6), except that we are using
  $+$ for the `colorless' horizontal edges, and in place of their multiset
  $\I$ of colors, we use a subset $\Sigma$ of the palette. We are assuming $c < d$.}
\label{coloredalabw}
\end{figure}

\begin{remark}
\label{bwcomparisonremark}
The weights in Figure~\ref{coloredalabw} closely resemble weights presented in 
Section~2.2 of~\cite{BorodinWheelerColored}. One important distinction is that our
weights are `fermionic' --- we do not allow multiple copies of any given color on an edge ---
while their weights are `bosonic' (allowing multiplicities). Nevertheless, we
may compare the weights of the multiplicity-free colored vertices
in~\cite{BorodinWheelerColored} with those in Figure~\ref{coloredalabw}; even
allowing for changes of variables and Drinfeld twisting, small differences
persist.  For example, we may take the mirror image of the weights
in~\cite{BorodinWheelerColored} (2.2.2) specialized by setting $s=0$ and
compare to the weights in Figure~\ref{coloredalabw} by making the
substitutions $v=q^{-1}$ and $z_i=x$. Note that if no color occurs more than
once on the top boundary then no vertical edge on the interior of the grid
will ever have a color with multiplicity larger than one, so that the
distinction between bosonic and fermionic weights becomes unimportant. Then
both weight schemes have the same admissible vertices grouped into types as in
Figure~\ref{coloredalabw}. Upon Drinfeld twisting, the Boltzmann weights agree
in their powers of $x$ and $(1-q^{-1})$ but differ by various factors of $-1$
and $q$ that cannot be resolved. Furthermore, although we have noted that
bosonic weights can be excluded by imposing boundary conditions they are still
important for the Yang-Baxter equation, so the R-matrices
for~\cite{BorodinWheelerColored} and for our system must definitely be
different. 
Indeed, our R-matrix is a Drinfeld twist of the R-matrix for the quantum group for $\widehat{\mathfrak{gl}}(r|1)$ while the R-matrix in~\cite{BorodinWheelerColored} is a Drinfeld twist of the one for $\widehat{\mathfrak{sl}}_{r+1}$.
For further comparison of our models
with~\cite{BorodinWheelerColored} see arXiv version~1 of this paper~\cite{ThisPaper}.
\end{remark}


\section{The Iwahori lattice model}\label{coloredsystems}

We now describe a family of statistical-mechanical systems made from fused
vertices whose partition functions may be shown to give values
of Iwahori Whittaker functions in the case $G = \GL_r$. Indeed, if
$g\in\GL_r(F)$, where $F$ is a nonarchimedean local field,
we will see that the Whittaker function $\phi_w(g)$ defined by (\ref{eq:iwahori-whittaker-function})
can be represented as the partition function of such a model. First note that, as mentioned in Remark~\ref{iwahoripg}, we may assume that $g$ is of the form $\varpi^{-\lambda}w_2$ for some
weight $\lambda$ and Weyl group element $w_2$. By Lemma~\ref{almostdominantlemma}, we may assume that $\lambda$
is $w_2$-almost dominant. Finally, multiplying $g$ by an element of the center
simply multiplies $\phi_w(g)$ by a scalar, so we may assume with no loss
of generality that the entries in $\lambda$ are nonnegative.

Having already explained how the Boltzmann weights for fused vertices are
defined, it remains to explain the boundary conditions for the model and the
labels on each of the vertices. For any positive integer $r$, the boundary
conditions and vertex labels depend on three pieces of data: a partition
denoted $\lambda+\rho$ with at most $r$ nonzero parts, a pair of permutations
$w_1, w_2 \in S_r = W$, the Weyl group of $\GL_r$, and $r$ complex
parameters $\z = (z_1, \ldots, z_r)$. The systems we present here, denoted
$\mathfrak{S}_{\z, \lambda, w_1, w_2}$ and referred to as \emph{Iwahori systems}, may be considered as simultaneous
generalizations of those appearing previously in the Tokuyama model and in the
colored systems of~\cite{demice}. Our goal in this section is to equate the
partition function of $\mathfrak{S}_{\z, \lambda, w_1, w_2}$ with the value of
the Whittaker function $\phi_{w_1}(\z; \varpi^{-\lambda} w_2)$ of
Section~\ref{Iwahoripreliminaries}.

With $r$ fixed, let $\rho = (r-1, \ldots, 1, 0)$ and let $\lambda+\rho=(\lambda_1+r-1,\cdots,\lambda_r)$ be a partition, whose parts are written in weakly decreasing order as usual. In the identification of the weight lattice of $\GL_r$ with $\mathbb{Z}^r$, the corresponding weight $\lambda$ satisfies $\langle \alpha, \lambda \rangle \geqslant -1$ for all simple roots $\alpha$, a necessary condition for the non-vanishing of the Whittaker function according to Lemma~\ref{almostdominantlemma}.

Given the partition $\lambda+\rho$, we form a rectangular lattice consisting
of $N+1$ columns and $r$ rows, where $N$ is any integer at least
$\lambda_1+r-1$.  The columns will be numbered from left to right from $N$ to
$0$ in decreasing order.  The rows are numbered $1$ to $r$, in increasing
order from top to bottom. Given $\z$, each vertex in the $i$-th row receives
the label $z_i$. The Boltzmann weights are the fused weights in
Figure~\ref{coloredalabw}. Unless otherwise stated we will henceforth assume
that the parameter $v$ appearing in the Boltzmann weights (as well as in the
Demazure-Whittaker operators among other places) equals $q^{-1}$ with $q$ the
cardinality of the residue field of $F$. We will prefer the use of $v$ to avoid confusion in later sections where, to follow tradition, $q$ will have another meaning. 

It remains to describe the boundary spins and colors located around the edge of the rectangular grid. They depend on the choice of the weight $\lambda$ and the two Weyl group elements $w_1, w_2$ as follows and summarized in Figure~\ref{fig:boundary}.
\begin{figure}
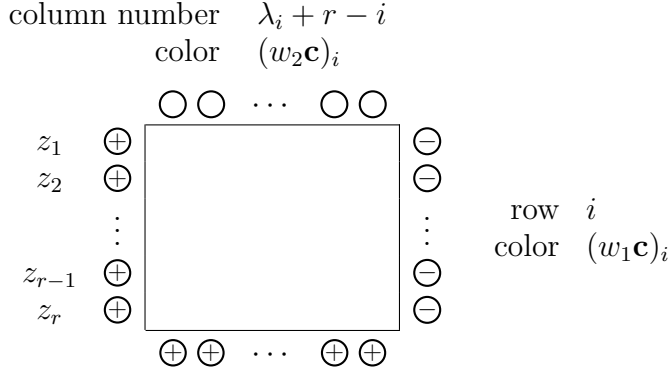

\newcommand{\plusminus}{\tikz[anchor=base, baseline=-1pt]{\node[scale=0.8, draw, thick, circle, inner sep=0pt] {\phantom{$\pm$}};}\xspace}
\begin{tabular}{rlc}
        {column number} \quad & {$\lambda_i+r-i$}\\
        {color} \quad & {$(w_2\mathbf{c})_i$} \\[0.5em]
  \multicolumn{2}{c}{
  \begin{tabular}{cc|ccc|c}
    \multicolumn{2}{c}{} & \plusminus \plusminus & $\cdots$ & \multicolumn{1}{c}{\plusminus \plusminus} &  \tabularnewline \cline{3-5}
    $z_1$ & \plus & & & & \minus \tabularnewline
    $z_2$ & \plus & & & & \minus \tabularnewline
          & $\vdots$ & & & & $\vdots$ \tabularnewline
    $z_{r-1}$ & \plus & & & & \minus \tabularnewline
    $z_r$ & \plus & & & & \minus \tabularnewline \cline{3-5}
    \multicolumn{6}{c}{} \tabularnewline[-1em]
    \multicolumn{2}{c}{} & \plus \plus & $\cdots$ & \multicolumn{1}{c}{\plus \plus} & 
  \end{tabular}}
  & \begin{tabular}{rl} row & $i$ \\ color & {$(w_1\mathbf{c})_i$} \end{tabular}
\end{tabular}
\caption{\label{fig:boundary} Summary of boundary conditions for
the colored systems. Let $\mathbf{c}=(c_1,\cdots,c_r)$ be a
semistandard flag of colors $c_1\geqslant \cdots \geqslant c_r$;
for the systems $\mathfrak{S}_{\z, \lambda, w_1, w_2}$, if we identify
the colors $c_i$ with integers we take $c_i=r+1-i$, while in
Section~\ref{sec:parahoric_systems} we will take more general
flags. On the top boundary the minus signs are positioned at columns
$\lambda_i + r - i$ with color $(w_2\mathbf{c})_i=c_{w_2^{-1}(i)}$,
and on the right boundary the edge in row $i$ is colored
$(w_1\mathbf{c})_i=c_{w_1^{-1}(i)}$.}

\end{figure}
We have colors numbered $1, \ldots, r$
at our disposal. For the top boundary, we
assign the color~$r+1-w_2^{-1}(i)$ to the edge in the column labeled $\lambda_i+r-i$ for each
$i \in \{1, \ldots, r\}$ and a \plus spin in the remaining columns. That is, we color each edge whose column index is a part of $\lambda+\rho$ and we have multiple colors on a given top boundary edge according to the multiplicity of parts in the partition. Then, we put a \plus spin on
all the left and bottom boundary edges. This leaves the right
boundary edges to be described. These will depend on the choice of permutation $w_1 \in W$ for the system.
The right boundary edge in the $i$-th row is assigned the color $r+1-w_1^{-1}(i)$. 
For $w_2=1$, these boundary conditions are exactly as in~\cite{demice}. A particularly simple admissible state in an Iwahori system is given in Figure~\ref{groundstate}.

\begin{figure}[t]
\begin{tikzpicture}[scale=0.75, font=\small]
    \draw (1,0)--(1,6);
    \draw (3,0)--(3,6);
    \draw (5,0)--(5,6);
    \draw (7,0)--(7,6);
    \draw (9,0)--(9,6);
    \draw (11,0)--(11,6);
    \draw (13,0)--(13,6);
    \draw (15,0)--(15,6);
    \draw (0,1)--(16,1);
    \draw (0,3)--(16,3);
    \draw (0,5)--(16,5);
    \draw [line width=0.5mm,red] (1.00,5.75)--(1,5)--(1.75,5.00);
    \draw [line width=0.5mm,red] (2.35,5.00)--(3.75,5.00);
    \draw [line width=0.5mm,red] (4.35,5.00)--(5.75,5.00);
    \draw [line width=0.5mm,red] (6.35,5.00)--(7.75,5.00);
    \draw [line width=0.5mm,blue] (9,4)--(9,6);
    \draw [line width=0.5mm,red] (8.35,5.00)--(9.75,5.00);
    \draw [line width=0.5mm,red] (10.35,5.00)--(11.75,5.00);
    \draw [line width=0.5mm,red] (12.35,5.00)--(13.75,5.00);
    \draw [line width=0.5mm,darkgreen] (15,4)--(15,6);
    \draw [line width=0.5mm,red] (14.35,5.00)--(15.75,5.00);
    \draw [line width=0.5mm,blue] (9.00,3.75)--(9,3)--(9.75,3.00);
    \draw [line width=0.5mm,blue] (10.35,3.00)--(11.75,3.00);
    \draw [line width=0.5mm,blue] (12.35,3.00)--(13.75,3.00);
    \draw [line width=0.5mm,darkgreen] (15,2)--(15,4);
    \draw [line width=0.5mm,blue] (14.35,3.00)--(15.75,3.00);
    \draw [line width=0.5mm,darkgreen] (15.00,1.75)--(15,1)--(15.75,1.00);
    \draw[line width=0.5mm,red,fill=white] (1,6) circle (.35);
    \draw[fill=white] (3,6) circle (.35);
    \draw[fill=white] (5,6) circle (.35);
    \draw[fill=white] (7,6) circle (.35);
    \draw[line width=0.5mm,blue,fill=white] (9,6) circle (.35);
    \draw[fill=white] (11,6) circle (.35);
    \draw[fill=white] (13,6) circle (.35);
    \draw[line width=0.5mm,darkgreen,fill=white] (15,6) circle (.35);
    \draw[fill=white] (1,4) circle (.35);
    \draw[fill=white] (3,4) circle (.35);
    \draw[fill=white] (5,4) circle (.35);
    \draw[fill=white] (7,4) circle (.35);
    \draw[line width=0.5mm,blue,fill=white] (9,4) circle (.35);
    \draw[fill=white] (11,4) circle (.35);
    \draw[fill=white] (13,4) circle (.35);
    \draw[line width=0.5mm,darkgreen,fill=white] (15,4) circle (.35);
    \draw[fill=white] (1,2) circle (.35);
    \draw[fill=white] (3,2) circle (.35);
    \draw[fill=white] (5,2) circle (.35);
    \draw[fill=white] (7,2) circle (.35);
    \draw[fill=white] (9,2) circle (.35);
    \draw[fill=white] (11,2) circle (.35);
    \draw[fill=white] (13,2) circle (.35);
    \draw[line width=0.5mm,darkgreen,fill=white] (15,2) circle (.35);
    \draw[fill=white] (1,0) circle (.35);
    \draw[fill=white] (3,0) circle (.35);
    \draw[fill=white] (5,0) circle (.35);
    \draw[fill=white] (7,0) circle (.35);
    \draw[fill=white] (9,0) circle (.35);
    \draw[fill=white] (11,0) circle (.35);
    \draw[fill=white] (13,0) circle (.35);
    \draw[fill=white] (15,0) circle (.35);
    \draw[fill=white] (0,5) circle (.35);
    \draw[line width=0.5mm,red,fill=white] (2,5) circle (.35);
    \draw[line width=0.5mm,red,fill=white] (4,5) circle (.35);
    \draw[line width=0.5mm,red,fill=white] (6,5) circle (.35);
    \draw[line width=0.5mm,red,fill=white] (8,5) circle (.35);
    \draw[line width=0.5mm,red,fill=white] (10,5) circle (.35);
    \draw[line width=0.5mm,red,fill=white] (12,5) circle (.35);
    \draw[line width=0.5mm,red,fill=white] (14,5) circle (.35);
    \draw[line width=0.5mm,red,fill=white] (16,5) circle (.35);
    \draw[fill=white] (0,3) circle (.35);
    \draw[fill=white] (2,3) circle (.35);
    \draw[fill=white] (4,3) circle (.35);
    \draw[fill=white] (6,3) circle (.35);
    \draw[fill=white] (8,3) circle (.35);
    \draw[line width=0.5mm,blue,fill=white] (10,3) circle (.35);
    \draw[line width=0.5mm,blue,fill=white] (12,3) circle (.35);
    \draw[line width=0.5mm,blue,fill=white] (14,3) circle (.35);
    \draw[line width=0.5mm,blue,fill=white] (16,3) circle (.35);
    \draw[fill=white] (0,1) circle (.35);
    \draw[fill=white] (2,1) circle (.35);
    \draw[fill=white] (4,1) circle (.35);
    \draw[fill=white] (6,1) circle (.35);
    \draw[fill=white] (8,1) circle (.35);
    \draw[fill=white] (10,1) circle (.35);
    \draw[fill=white] (12,1) circle (.35);
    \draw[fill=white] (14,1) circle (.35);
    \draw[line width=0.5mm,darkgreen,fill=white] (16,1) circle (.35);
    \node at (1,6) {$3$};
    \node at (3,6) {$+$};
    \node at (5,6) {$+$};
    \node at (7,6) {$+$};
    \node at (9,6) {$2$};
    \node at (11,6) {$+$};
    \node at (13,6) {$+$};
    \node at (15,6) {$1$};
    \node at (1,4) {$+$};
    \node at (3,4) {$+$};
    \node at (5,4) {$+$};
    \node at (7,4) {$+$};
    \node at (9,4) {$2$};
    \node at (11,4) {$+$};
    \node at (13,4) {$+$};
    \node at (15,4) {$1$};
    \node at (1,2) {$+$};
    \node at (3,2) {$+$};
    \node at (5,2) {$+$};
    \node at (7,2) {$+$};
    \node at (9,2) {$+$};
    \node at (11,2) {$+$};
    \node at (13,2) {$+$};
    \node at (15,2) {$1$};
    \node at (1,0) {$+$};
    \node at (3,0) {$+$};
    \node at (5,0) {$+$};
    \node at (7,0) {$+$};
    \node at (9,0) {$+$};
    \node at (11,0) {$+$};
    \node at (13,0) {$+$};
    \node at (15,0) {$+$};
    \node at (0,5) {$+$};
    \node at (2,5) {$3$};
    \node at (4,5) {$3$};
    \node at (6,5) {$3$};
    \node at (8,5) {$3$};
    \node at (10,5) {$3$};
    \node at (12,5) {$3$};
    \node at (14,5) {$3$};
    \node at (16,5) {$3$};
    \node at (0,3) {$+$};
    \node at (2,3) {$+$};
    \node at (4,3) {$+$};
    \node at (6,3) {$+$};
    \node at (8,3) {$+$};
    \node at (10,3) {$2$};
    \node at (12,3) {$2$};
    \node at (14,3) {$2$};
    \node at (16,3) {$2$};
    \node at (0,1) {$+$};
    \node at (2,1) {$+$};
    \node at (4,1) {$+$};
    \node at (6,1) {$+$};
    \node at (8,1) {$+$};
    \node at (10,1) {$+$};
    \node at (12,1) {$+$};
    \node at (14,1) {$+$};
    \node at (16,1) {$1$};
    \node at (1,6.8) {$7$};
    \node at (3,6.8) {$6$};
    \node at (5,6.8) {$5$};
    \node at (7,6.8) {$4$};
    \node at (9,6.8) {$3$};
    \node at (11,6.8) {$2$};
    \node at (13,6.8) {$1$};
    \node at (15,6.8) {$0$};
    \node at (-.75,5) {$ 1$};
    \node at (-.75,3) {$ 2$};
    \node at (-.75,1) {$ 3$};
    \path[fill=white] (1,1) circle (.25);
    \node at (1,1) {\scriptsize$z_3$};
    \path[fill=white] (3,1) circle (.25);
    \node at (3,1) {\scriptsize$z_3$};
    \path[fill=white] (5,1) circle (.25);
    \node at (5,1) {\scriptsize$z_3$};
    \path[fill=white] (7,1) circle (.25);
    \node at (7,1) {\scriptsize$z_3$};
    \path[fill=white] (9,1) circle (.25);
    \node at (9,1) {\scriptsize$z_3$};
    \path[fill=white] (11,1) circle (.25);
    \node at (11,1) {\scriptsize$z_3$};
    \path[fill=white] (13,1) circle (.25);
    \node at (13,1) {\scriptsize$z_3$};
    \path[fill=white] (15,1) circle (.25);
    \node at (15,1) {\scriptsize$z_3$};

    \path[fill=white] (1,3) circle (.25);
    \node at (1,3) {\scriptsize$z_2$};
    \path[fill=white] (3,3) circle (.25);
    \node at (3,3) {\scriptsize$z_2$};
    \path[fill=white] (5,3) circle (.25);
    \node at (5,3) {\scriptsize$z_2$};
    \path[fill=white] (7,3) circle (.25);
    \node at (7,3) {\scriptsize$z_2$};
    \path[fill=white] (9,3) circle (.25);
    \node at (9,3) {\scriptsize$z_2$};
    \path[fill=white] (11,3) circle (.25);
    \node at (11,3) {\scriptsize$z_2$};
    \path[fill=white] (13,3) circle (.25);
    \node at (13,3) {\scriptsize$z_2$};
    \path[fill=white] (15,3) circle (.25);
    \node at (15,3) {\scriptsize$z_2$};

    \path[fill=white] (1,5) circle (.25);
    \node at (1,5) {\scriptsize$z_1$};
    \path[fill=white] (3,5) circle (.25);
    \node at (3,5) {\scriptsize$z_1$};
    \path[fill=white] (5,5) circle (.25);
    \node at (5,5) {\scriptsize$z_1$};
    \path[fill=white] (7,5) circle (.25);
    \node at (7,5) {\scriptsize$z_1$};
    \path[fill=white] (9,5) circle (.25);
    \node at (9,5) {\scriptsize$z_1$};
    \path[fill=white] (11,5) circle (.25);
    \node at (11,5) {\scriptsize$z_1$};
    \path[fill=white] (13,5) circle (.25);
    \node at (13,5) {\scriptsize$z_1$};
    \path[fill=white] (15,5) circle (.25);
    \node at (15,5) {\scriptsize$z_1$};
  \end{tikzpicture}
\caption{The ground state. Here $r=3$, $\lambda=(5,2,0)$ so
  $\lambda+\rho=(7,3,0)$, $w_1 = w_2 = 1$ (the identity permutation), and $\z = (z_1, z_2, z_3)$
  is an arbitrary triple of complex numbers. 
  The top row colors read from left to right are are
  $(3,2,1)$. The colors on the right edge, read
  from top to bottom, are also $(3,2,1)$. This is the unique
  state of the system $\mathfrak{S}_{\z,\lambda,1,1}$.
  Its Boltzmann weight is $\z^{\lambda+\rho}$.}
\label{groundstate}
\end{figure}
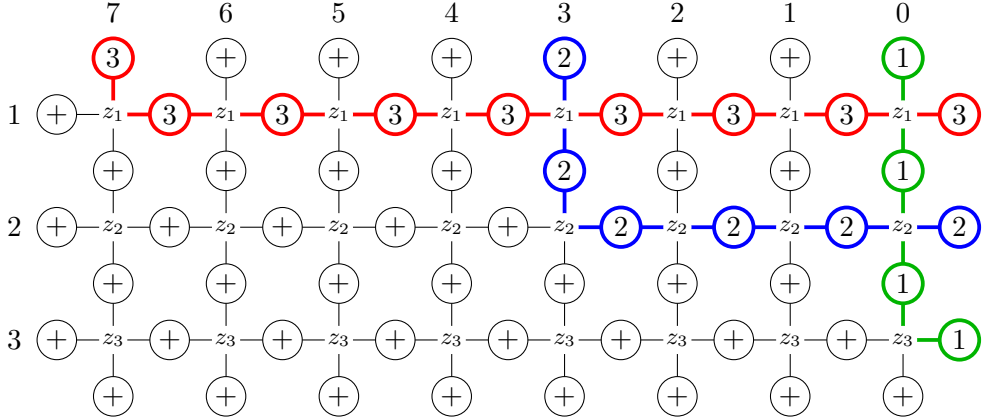

In any state of the system $\mathfrak{S}_{\z,\lambda,w_1, w_2}$, the
edges of any one particular color form a line or path starting at the
top boundary and ending at the right boundary. This depiction of admissible states as configurations of 
lines is present in many works on lattice models, for example Baxter's book~\cite{Baxter}, 
Chapter~8. The idea of using \textit{colored} lines and refined systems that specify starting and ending points of each colored line is presented in \cite{BorodinWheelerColored}. We exploited this idea in a prior paper~\cite{demice}
to give a new theory of Demazure atoms, nonsymmetric pieces of Schur functions. The colored weights in this paper specialize to those of~\cite{demice} by setting $v=0$, which leads to a vast simplification. In particular, every edge in~\cite{demice} may carry at most one color (even in the fused model) and two colored lines can
cross at most once. In this paper, weights and subsequent Yang-Baxter equations are understood via fusion, and two colored lines can cross more than once.

Let $Z (\mathfrak{S}_{\mathbf{z}, \lambda, w_1,w_2})$ denote the partition
function of the system $\mathfrak{S}_{\z,\lambda,w_1, w_2}$. We will now demonstrate that this partition function satisfies the same functional equation as the Iwahori Whittaker function $\phi_{w_1}$ in Theorem~\ref{mainrecursion} under Demazure-Whittaker operators using the Yang-Baxter equation. 
It will be convenient to conjugate the Demazure-Whittaker operators $\mathfrak{T}_i$ of \eqref{demwhitop} as follows
\begin{equation}
  \label{tconjugation}
  T_i=\z^\rho\mathfrak{T}_i\z^{-\rho}.
\end{equation}
such that  
\begin{align} 
  \label{eq:demazurewhittaker}
  T_i\cdot f(\z) &= \frac{\z^{\alpha_i} - v}{1 - \z^{\alpha_i}} f(s_i \z) + \frac{v-1}{1 - \z^{\alpha_i}} f(\z) \\
  \intertext{and}
  \label{eq:demazurewhittakershort}
  T_i^{-1} \cdot f(\z) &= \frac{\z^{\alpha_i} - v}{v(1 - \z^{\alpha_i})} f(s_i \z) + \frac{(v-1) \z^{\alpha_i}}{v(1 - \z^{\alpha_i})} f(\z) \, .
\end{align}

\begin{proposition}\label{proposition:funceq} For any partition $\lambda+\rho$, simple reflection $s_i$, and any pair of Weyl group elements $w_1, w_2 \in W$,
\begin{equation}
  \label{zwrecur}
  Z (\mathfrak{S}_{\mathbf{z}, \lambda, s_i w_1,w_2}) =
  \begin{cases}
    T_i Z (\mathfrak{S}_{\mathbf{z}, \lambda, w_1,w_2})  & \textrm{if $\ell(s_i w_1) > \ell(w_1)$}, \\ 
    T_i^{-1} Z (\mathfrak{S}_{\mathbf{z}, \lambda, w_1,w_2}) & \textrm{if $\ell(s_i w_1) < \ell(w_1)$.} 
  \end{cases}
\end{equation}

\end{proposition}

\begin{proof}
Repeated use of the Yang-Baxter equation gives the equality of the
partition functions in Figure~\ref{modpf}.
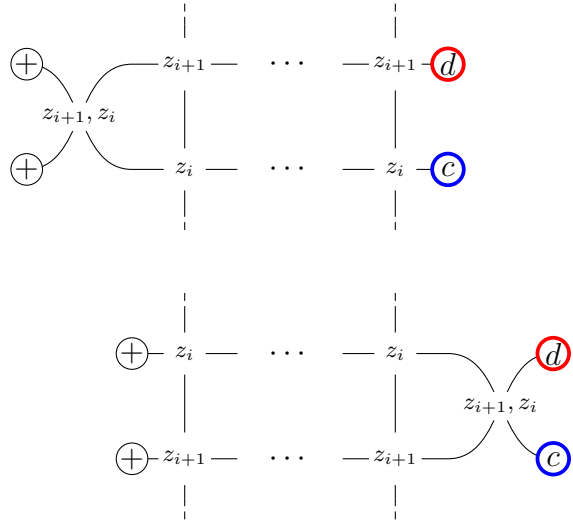
\begin{figure}[htb]
\begin{tikzpicture}[scale=0.7]
\begin{scope}[shift={(-1,0)}]
  \draw (0,1) to [out = 0, in = 180] (2,3) to (4,3);
  \draw (0,3) to [out = 0, in = 180] (2,1) to (4,1);
  \draw (3,0.25) to (3,3.75);
  \draw (7,0.25) to (7,3.75);
  \draw (6,1) to (8,1);
  \draw (6,3) to (8,3);
  \draw[fill=white] (0,1) circle (.3);
  \draw[fill=white] (0,3) circle (.3);
  \draw[line width=0.5mm,red,fill=white] (8,3) circle (.3);
  \draw[line width=0.5mm,blue,fill=white] (8,1) circle (.3);
  \node at (0,1) {$+$};
  \node at (0,3) {$+$};
  \node at (5,3) {$\cdots$};
  \node at (5,1) {$\cdots$};
  \draw[densely dashed] (3,3.75) to (3,4.25);
  \draw[densely dashed] (3,0.25) to (3,-0.25);
  \draw[densely dashed] (7,3.75) to (7,4.25);
  \draw[densely dashed] (7,0.25) to (7,-0.25);
  \node at (8,1) {$c$};
  \node at (8,3) {$d$};
\path[fill=white] (3,3) circle (.5);
\node at (3,3) {\scriptsize$z_{i+1}$};
\path[fill=white] (3,1) circle (.4);
\node at (3,1) {\scriptsize$z_i$};
\path[fill=white] (7,3) circle (.5);
\node at (7,3) {\scriptsize$z_{i+1}$};
\path[fill=white] (7,1) circle (.4);
\node at (7,1) {\scriptsize$z_i$};
\path[fill=white] (1,2) circle (.3);
\node at (1,2) {\scriptsize$z_{i+1},z_i$};
\end{scope}
\begin{scope}[shift={(1,-5.5)}]
  \draw (4,1) to (6,1) to [out = 0, in = 180] (8,3);
  \draw (4,3) to (6,3) to [out = 0, in = 180] (8,1);
  \draw[line width=0.5mm,red,fill=white] (8,3) circle (.3);
  \draw[line width=0.5mm,blue,fill=white] (8,1) circle (.3);
  \draw (0,1) to (2,1);
  \draw (0,3) to (2,3);
  \draw (5,0.25) to (5,3.75);
  \draw (1,0.25) to (1,3.75);
  \draw[fill=white] (0,1) circle (.3);
  \draw[fill=white] (0,3) circle (.3);

  \node at (3,1) {$\cdots$};
  \node at (3,3) {$\cdots$};

  \draw[densely dashed] (1,3.75) to (1,4.25);
  \draw[densely dashed] (1,0.25) to (1,-0.25);
  \draw[densely dashed] (5,3.75) to (5,4.25);
  \draw[densely dashed] (5,0.25) to (5,-0.25);
  \path[fill=white] (1,3) circle (.4);
  \node at (1,3) {\scriptsize$z_{i}$};
  \path[fill=white] (1,1) circle (.5);
  \node at (1,1) {\scriptsize$z_{i+1}$};

  \path[fill=white] (5,3) circle (.4);
  \node (a) at (5,3) {\scriptsize$z_{i}$};
  \path[fill=white] (5,1) circle (.5);
  \node at (5,1) {\scriptsize$z_{i+1}$};

  \path[fill=white] (7,2) circle (.3);
  \node at (7,2) {\scriptsize$z_{i+1},z_i$};
  \node at (8,1) {$c$};
  \node at (8,3) {$d$};
  \node at (0,1) {$+$};
  \node at (0,3) {$+$};
\end{scope}
\end{tikzpicture}
\caption{Top: the system $\mathfrak{S}_{s_i\z,\lambda,w_1,w_2}$ with
  the R-matrix attached. Bottom: after using the Yang-Baxter equation.}
\label{modpf}
\end{figure}

Using the $R$-matrix weights from Figure~\ref{coloredrmat}, we obtain the following identity of partition functions:
\begin{multline}
 \label{predemazure}
 (z_i - v z_{i+1}) Z (\mathfrak{S}_{s_i \mathbf{z}, \lambda, w_1,w_2})  = \\
  \begin{cases}
       (1 - v) z_{i+1} Z (\mathfrak{S}_{\mathbf{z}, \lambda, w_1,w_2})  + (z_{i+1} - z_i)
 Z (\mathfrak{S}_{\mathbf{z}, \lambda, s_i w_1,w_2})  &\quad\text{if } \ell(s_i w_1) > \ell(w_1),\\
       (1 - v) z_{i} Z (\mathfrak{S}_{\mathbf{z}, \lambda, w_1,w_2})  + v(z_{i+1} - z_i)
 Z (\mathfrak{S}_{\mathbf{z}, \lambda, s_iw_1,w_2}) &\quad\text{if } \ell(s_i w_1) < \ell(w_1). \\
     \end{cases}
\end{multline}
Note that $s_iw>w$ is equivalent
to $c>d$ in the notation of Figure~\ref{coloredrmat}. Consulting the table, there is one possible choice for the
R-matrix for the top state in Figure~\ref{modpf}, and two possible choices for the bottom state, accounting for the
three terms in the identity (\ref{predemazure}). Note that we take
$(i,j)$ in Figure~\ref{coloredrmat} to be $(i+1,i)$.

Setting $\z^{\alpha_i} = z_i / z_{i+1}$, and rearranging terms
in~\eqref{predemazure} upon division by $z_{i+1}$, we obtain the desired equality.
\end{proof}

We have noted in Remark~\ref{expandedcollapsed} that we may replace
a system such as $\mathfrak{S}_{\z,\lambda,w_1,w_2}$ made with fused
Boltzmann weights by an equivalent system with $r$ times as many
vertices, using monochrome weights. The expanded monochrome system
will appear in the following proof. See Figure~\ref{oceanmisty} for
an example.

\begin{theorem}
\label{coloredwhittaker}
Given any $w_1, w_2 \in W$, let $\lambda$ be a $w_2$-almost dominant weight and let $\mathfrak{S}_{\z,\lambda,w_1, w_2}$ be the corresponding Iwahori system. Then
\begin{equation}
\label{zisphi} 
 Z (\mathfrak{S}_{\mathbf{z}, \lambda, w_1,w_2}) = \z^{\rho}\phi_{w_1}(\z; \varpi^{-\lambda} w_2).
\end{equation}
\end{theorem}

\begin{proof}
  Comparing Proposition~\ref{proposition:funceq} and
  Theorem~\ref{mainrecursion} while bearing in mind~(\ref{tconjugation}), both
  sides of~(\ref{zisphi}) satisfy the same recursive formula, so if~(\ref{zisphi}) is true for one value of $w_1$, it is true for all
  $w_1$. Thus we may assume that $w_1 = w_2$.
  
  We will show that when $w_1 = w_2 = w$ the system has a unique state. We
  will use the monochrome model, in which each vertical edge has been broken
  into $r$ distinct vertical edges, and the color $c$, if we identify $c$ with
  an integer $1 \leqslant c \leqslant r$ can only be carried by the $c$-th
  such vertical edge. The following argument shows that the condition that $\lambda$ is $w_2$-dominant
  implies that the sequence of colors on the top boundary edges are the same
  as the sequence of colors on the right boundary edges. By definition of
  $\mathfrak{S}_{\mathbf{z}, \lambda, w, w}$ the sequence of colors on the
  right edge are $r + 1 - w^{- 1} (i)$. On the top edge, the color in the
  $\lambda_i + r - i$ column is also $r + 1 - w^{- 1} (i)$, and the sequence
  of integers $\lambda_i + r - i$ is weakly decreasing. Since columns are
  labeled in decreasing order, we see that if the $\lambda_i + r - i$ are
  distinct, then the colors are in the same order on the top boundary and on the
  right boundary, as claimed. But we must consider what happens if several
  $\lambda_i + r - i$ are equal, as in Figure~\ref{oceanmisty}. If $\lambda_i + r - i =
  \lambda_{i + 1} + r - (i + 1)$ then $\langle \lambda, \alpha_i^{\vee}
  \rangle = - 1$ so our condition that $\lambda$ is $w$-almost dominant
  implies that $w^{- 1} \alpha_i$ is a negative root. Therefore $w^{- 1} (i) >
  w^{- 1} (i + 1)$ and so the colors on the right edge in rows $i, i + 1$ are
  $r + 1 - w^{- 1} (i) < r + 1 - w^{- 1} (i + 1)$. Now let us see that this
  agrees with the condition for the top boundary. Indeed, when we split the vertices
  into monochrome vertices as in Convention~\ref{monochromedescriptions}, they are in
  increasing order.
  
  We have shown that the colors of the top boundary edges of the monochrome model are
  in the same order as those of the right boundary edges. From this it is easily
  deduced that there is only one possible state, and that every colored line
  crosses every other colored line (exactly once). We need to consider the Boltzmann weights
  that arise from these crossings. Consulting the second case in Figure~\ref{cvmonochrome}
  we see that when $c > d$, the crossing produces a factor of $v$, otherwise
  it does not. The total number of such crossings is the number of inversions
  of $w^{- 1}$, that is $\ell (w)$. Also in the $i$-th row, using the fused
  (non-monochrome) description, the number of factors $z_i$ will be the
  number of vertices with a colored edge to the left, which will be
  $z_i^{\lambda_i + r - i}$. Therefore the Boltzmann weight of the state is
  therefore $v^{\ell (w)} \mathbf{z}^{\lambda + \rho}$. By Proposition~\ref{basecaseprop}
  this equals $\mathbf{z}^{\rho} \phi_w (\mathbf{z}, \varpi^{- \lambda} w)$,
  and this concludes the proof.
\end{proof}

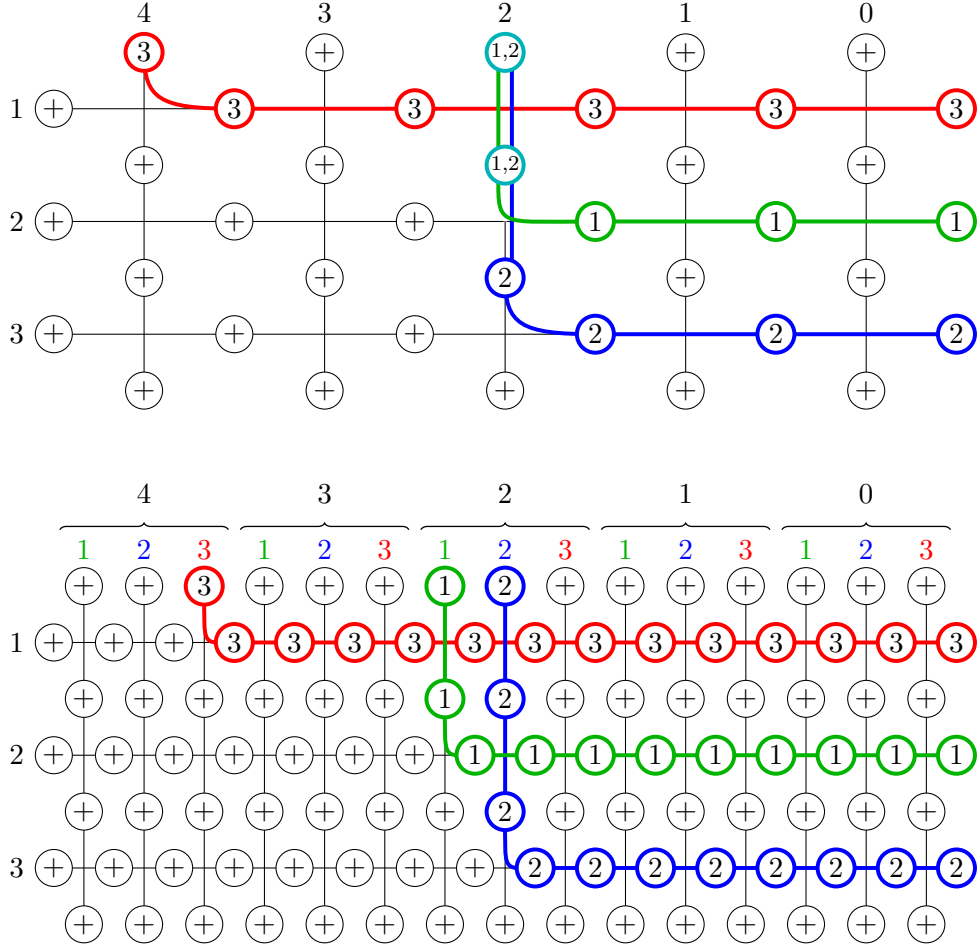
\begin{figure}[htb]
  \begin{centering}
  \begin{tikzpicture}[xscale=1.2, yscale=0.75, font=\small]
    \icegrid{4}{3}
    \draw[white,line width=0.5mm] (5,3) to (5,6);
    \draw [red, path] 
    (1,6) node {$3$} to [out=-90,in=180,looseness=1.5]
    (2,5) node {$3$} -- 
    (4,5) node {$3$} -- 
    (6,5) node {$3$} -- 
    (8,5) node {$3$} -- 
    (10,5) node {$3$};

    \draw [blue, path] (5+0.075,6) -- (5+0.075,2);
    \draw [darkgreen, path] (5-0.075,6) -- (5-0.075,4) to [out=-90,in=180,looseness=2.0] (6,3);
    
    \draw [BlueGreen, invisiblepath] 
    (5,6) node[draw=BlueGreen] {$\scriptstyle 1,2$} -- 
    (5,4) node[draw=BlueGreen] {$\scriptstyle 1,2$};

    \draw [blue, path]
    (5,2) node{$2$} to [out=-90,in=180,looseness=1.5] 
    (6,1) node{$2$} -- 
    (8,1) node{$2$} -- 
    (10,1) node{$2$};

    \draw[darkgreen, path] 
    (6,3) node {$1$} --
    (8,3) node {$1$} --
    (10,3) node {$1$};

  \end{tikzpicture} \\[2em]

\begin{tikzpicture}[xscale=0.4, yscale=0.75, font=\small]

      \icegridnocolumnlabels{14}{3}

      \draw[red, path]
      (5,6) node {$3$} to [out=-90,in=180,looseness=2]
      (6,5) node {$3$} --
      (8,5) node {$3$} --
      (10,5) node {$3$} -- 
      (12,5) node {$3$} -- 
      (14,5) node {$3$} -- 
      (16,5) node {$3$} -- 
      (18,5) node {$3$} -- 
      (20,5) node {$3$} -- 
      (22,5) node {$3$} -- 
      (24,5) node {$3$} -- 
      (26,5) node {$3$} -- 
      (28,5) node {$3$} -- 
      (30,5) node {$3$};

      \draw[blue, path]
      (15,6) node {$2$} --
      (15,4) node {$2$} --
      (15,2) node {$2$} to [out=-90,in=180,looseness=2]
      (16,1) node {$2$} --
      (18,1) node {$2$} --
      (20,1) node {$2$} --
      (22,1) node {$2$} --
      (24,1) node {$2$} --
      (26,1) node {$2$} --
      (28,1) node {$2$} --
      (30,1) node {$2$};

      \draw[darkgreen, path]
      (13,6) node {$1$} --
      (13,4) node {$1$} to [out=-90,in=180,looseness=2]
      (14,3) node {$1$} --
      (16,3) node {$1$} --
      (18,3) node {$1$} --
      (20,3) node {$1$} --
      (22,3) node {$1$} --
      (24,3) node {$1$} --
      (26,3) node {$1$} --
      (28,3) node {$1$} --
      (30,3) node {$1$};

      \foreach \i in {0,...,4}{
        \pgfmathtruncatemacro\col{int(4-\i)}
        \draw[line width=0.5pt, decoration={brace},decorate] (6*\i+0.2,7) -- node[above=6pt] {$\col$} (6*\i+6-0.2,7);
        \node[darkgreen] at (6*\i+1,6.65) {$1$};
        \node[blue] at (6*\i+3,6.65) {$2$};
        \node[red] at (6*\i+5,6.65) {$3$};

      }
  \end{tikzpicture}
  
  \end{centering}
  
  \caption{
    The unique state of $\mathfrak{S}_{\z,\lambda, w_1, w_2}$ for $G=\GL_3$ with $w_1=w_2=s_2$ in $W=S_3$ and $\lambda=(2,1,2)$ shown with fused vertices (top) and monochrome vertices (bottom). 
    Note that while $\lambda$ is not dominant, it is $w_2$-almost dominant. 
    The partition function $vz_1^4z_2^2z_3^2$ of this system equals $\mathbf{z}^\rho \phi_{s_2}(\varpi^{-\lambda}s_2)$.
    \label{oceanmisty} }
\end{figure}

\begin{proposition}
  \label{bijectprop}
  Let $\mu' \in \mathbb{Z}^r$. There
  exists a unique pair $(w, \lambda)$, with $w \in W$ and $\lambda$ a $w$-almost dominant weight, such that
  \[ w(\mu')= \lambda +\rho. \]
\end{proposition}

\begin{proof}
  We may find $w$ and $\lambda = (\lambda_1, \cdots,
  \lambda_r)$ such that $w(\mu')= \lambda+\rho$ and $\lambda+\rho$
  is dominant. Clearly $\lambda$ is unique but $w$ may not be if we only
  require $w(\mu')$ to be dominant. However the stronger condition
  that $w(\mu')-\rho$ is $w$-almost dominant will force $w$ to be unique as follows.
  
  We recall that $w^{- 1} \alpha_i \in \Delta^+$ if and only if
  $\ell (s_i w) > \ell (w)$. The $w$ such that $w(\mu')=\lambda+\rho$
  lie in a single left coset of the stabilizer of $\lambda+\rho$,
  which is a Coxeter group generated by the $s_i$ such that $\lambda_i+1=\lambda_{i+1}$.
  But the condition that $w(\mu')-\rho$ is $w$-almost dominant is equivalent to
  the assumption that $\ell (s_i w) < \ell (w)$ whenever $\lambda_i+1 =
  \lambda_{i+1}$. So this condition means that any $s_i$ among
  the generators of this stabilizer is a left
  descent of $w$. Thus clearly there is a unique $w$ in this coset
  such that $\lambda$ is $w$-almost dominant, and that is the longest
  element of~$W$ such that $w (\mu') = \lambda+\rho$.
\end{proof}

\begin{remark}\label{bijectremark}
  As noted above, the `standard basis' of Iwahori Whittaker functions $\phi_{w_1}$ are determined by their values at $\varpi^{-\lambda} w_2$. 
  We have shown in Theorem~\ref{coloredwhittaker}
  that these values are partition functions of
  certain systems $\mathfrak{S}_{\mathbf{z},\lambda,w_1,w_2}$. Proposition~\ref{bijectprop}
  shows that the partition function of every Iwahori system is a value of an Iwahori Whittaker function.
  Indeed, the data describing the system are colorings of
  the top and right boundary edges. In other words, the data are
  two maps from
  the set of colors to the top boundary edges (labeled by
  columns) and to the right boundary edges (labeled by rows). The map to
  rows is bijective but the map to columns can be any map;
  as in Figure~\ref{oceanmisty} it does not need to be
  injective. Let $\mu'_i$ be the column corresponding
  to the $(r+1-i)$-th color. Applying Proposition~\ref{bijectprop}
  to $\mu'=(\mu'_1,\cdots,\mu'_r)$ produces a pair
  $(w_2,\lambda)$ such that $\lambda$ is $w_2$-almost dominant and $w_2 \mu' = \lambda + \rho$. In column $\lambda_i + \rho_i = (w_2 \mu')_i = \mu'_{w_2^{-1}(i)}$ we then have the color $r+1-w_2^{-1}(i)$ exactly as specified for the top boundary edges in Figure~\ref{fig:boundary}. Thus, from every $\mu'$ we obtain a system $\mathfrak{S}_{\mathbf{z},\lambda,w_1,w_2}$
  with $w_1$ determined by the permutation of colors on the right edge.
\end{remark}

\section{\label{sec:parahoric_systems}The parahoric lattice model}

In this section we will generalize the Iwahori lattice model to allow
multiple colored lines of the same color. This allows us to represent
parahoric Whittaker functions as partition functions. See
Figure~\ref{figureparahoricstates} for an example.

We will call a sequence of $r$ colors a \emph{flag}.  The boundary conditions
of the colored systems are represented by two flags: one on the top edge, and
one on the right edge, which is a permutation of the former.  A flag
$\mathbf{c}= (\gamma_1, \ldots, \gamma_r)$ is called {\textit{standard}} if
$\gamma_1 > \cdots > \gamma_r$, and {\textit{semistandard}} if $\gamma_1
\geqslant \cdots \geqslant \gamma_r$. Since we have a palette of $r$ colors there is a
unique standard flag, and if, as in Section~\ref{sec:some-YBEs}, we identify
the colors with integers $1 \leqslant c \leqslant r$, then the unique standard flag is
$(r, r - 1, \ldots, 1)$. Any flag may be represented as $w\mathbf{c}=
(\gamma_{w^{- 1} (1)}, \ldots, \gamma_{w^{- 1} (r)})$, where $\mathbf{c}$ is a
semistandard flag and $w \in W = S_r$. For the Iwahori systems in Section~\ref{coloredsystems} every color
appeared exactly once on the top boundary and on the right boundary,
meaning the boundaries were represented by permutations $w_1, w_2 \in W$ of a
standard flag $\mathbf{c}$.
For the parahoric systems these boundaries will instead be represented by permutations of a semistandard flag.

If $\mathbf{c}$ is a semistandard flag then we may write
\begin{equation}
  \label{eq:semistandard-flag}
  \mathbf{c}= (\gamma_1, \cdots, \gamma_r) = (c_1^{r_1}, c_2^{r_2}, \cdots,
   c_k^{r_k}) 
\end{equation}
with $c_1 > \cdots > c_k$ and $\sum r_i = r$. Here the notation means that
we have $r_1$ copies of $c_1$ followed by $r_2$ copies of $c_2$, and so
forth. The stabilizer of this flag in $W$ is the parabolic subgroup
$W_{\mathbf{J}} = S_{r_1} \times \cdots \times S_{r_k}$, which is
the Weyl group of $\GL_{r_1} \times \cdots \times \GL_{r_k}$. Here, using the
notation of Section~\ref{sec:parahoric}, $\mathbf{J}$ is the index set of simple
reflections generating this subgroup of~$W$, that is, all simple reflections except $s_{r_1}, s_{r_1 +
r_2}, \cdots$. If a general flag is written $w\mathbf{c}$ with $\mathbf{c}=
(c_1^{r_1}, c_2^{r_2}, \ldots, c_k^{r_k})$ then we may choose the representative
$w$ to be in $W^{\mathbf{J}}$, meaning that it is the shortest element of the coset
$w W_{\mathbf{J}}$.

\begin{figure}[tb]
  \begin{centering}
\begin{tikzpicture}[scale=0.7, font=\small]
    \newcommand\eps{0.15}
    \icegrid{4}{3}
    \draw [red, path] 
    (1,6) node {$3$} to [out=-90,in=180,looseness=2.0]
    (2,5) node {$3$} -- 
    (4,5) node {$3$} -- 
    (6,5) node {$3$} to [out=0, in=90, looseness=2.0]
    (7,4) node {$3$} to [out=-90, in=180, looseness=2.0]
    (8,3) node {$3$} --
    (10,3) node {$3$};
    
    \draw [blue, path]
    (5,6) node{$2$} --
    (5,4) node{$2$} to [out=-90, in=180, looseness=2.0]
    (6,3) node{$2$} to [out=0, in=90, looseness=2.0]
    (7,2) node{$2$} to [out=-90, in=180,looseness=2.0]
    (8,1) node{$2$} -- 
    (10,1) node{$2$};

    \draw [red, path]
    (9,6) node{$3$} to [out=-90, in=180, looseness=2.0]
    (10,5) node{$3$};
  \end{tikzpicture}
  \hfill
  \begin{tikzpicture}[scale=0.7, font=\small]
    \newcommand\eps{0.15}
    \icegridnorowlabels{4}{3}
    \draw [red, path] 
    (1,6) node {$3$} to [out=-90,in=180,looseness=2.0]
    (2,5) node {$3$} -- 
    (4,5) node {$3$} -- 
    (6,5) node {$3$} -- 
    (8,5) node {$3$} -- 
    (10,5) node {$3$};

    \draw [blue, path] (5,6) -- (5,2);
    \draw [red, path] (5+\eps,6) -- (5+\eps,4) to [out=-90,in=180,looseness=2.0] (6,3);
    
    \draw [purple, invisiblepath] 
    (5,6) node {$\scriptstyle 2,3$} -- 
    (5,4) node {$\scriptstyle 2,3$};

    \draw [blue, path]
    (5,2) node{$2$} to [out=-90,in=180,looseness=2.0] 
    (6,1) node{$2$} -- 
    (8,1) node{$2$} -- 
    (10,1) node{$2$};

    \draw[red, path] 
    (6,3) node {$3$} --
    (8,3) node {$3$} --
    (10,3) node {$3$};
  \end{tikzpicture}
\end{centering}
  \caption{Two states for parahoric Whittaker functions corresponding to the
  Levi subgroup $\GL_2 \times \GL_1 \subset \GL_3$. The set of
  simple reflections in this Levi subgroup indexed by $\mathbf{J}$ is $\{s_1\}$.
  On the left we have a state for the system
  $\mathfrak{S}_{\mathbf{z}, \lambda, w_1,w_2}^{\mathbf{J}}$ with $\lambda=
  (2,1,0)$, $w_1=1$ and $w_2=s_2$, and on the right we have a state for the system with
  $\lambda=(2,2,1)$, $w_1=1$ and $w_2=s_2$.}
  \label{figureparahoricstates}
\end{figure}
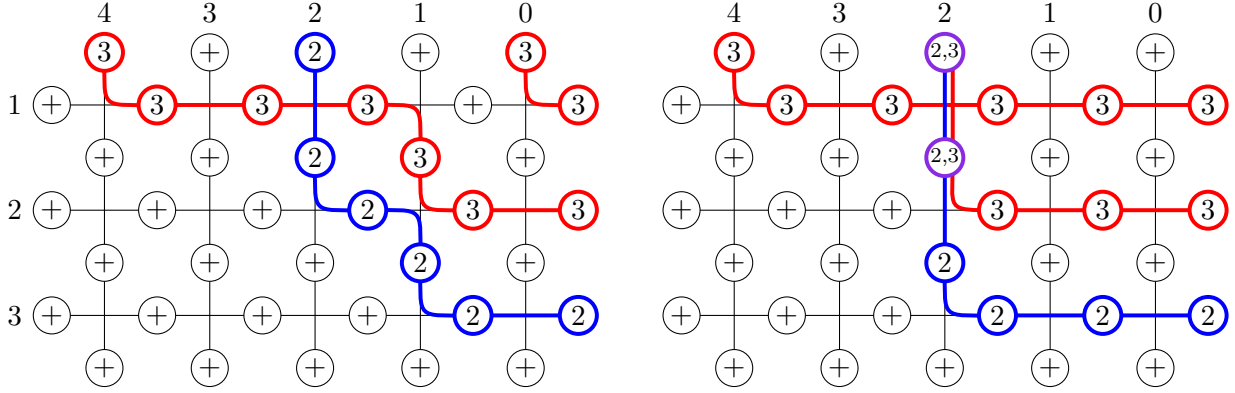    

We now explain the parahoric lattice model generalizing the Iwahori lattice model of Section~\ref{coloredsystems}. 
Let $\mathbf{c}$ be a semistandard flag parametrized as above and let
$w_1, w_2 \in W^{\mathbf{J}}$.
Assume that $\lambda$ is $w_2$-almost dominant. 

To this data we associate a parahoric system $\mathfrak{S}_{\mathbf{z},\lambda,w_1, w_2}^{\mathbf{J}}$ as follows.
We start from the same construction as for the Iwahori lattice model, with the only difference being which boundary conditions we allow.
We take the top boundary edge spin in column $\lambda_i + r - i$ to be
the color $(w_2 \mathbf{c})_i = \gamma_{w_2^{- 1} (i)}$; the remaining top boundary
edge spins are $\plus$. The right boundary spin in row $i$ is $(w_1 \mathbf{c})_i =
\gamma_{w_1^{- 1} (i)}$. The boundary spins on the left and bottom edges are $\plus$.
For the interior, we use the same Boltzmann weights in Figure~\ref{coloredalabw} as before.
Note that if $\mathbf{c}$ is a standard flag we recover the Iwahori lattice model from the parahoric model.

Recall that our `fermionic' Boltzmann weights do not allow more than one instance
of the same color on a given vertical edge.
The following proposition implies that this requirement is satisfied for the top boundary edges (and therefore for all vertical edges).

\begin{proposition}
  \label{nodupcolors}Suppose that $w_2 \in W^{\mathbf{J}}$, and assume that
  $\lambda$ is a $w_2$-almost dominant weight. If $\lambda_i + r - i = \lambda_j + r -
  j$ then $\gamma_{w_2^{- 1} (i)} \neq \gamma_{w_2^{- 1} (j)}$.
\end{proposition}

\begin{proof}
  If $1 \leqslant i, j \leqslant r$ let $t_{ij}$ denote the transposition in
  $W = S_r$ that interchanges $i$ and $j$. Then obviously
  \begin{equation}
    \label{transpocriterion} \gamma_i = \gamma_j \hspace{0.27em}
    \Longleftrightarrow \hspace{0.27em} t_{ij} \in W_{\mathbf{J}}
    \hspace{0.27em} .
  \end{equation}
  Without loss of generality assume $i < j$. If $\lambda$ is $w_2$-almost
  dominant then $\lambda + \rho$ is dominant, so the sequence $\lambda_i + r -
  i$ is monotone nonincreasing. Therefore
  \[ \lambda_i + r - i = \lambda_{i + 1} + r - i - 1 = \cdots = \lambda_j + r
     - j \]
  and so $\langle \lambda, \alpha_i^{\vee} \rangle = \langle \lambda,
  \alpha_{i + 1}^{\vee} \rangle = \cdots = \langle \lambda, \alpha_{j -
  1}^{\vee} \rangle = - 1$. Because $\lambda$ is $w_2$-almost dominant, it
  follows that $w_2^{- 1} (\alpha_i), \cdots, w_2^{- 1} (\alpha_{j - 1})$ are
  all negative roots. Thus $w_2^{- 1}  (\alpha_i + \alpha_{i + 1} + \ldots +
  \alpha_{j - 1})$ is a negative root. Now the reflection in the hyperplane
  orthogonal to the root $\alpha_i + \alpha_{i + 1} + \ldots + \alpha_{j - 1}$
  is $t_{ij}$ so by~{\cite{BjornerBrenti}} Proposition~4.4.6 we have $\ell
  (w_2^{- 1} t_{ij}) < \ell (w_2^{- 1})$, or equivalently $\ell (t_{ij} w_2) <
  \ell (w_2)$. Now $t_{ij} w_2 = w_2 t_{w_2^{- 1} (i), w_2^{- 1} (j)}$. We
  claim that this implies that $t_{w_2^{- 1} (i), w_2^{- 1} (j)} \not\in
  W_{\mathbf{J}}$. Indeed, $w_2 \in W^{\mathbf{J}}$ so $w_2$ is the shortest
  element in the coset $w_2 W_{\mathbf{J}}$. Since we have shown that $\ell
  (w_2 t_{w_2^{- 1} (i), w_2^{- 1} (j)}) < \ell (w_2)$, this would be a
  contradiction if $t_{w_2^{- 1} (i), w_2^{- 1} (j)} \in W_{\mathbf{J}}$. It
  now follows from (\ref{transpocriterion}) that $\gamma_{w_2^{- 1} (i)} \neq
  \gamma_{w_2^{- 1} (j)}$.
\end{proof}

\begin{remark}
  \label{rem:reparam}We observe that the Boltzmann weights of Figure~\ref{coloredalabw} are constructed from the monochrome weights of Figure~\ref{cvmonochrome} and that the monochrome weights 
  involving more than one color only depend on the orders of the colors.
  Therefore if we choose another set of colors such that $c_1' > \cdots >
  c_k'$ and replace each color $c_i$ by $c_i'$, the partition function is
  unchanged. We refer to this type of transformation as
  \textit{reparametrization}. Because of this, we are permitted to omit the
  colors $c_1, \ldots, c_k$ from the notation
  $\mathfrak{S}_{\mathbf{z}, \lambda,w_1, w_2}^{\mathbf{J}}$. 
\end{remark}

We may now state the main theorem of this section that relates a parahoric Whittaker function $\psi_{w_1}^{\mathbf{J}}$ of
Section~\ref{sec:parahoric} to the partition function of a parahoric system.

\begin{theorem}
  \label{thm:parahoric-lattice-model}Assume that $w_1, w_2 \in W^{\mathbf{J}}$
  and that $\lambda$ is $w_2$-almost dominant. Then
  \begin{equation}
    \label{eq:parpar} Z (\mathfrak{S}_{\mathbf{z}, \lambda, w_1,
    w_2}^{\mathbf{J}}) = \mathbf{z}^{\rho} \psi_{w_1}  (\mathbf{z} ; \varpi^{-
    \lambda} w_2) .
  \end{equation}
\end{theorem}

The proof follows from Lemmas~\ref{lem:jkind} and~\ref{lem:parahoric-1-reduction} below.
We begin with a reduction. Suppose
that $\mathbf{K} \subseteq \mathbf{J}$.  Then $W_{\mathbf{K}} \subseteq W_{\mathbf{J}}$ and $W^{\mathbf{K}} \supseteq W^{\mathbf{J}}$. From the
definition {\eqref{eq:psidef}} of $\psi^{\mathbf{J}}_w$ we have
\begin{equation}
  \label{kjreduction} \psi_{w_1}^{\mathbf{J}} (\mathbf{z} ; g) = \sum_{y \in
  W_{\mathbf{J}} / W_{\mathbf{K}}} \psi_{w_1 y}^{\mathbf{K}} (\mathbf{z} ; g),
\end{equation}
where we may choose the coset representatives $y$ so that $w_1 y \in
W^{\mathbf{K}}$. Thus if the theorem is true, we must have
\begin{equation}
  \label{zkjred} Z (\mathfrak{S}_{\mathbf{z}, \lambda, w_1, w_2}^{\mathbf{J}})
  = \sum_{y \in W_{\mathbf{J}} / W_{\mathbf{K}}} Z (\mathfrak{S}_{\mathbf{z},
  \lambda, w_1 y, w_2}^{\mathbf{K}}) .
\end{equation}
For the converse we have the following lemma.
\begin{lemma}
  \label{lem:jkind}Suppose that whenever $\mathbf{J}$ is nonempty, there
  exists a proper subset $\mathbf{K}$ of $\mathbf{J}$ such that \eqref{zkjred}
  is true. Then Theorem~\ref{thm:parahoric-lattice-model} is true.
\end{lemma}

\begin{proof}
  We may prove (\ref{eq:parpar}) by induction. For the base case $\mathbf{J} =
  \varnothing$, (\ref{eq:parpar}) is true by Theorem~\ref{coloredwhittaker}
  because $\psi^{\varnothing}_w = \phi_w$ and $\mathfrak{S}_{\mathbf{z},
  \lambda, w_1, w_2}^{\varnothing} =\mathfrak{S}_{\mathbf{z}, \lambda, w_1,
  w_2}$. Now assume inductively that (\ref{eq:parpar}) is true
  for proper subsets $\mathbf{K}$ of $\mathbf{J}$. By our assumption
  there exists such a $\mathbf{K}$ such that (\ref{zkjred}) is
  satisfied. Then (\ref{eq:parpar}) follows for $\mathbf{J}$ by
  combining (\ref{zkjred}), (\ref{kjreduction}) and (\ref{eq:parpar})
  for $\mathbf{K}$.
\end{proof}

In preparation for applying Lemma~\ref{lem:jkind} assume that $\mathbf{J}$
is nonempty, so $r_i > 1$ for some~$i$. Then, recalling that $W_{\mathbf{J}} = S_{r_1} \times
\cdots \times S_{r_k}$, the last simple reflection that is contained in
$S_{r_i}$ is $s_m$ where $m = r_1 + \ldots + r_i - 1$. 
Let $\mathbf{K}$ be obtained by removing $m$ from $\mathbf{J}$ so that $W_{\mathbf{K}} = S_{r_1}
\times \cdots \times S_{r_i - 1} \times S_1 \times \cdots \times S_{r_k}$ and let
the system $\mathfrak{S}=\mathfrak{S}_{\mathbf{z}, \lambda, w_1,
w_2}^{\mathbf{J}}$ be described by the semistandard flag $(c_1^{r_1}, c_2^{r_2}, \cdots, c_k^{r_k})$ as in \eqref{eq:semistandard-flag}.

Let $c = c_i$. We wish to insert a color between $c_i$ and $c_{i + 1}$.
Reparametrizing by Remark~\ref{rem:reparam} if necessary, we may assume that
there is a color $c'$ such that
\[ c_1 > c_2 > \cdots > c_i = c > c' > c_{i + 1} > \cdots > c_k . \]
We may use these colors to describe~$\mathfrak{S}_{\mathbf{z}, \lambda, w_1 y,
w_2}^{\mathbf{K}}$. Let
\begin{equation}
  \label{eq:Sprime-union}
  \mathfrak{S}' := \bigsqcup_{y \in W_{\mathbf{J}} / W_{\mathbf{K}}} \mathfrak{S}_{\mathbf{z}, \lambda, w_1 y, w_2}^{\mathbf{K}} .
\end{equation}
Note that each state in $\mathfrak{S}'$ has one line of color $c'$ that starts
at a fixed location in the top row, replacing the one of the top vertical
edges colored $c$ in $\mathfrak{S}_{\mathbf{z}, \lambda, w_1,
w_2}^{\mathbf{J}}$. Since $w_2\in W^{\mathbf{J}}$, the edge
containing the instance of the color $c$ that is replaced by $c'$ is the
rightmost top vertical edge. The $c'$ colored line ends up on the right edge,
replacing one of the horizontal vertical edges colored $c$. There are $r_i$
possible such locations and the decomposition of $\mathfrak{S}'$ into the
$r_i = |W_{\mathbf{J}} / W_{\mathbf{K}} |$ ensembles
$\mathfrak{S}_{\mathbf{z}, \lambda, w_1 y, w_2}^{\mathbf{K}}$ corresponds to
these $r_i$ possibilities.

We will now instead decompose $\mathfrak{S}'$ into two parts. We call a state $\mathfrak{s}' \in
\mathfrak{S}'$ \textit{strict} if it has no vertical edge carrying both colors
$c$ and $c'$. Let $\mathfrak{S}'_{\operatorname{strict}}$ consist of strict states in
$\mathfrak{S}'$, and let $\mathfrak{S}'_{\operatorname{ns}}$ be the remaining, nonstrict states.

\begin{lemma}
  \begin{samepage}
  \label{lem:parahoric-1-reduction}
  Let $\mathfrak{S}'_\mathrm{strict}$ and $\mathfrak{S}'_\mathrm{ns}$ be as above.
  Then, 
  \begin{enumerate}[label=\textnormal{(\roman*)}, leftmargin=*]
    \item \label{itm:parahoric-1-reduction-strict}
      $Z(\mathfrak{S}'_{\operatorname{strict}}) = Z(\mathfrak{S}_{\mathbf{z}, \lambda, w_1, w_2}^{\mathbf{J}})$,
    \item \label{itm:parahoric-1-reduction-ns}
      $Z(\mathfrak{S}_{\operatorname{ns}}') = 0$.
  \end{enumerate}
  \end{samepage}
\end{lemma}

Note that Lemma~\ref{lem:parahoric-1-reduction} implies \eqref{zkjred} for the above chosen subset $\mathbf{K}$ of $\mathbf{J}$, which together with Lemma~\ref{lem:jkind}, proves Theorem~\ref{thm:parahoric-lattice-model}.
We will prove Lemma~\ref{lem:parahoric-1-reduction} later in this section, but first we will need to introduce some terminology.

Let $\mathcal{L}$ be the set of vertices for our lattice models as first introduced in Section~\ref{coloredsystems}. We order
the vertices $\mathcal{L}$ lexicographically from left to right, top to
bottom. We will denote this total order on vertices $\prec$. Thus $x \prec y$
if and only if $x$ is in a row above $y$, or $x$ and $y$ are in the same row
and $x$ is to the left of $y$. An initial segment of $\mathcal{L}$ is either
the empty set $\varnothing$ or $I (x) = \{ y \in \mathcal{L} \mid y \preccurlyeq x
\}$ for some $x \in \mathcal{L}$. Let $\hat{\mathcal{L}}$ be the set of
initial segments of $\mathcal{L}$. If $I \in \hat{\mathcal{L}}$ is not
$\mathcal{L}$ then $I$ has a unique successor $\operatorname{succ} (I)$
in $\hat{\mathcal{L}}$, which is the unique initial segment of
cardinality $| I | + 1$.
If $x$ is a vertex of $\mathcal{L}$ we will call the edges above and left of
$x$ {\textit{inputs}}, and the edges below and to the right {\textit{outputs}}.

For a state $\mathfrak{s}$ in some system with lattice $\mathcal{L}$ and a vertex $x$ in $\mathcal{L}$ we denote by $\mathfrak{s}|_x$ the vertex configuration of $\mathfrak{s}$ at $x$, that is, the sequence of spins for the four adjacent edges.
Similarly, for any sequence $I$ of vertices in $\mathcal{L}$ we denote by
$\mathfrak{s}|_I$ the sequence of vertex configurations of $\mathfrak{s}$ for
vertices in $I$. We may extend our previous notation and denote by
$\beta(\mathfrak{s}|_I)$ the product of the Boltzmann weights for the
state $\mathfrak{s}$ at the vertices in~$I$.

Let $\Gamma$ be a finite directed graph with no cycles. We call $\Gamma$ a
{\textit{tree}} if it has a unique initial node, called the {\textit{root}}; the
terminal nodes are called {\textit{leaves}}. If $X, Y \in \Gamma$ we say that $Y$ is a
{\textit{child}} of $X$ if $X \rightarrow Y$ is an edge of $\Gamma$. Let $R$ be an
abelian group, which in our applications will be $\mathbb{C} (v)$. We call a
function $F : \Gamma \longrightarrow R$ {\textit{additive}} if for every non-leaf $X \in
\Gamma$ we have $F (X) = \sum F (Y)$ where the sum is over
the children of $X$ (but not over further descendants). If $\Gamma$ is a tree with root $X_{\operatorname{root}}$ and an additive function $F$
then clearly
\[ F (X_{\operatorname{root}}) = \sum_{\text{leaves $Y$}} F (Y) . \]

We will now define a tree and an additive function that we will use to prove Lemma~\ref{lem:parahoric-1-reduction}.
Let $\mathfrak{S}'$ be defined as in \eqref{eq:Sprime-union}.
If $I \in \hat{\mathcal{L}}$ define an equivalence relation on states in $\mathfrak{S}'$ where $\mathfrak{s}'_1 \equiv_I \mathfrak{s}'_2$ if $\mathfrak{s}'_1|_I = \mathfrak{s}'_2|_I$, that is, the two states have the same vertex configurations at each vertex in $I$. 
Let $[\mathfrak{s}']_I$ be the equivalence class of a state $\mathfrak{s}'$ under this relation.
We may define a directed graph $\Gamma$ whose nodes are pairs $(I, [\mathfrak{s}']_I)$ and with edges of the form $(I, [\mathfrak{s}'_1]_I) \rightarrow (J, [\mathfrak{s}'_2]_J)$ where $J = \operatorname{succ} (I)$ and $\mathfrak{s}'_2 \in [\mathfrak{s}'_1]_I$.
Note that the equivalence class $[\mathfrak{s}'_1]_I$ is a union of $\equiv_J$ equivalence classes.

We may enumerate the children of a parent node $(I,[\mathfrak{s}'_\text{parent}]_I)$ as follows.
Let $J = \operatorname{succ}(I)=I \cup \{ x \}$ where $x$ is the last vertex in $J$.
A child of $(I,[\mathfrak{s}'_\text{parent}]_I)$ has the form $(J,[\mathfrak{s}'_\text{child}]_J)$ where the class $[\mathfrak{s}'_\text{child}]_J$ is determined by the spins of the four edges adjacent to $x$ in a representative $\mathfrak{s}'_\text{child}$.
Moreover, the input spins for $\mathfrak{s}'_\text{child}$ at $x$ are determined by $\mathfrak{s}'_\text{parent}$ since the input edges at $x$ are either output edges for vertices in $I$ or boundary edges.
Thus, $[\mathfrak{s}'_\text{child}]_J$ is determined by the two output spins at $x$.
We see that the graph $\Gamma$ is a tree, and its branching at the node $(I, [\mathfrak{s}'_\text{parent}]_I)$ is determined by the different ways that the two inputs at the vertex $x$ can be completed with the output spins to an admissible configuration at the vertex $x$.
The root of the tree $\Gamma$ is $(\varnothing, \mathfrak{S}')$ where all states in $\mathfrak{S}'$ are equivalent under $\equiv_\varnothing$, and each leaf of the tree corresponds to an individual state in $\mathfrak{S}'$.

We define a function $F$ on $\Gamma$ as follows
\begin{equation}
  \label{eq:additive-function}
  F (I, [\mathfrak{s}']_I) = \sum_{\mathfrak{s} \in [\mathfrak{s}']_I} \beta(\mathfrak{s}),
\end{equation}
which is additive by the fact that $[\mathfrak{s}']_I$ is a union of its child $\equiv_{\operatorname{succ}(I)}$ equivalence classes.
The root value is the partition function, and the leaf values are the  Boltzmann weights of the individual states.

We will consider subgraphs of the tree $\Gamma$ to prove Lemma~\ref{lem:parahoric-1-reduction}.
These graphs are used to organize the application of the following two properties of the Boltzmann weights, which are explained in Figures~\ref{IandIIandIII} and~\ref{IVandV}. 

\begin{property}
  \label{prop:colorbwrel}
  We consider the situation where, for a state $\mathfrak{s}' \in \mathfrak{S}'$ and a vertex $x$ in $\mathcal{L}$, the vertex configuration $\mathfrak{s}'|_x$ has one input edge carrying the color $c$ while the other carries $c'$, but with no vertical edge carrying both colors $c$ and $c'$.
  In this case one
  output edge must carry the color $c$ and the other must carry $c'$. There are in total four configurations to consider but we group together the configurations that are mapped to each other under the interchange of $c$ and $c'$.
  We denote these groups as~$\text{II}$ and~$\text{III}$ shown in Figure~\ref{IandIIandIII}.
  Because of the allowed vertex configurations, only one of the two configurations within each group is possible for any given pair of input edges.
  If the color $c'$ is replaced by $c$, the four vertex configurations map to a single vertex configuration for $\mathfrak{S}$ that we denote by~$\text{I}$.
  The fact that we need is that the Boltzmann weights satisfy
  \begin{equation}
  \label{betadd}
  \beta_x (\text{II}) + \beta_x (\text{III}) = \beta_x (\text{I}),
  \end{equation}
  for any given input edges on the left-hand side of the equation.
  This is shown in Figure~\ref{IandIIandIII}.
  It is possible that both vertical edges carry other colors besides $c$ and $c'$ but if $c''$ is such a color, since $c$ and $c'$ are adjacent, we have either $c''>c,c'$ or $c,c'>c''$, in other words, other colors cannot distinguish between $c$ and $c''$.
  Using this, one may check that (\ref{betadd}) remains true even with these extra colors by the fusion definition of the Boltzmann weights.
\end{property}

\begin{figure}[htb]
  \[ \begin{array}{|c|c|c|}
       \hline
       \mathfrak{S}&\multicolumn{2}{c|}{\mathfrak{S}'}\\
       \hline
       \begin{tikzpicture}[scale=.75]
       \draw (0,1)--(0,-1);
       \draw (1,0)--(-1,0);
       \draw[fill=white] (0,1) circle (.35);
       \draw[fill=white] (0,-1) circle (.35);
       \draw[fill=white] (1,0) circle (.35);
       \draw[fill=white] (-1,0) circle (.35);
       \path[fill=white] (0,0) circle (.3);
       \node at (-1,1) {I.};
       \node at (0,1) {$c$};
       \node at (0,-1) {$c$};
       \node at (1,0) {$c$};
       \node at (-1,0) {$c$};
       \node at (0,0) {$z_i$};
       \end{tikzpicture}
       &
       \begin{tikzpicture}[scale=.75]
       \draw (0,1)--(0,-1);
       \draw (1,0)--(-1,0);
       \draw[fill=white] (0,1) circle (.35);
       \draw[fill=white] (0,-1) circle (.35);
       \draw[fill=white] (1,0) circle (.35);
       \draw[fill=white] (-1,0) circle (.35);
       \path[fill=white] (0,0) circle (.3);
       \node at (-1.5,1) {II.};
       \node at (-1,0) {$d$};
       \node at (0,1) {$d'$};
       \node at (1,0) {$d$};
       \node at (0,-1) {$d'$};
       \node at (0,0) {$z_i$};
       \end{tikzpicture}
       &
       \begin{tikzpicture}[scale=.75]
       \draw (0,1)--(0,-1);
       \draw (1,0)--(-1,0);
       \draw[fill=white] (0,1) circle (.35);
       \draw[fill=white] (0,-1) circle (.35);
       \draw[fill=white] (1,0) circle (.35);
       \draw[fill=white] (-1,0) circle (.35);
       \path[fill=white] (0,0) circle (.3);
       \node at (-1.5,1) {III.};
       \node at (-1,0) {$d$};
       \node at (0,1) {$d'$};
       \node at (1,0) {$d'$};
       \node at (0,-1) {$d$};
       \node at (0,0) {$z_i$};
       \end{tikzpicture}\\
       \hline
       z_i & \left\{\begin{array}{ll}
         z_i & \text{if $d > d'$}\\
         v z_i & \text{if $d' > d$}
       \end{array}\right. & \left\{\begin{array}{ll}
         0 & \text{if $d > d'$}\\
         (1 - v) z_i & \text{if $d' > d$}
       \end{array}\right.\\
       \hline
     \end{array} \]
\caption{Illustrating Property~\ref{prop:colorbwrel} in the case where the
vertex $x$ carries the colors $c$ and $c'$ and no others. Here $d$
and $d'$ are $c$ and $c'$ in some order.}
\label{IandIIandIII}
\end{figure}

\begin{property}
  \label{prop:colorcancel}
  We also consider the situation where the vertical edge above vertex $x$
  carries both $c$ and $c'$, but one color exits to the right, and the
  other to the bottom. There are two ways this can happen, as shown in
  Figure~\ref{IVandV}, and the Boltzmann weights of these patterns
  cancel:
  \begin{equation}
  \label{eq:colcan}
  \beta_x (\text{IV}) + \beta_x (\text{V}) = 0.
  \end{equation}
  As for Property~\ref{prop:colorbwrel}, this relation also remains true even if the vertical edges carry one or more additional colors.
\end{property}

\begin{figure}[htb]
  \[ \begin{array}{|c|c|}
       \hline
       \multicolumn{2}{|c|}{\mathfrak{S}'}\\
       \hline
       \begin{tikzpicture}[scale=.75]
       \draw (0,-1)--(0,0);
       \draw (-.1,0)--(-.1,1);
       \draw (.1,0)--(.1,1);
       \draw (1,0)--(-1,0);
       \draw[fill=white] (0,1) circle (.35);
       \draw[fill=white] (0,-1) circle (.35);
       \draw[fill=white] (1,0) circle (.35);
       \draw[fill=white] (-1,0) circle (.35);
       \path[fill=white] (0,0) circle (.3);
       \node at (-1.5,1) {IV.};
       \node at (0,1) {$\scriptstyle c,c'$};
       \node at (1,0) {$c$};
       \node at (0,-1) {$c'$};
       \node at (0,0) {$z_i$};
       \end{tikzpicture}
       &
       \begin{tikzpicture}[scale=.75]
       \draw (0,-1)--(0,0);
       \draw (-.1,0)--(-.1,1);
       \draw (.1,0)--(.1,1);
       \draw (1,0)--(-1,0);
       \draw[fill=white] (0,1) circle (.35);
       \draw[fill=white] (0,-1) circle (.35);
       \draw[fill=white] (1,0) circle (.35);
       \draw[fill=white] (-1,0) circle (.35);
       \path[fill=white] (0,0) circle (.3);
       \node at (-1.5,1) {V.};
       \node at (0,1) {$\scriptstyle c,c'$};
       \node at (1,0) {$c'$};
       \node at (0,-1) {$c$};
       \node at (0,0) {$z_i$};
       \end{tikzpicture}\\
       \hline
       -v & v\\
       \hline
     \end{array} \]
     \caption{Illustrating Property~\ref{prop:colorcancel}, 
the Boltzmann weights of these two nonstrict configurations
are negatives of each other.}
\label{IVandV}
\end{figure}
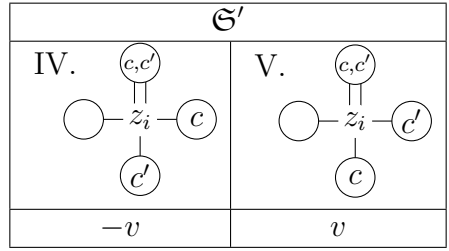

If $\mathfrak{s}'$ is a strict state of the system $\mathfrak{S}'$ then there
is a corresponding state $\pi (\mathfrak{s}')$ of the system $\mathfrak{S}$ in
which every instance of the color $c'$ is replaced by $c$. For each
$\mathfrak{s} \in \mathfrak{S}$ there exists at least one
$\mathfrak{s}' \in \mathfrak{S}'_\text{strict}$ such that $\pi(\mathfrak{s}')
= \mathfrak{s}$ and can be obtained by coloring one $c$-path in the color
$c'$.

For $I \in \hat{\mathcal{L}}$, let $\bar I$ be the complement of $I$ in the
set of vertices for the lattice $\mathcal{L}$.  We say that a state
$\mathfrak{s}'$ is \emph{$\bar I$-strict} if for all $x \in \bar I$ no
vertical edge of $\mathfrak{s}'|_x$ carries both $c$ and $c'$.
That is, if the vertex configurations $\mathfrak{s}'|_{\bar I}$ are strict.
In this case (generalizing the above notation) we may define
$\pi(\mathfrak{s}'|_{\bar I})$ to be the $\bar{I}$-sequence of vertex configurations obtained from the sequence $\mathfrak{s}'|_{\bar I}$ by replacing every instance of $c'$
by $c$. 
Note that these vertex configurations are admissible only when $\mathfrak{s}'$ is $\bar I$-strict.

\begin{lemma}
  \label{lem:additivity}
  Let $\mathfrak{s}'_0 \in \mathfrak{S}'$ and $X_0 = (I_0, [\mathfrak{s}'_0]_{I_0})$ be a node of $\Gamma$ such that $\mathfrak{s}'_0$ is $\bar I_0$-strict.
  Consider the subtree $\Gamma_{X_0}$ of $\Gamma$ with its root at $X_0$ together with all its descendants.
  Let $\Gamma_{X_0}^{(\mathfrak{s}'_0)}$ be the tree obtained from $\Gamma_{X_0}$ by selecting the branches whose leaves $\mathfrak{s}' \in \mathfrak{S}'$ are $\bar I_0$-strict and such that $\pi(\mathfrak{s}'|_{\bar I_0}) = \pi(\mathfrak{s}'_0|_{\bar I_0})$.
  Then
  \begin{equation}
    \label{eq:Fs}
    F_{\mathfrak{s}'_0}(I, [\mathfrak{s}']_I) := \beta(\mathfrak{s}'|_I) \beta(\pi(\mathfrak{s}'_0|_{\bar I}))
  \end{equation}
  is additive on $\Gamma_{X_0}^{(\mathfrak{s}'_0)}$.
\end{lemma}

Note that the representative $\mathfrak{s}' \in \mathfrak{S}'$ in \eqref{eq:Fs} need not be a leaf of $\Gamma_{X_0}^{(\mathfrak{s}'_0)}$; $F_{\mathfrak{s}'_0}$ is still well-defined and independent of the representative in $[\mathfrak{s}']_I$.

\begin{proof}
  Let $X = (I, [\mathfrak{s}'_\text{parent}]_I)$ be a node of
  $\Gamma_{X_0}^{(\mathfrak{s}'_0)}$ and let
  $Y = (J, [\mathfrak{s}'_\text{child}]_J)$ be a child of $X$, where
  $J = \operatorname{succ}(I)=I \cup \{x\}$.
  Denote the set of leaves of $\Gamma_{X_0}^{(\mathfrak{s}'_0)}$ by $\Xi$ and choose the representatives $\mathfrak{s}'_\text{parent}$ and $\mathfrak{s}'_\text{child}$ such that they are in $\Xi$.
  That is, $\mathfrak{s}'_\text{parent}$ and $\mathfrak{s}'_\text{child}$ are $\bar I_0$-strict and $\pi(\mathfrak{s}'_\text{parent}|_{\bar I_0}) = \pi(\mathfrak{s}'_\text{child}|_{\bar I_0}) = \pi(\mathfrak{s}'_0|_{\bar I_0})$.
  
  The fact that $\mathfrak{s}'_\text{child}$ is a child of $\mathfrak{s}'_\text{parent}$ means that
  $\mathfrak{s}'_\text{parent}|_I = \mathfrak{s}'_\text{child}|_I$.
  Furthermore $\mathfrak{s}'_\text{parent}$ and $\mathfrak{s}'_\text{child}$ agree on
  the two input edges of $x$, but they may differ on the two output edges of~$x$.
  Since $\pi(\mathfrak{s}'_\text{parent}|_x) = \pi(\mathfrak{s}'_\text{child}|_x) = \pi(\mathfrak{s}'_0|_x)$, this can happen only if, for both states, one of the input edges carries the color $c$ and the other carries $c'$ while the output edges for the two states have $c$ and $c'$ interchanged. 
  Note that since $\mathfrak{s}'_\text{child}$ is $\bar I_0$-strict, no edge of $x$ carries both $c$ and $c'$.
  If $\mathfrak{s}'_\text{parent}$ and $\mathfrak{s}'_\text{child}$ do
  agree on the two output edges of $x$, then they are $J$-equivalent, so the
  node $(I, [\mathfrak{s}'_\text{parent}]_I)$ has at most two children. 
  
  Thus, $\Gamma_{X_0}^{(\mathfrak{s}'_0)}$ is a binary tree and to prove that $F_{\mathfrak{s}'_0}$ is additive we need to check two cases: $X$ has one or two children.

  Assume first that $X = (I, [\mathfrak{s}'_\text{parent}]_I)$ has two children $Y_1 = (J, [\mathfrak{s}'_1]_J)$ and $Y_2 = (J,[\mathfrak{s}'_2]_J)$ where we have chosen the representatives $\mathfrak{s}'_1$ and $\mathfrak{s}'_2$ in $\Xi$.
  Then,
  \begin{equation}
    \label{eq:Fs-additivity-proof}
    \begin{split}
      F_{\mathfrak{s}'_0}(X) &= \beta(\mathfrak{s}'_\text{parent}|_I) \beta(\pi(\mathfrak{s}'_0|_x)) \beta(\pi(\mathfrak{s}'_0|_{\bar J}))
      \\
      F_{\mathfrak{s}'_0}(Y_i) &= \beta(\mathfrak{s}'_i|_I) \beta(\mathfrak{s}'_i|_x) \beta(\pi(\mathfrak{s}'_0|_{\bar J})) = \beta(\mathfrak{s}'_\text{parent}|_I) \beta(\mathfrak{s}'_i|_x) \beta(\pi(\mathfrak{s}'_0|_{\bar J})) .
    \end{split}
  \end{equation}

  As argued above, $\pi(\mathfrak{s}'_0|_x)$, $\mathfrak{s}'_1|_x$ and $\mathfrak{s}'_2|_x$ are in the situation of Property~\ref{prop:colorbwrel} meaning that $\beta(\pi(\mathfrak{s}'_0|_x)) = \beta(\mathfrak{s}'_1|_x) + \beta(\mathfrak{s}'_2|_x)$ and thus $F(X) = F(Y_1) + F(Y_2)$.
  
  For the case where $X = (I,[\mathfrak{s}'_\text{parent}]_I)$ has a single child $Y_1 = (J, [\mathfrak{s}'_1]_J)$ equation~\eqref{eq:Fs-additivity-proof} is still valid and the vertex configuration $\mathfrak{s}'_1|_x$ is uniquely determined by~$[\mathfrak{s}'_\text{parent}]_I$ and~$\mathfrak{s}'_0$.
  Since $\mathfrak{s}'_0$ is $\bar I_0$-strict this means that the vertex configuration does not contain both $c$ and $c'$.
  Because other colors cannot differentiate between $c$ and $c'$ we then have that $\beta(\mathfrak{s}'_1|_x) = \beta(\mathfrak{s}'_\text{parent}|_x) = \beta(\pi(\mathfrak{s}'_0|_x))$, which concludes the proof.
\end{proof}

If we apply Lemma~\ref{lem:additivity} to the full tree $\Gamma_{X_0} = \Gamma$  we get the following corollary.
\begin{corollary}
  \label{cor:beta-harmonic}
  If $\mathfrak{s}$ is a state of $\mathfrak{S}$ then
  \begin{equation}
    \label{eq:betaharmonic} \beta (\mathfrak{s}) =
    \sum_{\substack{
      \mathfrak{s}' \in \mathfrak{S}_{\operatorname{strict}}'\\
      \pi (\mathfrak{s}')=\mathfrak{s}}}
      \beta (\mathfrak{s}') .
  \end{equation}
\end{corollary}
\begin{proof}
  Let $X_0$ be the root $X_\text{root} = (\varnothing, \mathfrak{S}')$ of $\Gamma$.
  For each $\mathfrak{s} \in \mathfrak{S}$ there exists $\mathfrak{s}'_0 \in \mathfrak{S}'_\text{strict}$ such that $\pi(\mathfrak{s}'_0) = \mathfrak{s}$.
  Then $\Gamma_{X_0}^{(\mathfrak{s}'_0)}$ contains all branches of $\Gamma$ for which the leaves $\mathfrak{s}'$ are strict and map to $\mathfrak{s}$ under $\pi$ independent of the choice of $\mathfrak{s}'_0$.
  The value of $F_{\mathfrak{s}'_0}$ at the root $X_\text{root}$ equals $\beta(\mathfrak{s})$ while at a leaf $\mathfrak{s}'$ it equals $\beta(\mathfrak{s}')$ also independent of $\mathfrak{s}'_0$, and the statement follows from the additivity of $F_{\mathfrak{s}'_0}$ on $\Gamma_{X_0}^{(\mathfrak{s}'_0)}$.
\end{proof}

\begin{proof}[Proof of Lemma~\ref{lem:parahoric-1-reduction} and Theorem~\ref{thm:parahoric-lattice-model}]
  We have already noted that Lemma~\ref{lem:parahoric-1-reduction} implies
  Theorem~\ref{thm:parahoric-lattice-model}, so we turn to the proof of the Lemma.
  Statement~\ref{itm:parahoric-1-reduction-strict} follows from Corollary~\ref{cor:beta-harmonic} by summing over $\mathfrak{s} \in \mathfrak{S}$.
  
  For statement~\ref{itm:parahoric-1-reduction-ns} we will start with the tree $\Gamma$.
  Let $I$ be a maximal initial segment such that there is a nonstrict state $\mathfrak{s}'_1$ that is $\bar I$-strict.
  This means that all states in $[\mathfrak{s}'_1]_I$ are strict at the vertices in $\bar I$, but if $x$ is the last vertex in $I$, then the colors $c$ and $c'$ of $\mathfrak{s}'_1|_x$ are in one of the two configurations in Figure~\ref{IVandV} of Property~\ref{prop:colorcancel} (disregarding other colors).
  The node $X_1 = (I, [\mathfrak{s}'_1]_I)$ has a single sibling $X_2 = (I, [\mathfrak{s}'_2]_I)$ for which $\mathfrak{s}'_2|_x$ is in the other configuration of Figure~\ref{IVandV}.
  Indeed, we can construct a representative $\mathfrak{s}'_2$ by starting from $\mathfrak{s}'_1$, apply $\pi$ on $\mathfrak{s}'_1|_{\bar I}$ and then apply the color $c'$ to one of the $c$-paths going in the other direction at $x$.
  By construction $\pi(\mathfrak{s}'_1|_{\bar I}) = \pi(\mathfrak{s}'_2|_{\bar I})$ while $\mathfrak{s}'_1|_I$ and $\mathfrak{s}'_2|_I$ only differ at the vertex $x$ with an overall minus sign for their Boltzmann weights.
  Let $\Xi_1$ and $\Xi_2$ be the sets of leaves for $\Gamma_{X_1}^{(\mathfrak{s}'_1)}$ and $\Gamma_{X_2}^{(\mathfrak{s}'_2)}$ respectively.
  By Lemma~\ref{lem:additivity},
  \begin{multline}
    \label{eq:tree-cancellation}
    \sum_{\mathfrak{s}' \in \Xi_1} \beta(\mathfrak{s}') = F_{\mathfrak{s}'_1}(I, [\mathfrak{s}'_1]_I)
    = \beta(\mathfrak{s}'_1|_I) \beta(\pi(\mathfrak{s}'_1|_{\bar I})) = \\[-0.5em]
    = -\beta(\mathfrak{s}'_2|_I) \beta(\pi(\mathfrak{s}'_2|_{\bar I})) 
    = -F_{\mathfrak{s}'_2}(I, [\mathfrak{s}'_2]_I)
    = -\sum_{\mathfrak{s}' \in \Xi_2} \beta(\mathfrak{s}')
  \end{multline}
  Note that the freedom in constructing $\mathfrak{s}'_2$ above is given by $\Xi_2$.
  Any other choice of the pair $(\mathfrak{s}'_1, \mathfrak{s}'_2)$ in $\Xi_1 \times \Xi_2$ would give the same trees $\Gamma_{X_1}^{(\mathfrak{s}'_1)}$ and $\Gamma_{X_2}^{(\mathfrak{s}'_2)}$ with the sames sets of leaves, as well as the same equation~\eqref{eq:tree-cancellation} for these leaves.
  We may thus choose pairs of representatives in $[\mathfrak{s}'_1]_I \times [\mathfrak{s}'_2]_I$ such that the corresponding sets $\Xi_1 \times \Xi_2$ are disjoint and their union equals $\Xi(\Gamma_{X_1}) \times \Xi(\Gamma_{X_2})$ where $\Xi(\Gamma_{X_i})$ is the set of leaves for the tree $\Gamma_{X_i}$.

  Hence, we can remove the nodes $X_1$ and $X_2$ together with their descendants from the tree $\Gamma$ without affecting the values and additivity property of the function $F$ on the remaining nodes of $\Gamma$.
  That is, we may remove these nodes without changing the partition function of the leaves of $\Gamma$.
  Repeating the process with a new maximal initial segment and $I$ and siblings $X_1$ and $X_2$ of the remaining $\Gamma$ we have thus shown that we may remove all branches with non-strict states as leaves without changing the partition function.
  Together with statement~\ref{itm:parahoric-1-reduction-strict} that was shown above, this proves statement~\ref{itm:parahoric-1-reduction-ns}.
\end{proof}

The case of the maximal compact subgroup $K = K_{\mathbf{I}}$ is a special
case of the parahoric system where $\mathbf{J} = \mathbf{I}$ is the set of all
simple reflections. As previously mentioned, the Boltzmann weights for the Tokuyama model, shown for example in the row labeled $S_\Gamma(i)$ of~\cite[Table~2]{hkice},
are a special case of the weights in
Figures~\ref{cvmonochrome} (where $\minus$ is replaced by one
fixed color $c$). The partition function with these Boltzmann weights depends only on the partition $\lambda$ and is given by the
formula
\begin{equation}
\label{tokform}
Z (\mathbf{z}; \lambda) =\mathbf{z}^{\rho} \prod_{\alpha \in
     \Delta^+} (1 - v\mathbf{z}^{- \alpha}) s_{\lambda} (\mathbf{z})
\end{equation}
which is one version of Tokuyama's formula~\cite{Tokuyama,HamelKing,hkice}. It also agrees with the Casselman-Shalika formula for the spherical
Whittaker function.

See arXiv Version~1 of this paper (\cite{ThisPaper}) for further discussion
of the relationship between the uncolored Tokuyama models and the colored
models of this paper.

\section{Whittaker functions and Macdonald polynomials}
\label{sec:nsmac}

The purpose of this section is to give a dictionary between values of Whittaker functions on $\GL_r(F)$ and certain specializations and generalizations of Macdonald polynomials, as introduced in Table~\ref{tab:mini-nsmac} and further detailed here in Table~\ref{tab:nsmac}.
The second and last of these identities are new to our knowledge and will be proved later in this section.
Both sides of the dictionary can be studied either algebraically, usually involving some variations of Demazure-Lusztig operators, or combinatorially which, for the Whittaker functions, can be achieved via the theory of solvable lattice models.

\begin{table}[htpb]
  \centering
  \caption{Relations between different Whittaker functions and associated special polynomials.}
  \label{tab:nsmac}
  \small
  \begin{tabular}{lcl}
    \toprule
    \thead{Whittaker function} &  & \thead{Special polynomial} \\
    \midrule
    \\[-0.5em]
    Spherical Whittaker function & & Schur polynomial \\ 
    $\sum_{w \in W} \phi_w(\z; \varpi^{-\lambda})$ &$=$ & $\prod_{\alpha \in \Delta^+} (1- v\z^{-\alpha}) s_\lambda(\z)$ \\[1.5em]
    Li's Whittaker function & & Hall-Littlewood polynomial \\
    $\sum_{w \in W} (- v)^{- \ell (w)} \phi_w (\mathbf{z};
    \varpi^{- \lambda})$ & $=$ & $\mathbf{z}^{-\rho}P_{\lambda + \rho} (\mathbf{z}, v^{- 1})$ \\[1.5em]
    Iwahori Whittaker function & & \makecell[lb]{Non-symmetric Macdonald polynomial} \\ 
    $\phi_{w_1}(\z; \varpi^{-\lambda})$ &$=$ & $(- v)^{\ell (w)}\mathbf{z}^{- \rho} w_0 E_{w_0 w (\lambda + \rho)} (\mathbf{z}; \infty,v)$ \\[1.5em]
    Parahoric Whittaker function & & \makecell[lb]{Macdonald polynomial with prescribed symmetry}  \\
    $\psi^{\mathbf J}_1(\z; \varpi^{-\lambda})$ & $=$ & $\z^{-\rho} S_{\lambda+\rho}^{(\emptyset,\J)} (\z;0,v^{-1}) a_{\lambda+\rho}^{(\emptyset,\J)}$ \\[1.5em]
    \bottomrule
  \end{tabular}
\end{table}

We start with the well-known case of the spherical Whittaker function, which can be expressed as the product of a quantized Weyl denominator and a Schur polynomial. This result is due to Shintani~\cite{Shintani} and was generalized to all quasi-split reductive groups by Casselman-Shalika~\cite{CasselmanShalika}. It may be proved by studying the more refined Iwahori fixed vectors. One can use the Yang-Baxter equation to give a lattice model interpretation of the spherical function; this was done in~\cite{hkice} based on ideas of Tokuyama~\cite{Tokuyama}. 

Parallel to this work, Li~\cite{JianShuLiUnramified} studied certain Iwahori fixed vectors in the unramified principal series and their associated Whittaker functions which can be used to identify the unique genuine subquotient of $I(\z)$. Li computed a variation of the Casselman-Shalika formula for these Whittaker functions each of which we will express in terms of a Hall-Littlewood polynomial in Proposition~\ref{LiWhittakerHL}. A (bosonic) lattice model called the $q$-boson model exists for Hall-Littlewood polynomials (see~\cite{Tsilevich}); it has successfully been used to study Hall-Littlewood polynomials in both combinatorics~\cite{WheelerZinn-JustinAIM} and representation theory~\cite{KorffVerlinde}. By our results we can then associate this lattice model to Li's Whittaker function.  

To understand both examples presented above, we need to understand the passage to the Iwahori level. As documented earlier, these Whittaker functions 
may be described in terms of certain divided difference operators.
The definition of the Demazure-Lusztig and Demazure-Whittaker operators
$\mathfrak{L}_i$ and $\mathfrak{T}_i$ associated to simple reflections $s_i$ in the Weyl group $W$ were given in (\ref{demazurelusztig})
and (\ref{demwhitop}), respectively. They also arise naturally in certain induced representations of the affine Hecke algebra
made from the trivial and sign characters of the finite Hecke
algebra, respectively (see~\cite{BBL, BrubakerBumpFriedbergMatrix}). 

Given any $w \in W$ and a reduced expression $w = s_{i_1} \cdots s_{i_k}$,
set
$ \mathfrak{T}_w =\mathfrak{T}_{i_1} \cdots \mathfrak{T}_{i_k}$, which is well-defined because the $\mathfrak{T}_i$ satisfy the braid relations. We will similarly write
$\mathfrak{L}_w =\mathfrak{L}_{i_1} \cdots \mathfrak{L}_{i_k}$. Then the following relation holds between these operators, where we add $v$-dependence to the notation for
$\mathfrak{L}_w$ and $\mathfrak{T}_w$.

\begin{proposition}
  \label{tlrelation}
  For any $w \in W$, as operators on the ring $\mathcal{O}(\hat T)$ of regular (polynomial) functions on
  $\hat{T}(\mathbb{C}) \iso (\CC^{\times})^r$,
  \begin{equation}
    \label{eq:L-T} \mathfrak{L}_{w, v} = (- v)^{\ell
    (w)} \mathbf{z}^{\rho} \mathfrak{T}_{w, v^{- 1}} \mathbf{z}^{-
    \rho} .
  \end{equation}
\end{proposition}

\begin{proof}
  Using~\eqref{demwhitop} and~\eqref{demazurelusztig} it follows that $- v\mathfrak{L}_{i, v^{- 1}} =\mathbf{z}^{\rho}
  \mathfrak{T}_{i, v} \mathbf{z}^{- \rho}$ which implies (\ref{eq:L-T}).
\end{proof}

In order to understand the spherical Whittaker functions (which are sums of standard Iwahori Whittaker functions), we must compute how the corresponding 
spherical idempotents in the Hecke
algebra act on a dominant weight. Define
\[ \Theta = \mathbf{z}^{- \rho} \prod_{\alpha \in \Delta^+} (1 -
    \mathbf{z}^{- \alpha})^{- 1} \left( \sum_{w \in W} (-1)^{\ell (w)}
   w \right) \mathbf{z}^{\rho} \hspace{0.17em} . \] 
\begin{proposition}
  \label{TLspherical}
  The following identities of operators hold in $\mathcal{O}(\hat T)$:
  \[ \sum_{w \in W} \mathfrak{\mathfrak{T}}_w = \left( \prod_{\alpha \in
     \Delta^+} (1 - v \mathbf{z}^{- \alpha}) \right) \Theta, \hspace{2em}
     \sum_{w \in W} \mathfrak{L}_w = \Theta \prod_{\alpha \in \Delta^+} (1 - v
     \mathbf{z}^{- \alpha})\,. \]
\end{proposition}

\begin{proof}
  The operator (9) of \cite{BrubakerBumpFriedbergMatrix} becomes our $\mathfrak{L}_{i, v}$ under the specialization $\pi^{\lambda} \mapsto \mathbf{z}^{- \lambda}$ and
  $\epsilon (T_i) = q$, and taking $q$ to be our $v$.
  Therefore Theorem~14 of~\cite{BrubakerBumpFriedbergMatrix} gives both formulas.
\end{proof}

By the Weyl character formula if $\lambda$ is a partition then $\Theta
\mathbf{z}^{\lambda} = s_{\lambda} (\mathbf{z})$ is the corresponding Schur function. Thus the sum over $\mathfrak{T}_w$ produces a Schur function
times a deformed Weyl denominator matching the Casselman-Shalika formula for the spherical Whittaker function, as expected.
The sum over $\mathfrak{L}_w$ produces a (symmetric) Hall-Littlewood polynomial. While Proposition~\ref{tlrelation} gives a relationship between $\mathfrak{L}_w\z^\lambda$ and $\mathfrak{T}_w\z^\lambda$ for any $w$, it is remarkable that their sums over all $w \in W$ result in such different functions.

Next we explain precisely how these operators $\mathfrak{T}_i$ and $\mathfrak{L}_i$ may be used to compute
various specializations of non-symmetric Macdonald polynomials. In general, these polynomials depend on two parameters
which are usually denoted $q$ and $t$. There are differing
notations in the literature, but in this paper we will
follow the notation $E_\lambda(\z; q, t)$ of Haglund, Haiman and Loehr~\cite{HaglundHaimanLoehrCombinatorial}. Note that this $q$ is not the cardinality of the residue field of $F$ as in earlier sections; in this section, the notation $v^{-1}$ will be used for this quantity, where the $t$ of \cite{HaglundHaimanLoehrCombinatorial} equals our $v$.

If $q=0$ or $\infty$ the polynomials $E_\lambda(\z; q, t)$ are non-symmetric variants of Hall-Littlewood
polynomials. According to \cite{BBL}, one such specialization arises from successively applying the operators
$\mathfrak{T}_i$ to a dominant weight $\lambda$.
The notation of Haglund, Haiman and
Loehr~\cite{HaglundHaimanLoehrCombinatorial} that we follow here differs from the notation
in~\cite{BBL} by the variable change $(q, t) \mapsto (q^{- 1}, t^{- 1})$, so
Theorem~7 of~{\cite{BBL}} will now be written
\begin{equation}
  \label{eq:bblnsm} \phi_w (\mathbf{z}; \varpi^{-\lambda}) = \mathfrak{T}_{w,v}(\z^\lambda)=
(- v)^{\ell (w)}
  \mathbf{z}^{- \rho} w_0 E_{w_0 w (\lambda + \rho)} (\mathbf{z}; \infty,
  v) .
\end{equation}
The next result is an analog of this for the $\mathfrak{L}_{w,v}$ using Proposition~\ref{tlrelation}.

\begin{proposition}
  \label{lnsmaceval}
  If $\lambda$ is dominant, then
  \begin{equation}
    \label{eq:L-E}
    \mathfrak{L}_{w, v} (\mathbf{z}^{\lambda + \rho}) = w_0 E_{w_0 w
     (\lambda + \rho)} (\mathbf{z}; \infty, v^{- 1}) . 
  \end{equation} 
\end{proposition}
\begin{proof}
  This follows by comparing \eqref{eq:L-T} and \eqref{eq:bblnsm}. Another
  proof may be based on the Knop-Sahi recurrence
  (\cite{KnopIntegrality,SahiInterpolation,CherednikIntertwining}) and other
  facts that can be found in~\cite{HaglundHaimanLoehrCombinatorial}. 
   For brevity we will not give this alternative proof.
\end{proof}

Having made this connection, we can now relate Li's Whittaker functions to the
Hall-Littlewood polynomials $P_\lambda$ (\cite{MacdonaldHall}, Chapter~III). If we
denote by $\mathcal{W}_\text{Li}(\z,\varpi^{- \lambda})$ the Whittaker
function described by Jian-Shu Li in~{\cite{JianShuLiUnramified}}, we have the
following result:

\begin{proposition}\label{LiWhittakerHL}
  Let $\lambda$ be a dominant weight. Then
  \[ \mathbf{z}^{\rho} \mathcal{W}_\text{\rm Li}(\z,\varpi^{- \lambda}) := \mathbf{z}^{\rho} \sum_{w \in W} (- v)^{- \ell (w)} \phi_w (\mathbf{z};
     \varpi^{- \lambda}) = P_{\lambda + \rho} (\mathbf{z}, v^{- 1}) . \]
\end{proposition}

\begin{proof}
By (1.1) in Chapter~III of Macdonald~{\cite{MacdonaldHall}}, 
if $\lambda$ is a partition then
\begin{equation}
\label{lishlp}\sum_{w \in W} \mathfrak{L}_w \mathbf{z}^{\lambda} = R_{\lambda}
   (\mathbf{z} ; v) = v_{\lambda} (v) P_{\lambda} (\mathbf{z} ; v)
\end{equation}
where $R_{\lambda}$, $v_{\lambda}$ and $P_{\lambda}$ are as
in~{\cite{MacdonaldHall}} Section~III.1. Multiplying (\ref{eq:bblnsm}) by $\mathbf{z}^{\rho} (- v)^{- \ell (w)}$
  and summing over $w$, then using (\ref{eq:L-E}) and (\ref{lishlp}), the
  left-hand side of the desired identity equals
  \[ \sum_{w \in W} w_0 E_{w_0 w (\lambda + \rho)} (\mathbf{z}; \infty, v) = \sum_{w \in W}
     \mathfrak{L}_{w, v^{- 1}} (\mathbf{z}^{\lambda + \rho}) = v_{\lambda +
     \rho} (v^{- 1}) P_{\lambda + \rho} (\mathbf{z} ; v^{- 1}) . \]
  Because $\lambda + \rho$ is strongly dominant $v_{\lambda + \rho} = 1$ and
  the statement follows.
\end{proof}

To conclude this section, we relate parahoric Whittaker functions to generalizations of Macdonald polynomials called \emph{Macdonald polynomials with prescribed symmetry}. They were introduced by Baker, Dunkl and Forrester~\cite{BakerDunklForrester} and studied further by Marshall~\cite{Marshall} and Baratta~\cite{Baratta}. We shall follow the conventions in~\cite{Baratta}.

For $I, J$ disjoint subsets of $\{1,\cdots,r-1 \}$ such that $i\pm 1 \notin J$ for $i \in I$ and $j \pm 1 \notin I$ for $j \in J$, we define
\begin{equation}
  \label{eq:nsmac-prescribed}
  S_{\eta^*}^{(I,J)}(\z;q,t) := \frac{1}{a_\eta^{(I,J)}} \sum_{w \in W_{I \cup J}} (-t)^{-\ell(w)} \mathfrak{L}_{w,t} E_\eta(\z;q,t),
\end{equation}
where $\eta^*$ is a composition such that $\eta^*_i \geq \eta^*_{i+1}$ for all $i \in I$ and $\eta^*_j > \eta^*_{j+1}$ for all $j \in J$, $\eta$ is any composition in $W_{I\cup J} \eta^*$, and $a_\eta^{(I,J)}$ is an explicit normalization factor.

\begin{proposition}
  Let $\lambda$ be a dominant weight, $I = \emptyset$, $J = \J$ and $t = v^{-1}$. Then 
  \[ \psi_1^\J(\lambda, \z) =  \z^{-\rho} S_{\lambda+\rho}^{(\emptyset,\J)} (\z;0,t) a_{\lambda+\rho}^{(\emptyset,\J)}. \]
\end{proposition}
\begin{proof}
  From~\eqref{eq:bblnsm} and~\eqref{eq:L-E} with $t = v^{-1}$ we have that
  \begin{equation}
    \phi_w(\z; \varpi^{-\lambda}) = (-t)^{-\ell(w)} \z^{-\rho} \mathfrak{L}_{w,t}(\z^{\lambda+\rho}).
  \end{equation}
  Since $E_{\lambda+\rho}(\z; 0, t) = \z^{\lambda+\rho}$ for dominant weights $\lambda$, the result follows from~\eqref{eq:psidef} and~\eqref{eq:nsmac-prescribed}.
\end{proof}

The following corollary follows immediately from Theorem~\ref{thm:paracs}:

\begin{corollary} 
  Let $\lambda$ be a dominant weight and $t = v^{-1}$. Then
\[ S_{\lambda+\rho}^{(\emptyset,\J)} (\z;0,t) = \frac{\z^\rho}{a_{\lambda+\rho}^{(\emptyset,\J)}}  \prod_{\alpha \in \Delta^+_{\mathbf{J}}}
     (1 - v\mathbf{z}^{- \alpha}) \chi_{\lambda}^{\mathbf{J}} (\mathbf{z}). \]
\end{corollary}

At this point in time, we have a good understanding of the dictionary relating Whittaker functions to special polynomials when the group we are working with is $\GL_r$. It would be interesting for both combinatorial and number theoretic reasons to understanding generalizations of this dictionary to the metaplectic cover of $\GL_r$ and to other reductive groups. In the metaplectic setting, earlier results suggest these questions merit further inquiry: in~\cite{BBBGu}, metaplectic spherical Whittaker functions are related to supersymmetric LLT polynomials, while in~\cite{SahiStokmanVenkateswaran} a new family of special polynomials is introduced that generalizes metaplectic Iwahori Whittaker functions and non-symmetric Macdonald polynomials. 


\section{Intertwining integrals and \texorpdfstring{$R$}{R}-matrices\label{sec:intertwiningR}}

In this section, we will explore the dictionary between $p$-adic representation theory
and R-matrices of quantum groups, using lattice models for a pictorial interpretation of either
side.  
Roughly speaking, we will show that parts of the R-matrix for the quantum superalgebra $U_q(\widehat{\mathfrak{gl}}(r|1))$ 
corresponding to the smaller quantum groups $U_q(\widehat{\mathfrak{gl}}(r))$ and $U_q(\widehat{\mathfrak{gl}}(1))$ neatly
express the action of intertwining operators on standard Iwahori fixed vectors (Theorem~\ref{RmatrixintertwiningIwahori}) and on Whittaker functionals (Remark~\ref{intWRmatrix}), respectively.  
To prove Theorem~\ref{RmatrixintertwiningIwahori}, we identify the R-matrix
for $U_q(\widehat{\mathfrak{gl}}(r))$ and the intertwining integral acting on
standard Iwahori fixed vectors with a part of the colored R-matrix in
Figure~\ref{coloredrmat}. This allows us to give a pictorial interpretation of
the functional equations used to prove Theorem~\ref{mainrecursion} (see
equations~\eqref{pictorial},~\eqref{pictorial2}).

Before proving these facts, we will first make several comments related to Theorem~\ref{RmatrixintertwiningIwahori}. 
A common principle in the theory of symmetric
functions (related to Schur duality) is to consider the coefficient of $z_1 z_2 \cdots z_r$
in the $r$ variables $z_i$ as having some combinatorial significance.
Applying this to Schur functions gives the representation degrees of the
irreducible representations of the symmetric group, and this principle was
also used by Stanley~\cite{StanleyReduced} in counting the number of reduced words for
the longest element of~$S_r$.

A somewhat analogous procedure (related to Schur-Jimbo duality~\cite{JimboHecke})
is to consider the space of vectors of the form
(\ref{distinctv}) below in a tensor representation of
$U_q(\widehat{\mathfrak{gl}}(r))$.  These vectors are like the monomials
$z_1 z_2 \cdots z_r$, because there are no repetitions allowed among the indexing
set. The R-matrix acts on these vectors and we will relate this fact to the
action of the intertwining operators on the Iwahori fixed vectors.

The larger quantum group $U_q(\widehat{\mathfrak{gl}}(r|1))$ will not appear in 
Theorem~\ref{RmatrixintertwiningIwahori}, only $U_q(\widehat{\mathfrak{gl}}(r))$.
Concretely, the reason for this is that $\plus$ spins do not appear on the right
boundary of our systems. We relate the $\plus$ spins with the Whittaker
functional in Remark~\ref{intWRmatrix}.  We therefore want to think of the
boundary condition for the left side of our lattice model as indexing the
unique Whittaker functional of the unramified principal series, while the
boundary conditions for the right side index the basis of the space of Iwahori
fixed vectors. This can be generalized to the parahoric case as well, per
Remark~\ref{remark:parahoric}.

In this section, $q$ will \textit{not} denote the cardinality of the residue field; instead it will stand for the quantum parameter $q$ in $U_q(\widehat{\mathfrak{gl}}(r))$ as is customary in the theory of quantum groups. We will continue to denote the cardinality of the residue field of $F$ by $v^{-1}$. With these conventions, we set $q^2=v$ for Theorem~\ref{RmatrixintertwiningIwahori}, consistent with our relation between quantum groups and residue field cardinalities in earlier sections (where we wrote~$U_{\sqrt{v}}$).

Consider the quantum loop group $U_q(\widehat{\mathfrak{gl}}(r))$, which is a quantization of a central extension of the loop algebra of $\mathfrak{gl}(r)$; for its formal definition see Section 12.2 in \cite{ChariPressleyBook}. The quantum loop group acts on the evaluation representation $V_r(z)$ for $z\in \mathbb{C}^\times$. 
The evaluation representation has a basis $\{v_i(z), 1 \leqslant i \leqslant r\}$. Denote $V_r(\z) := V_r(z_1) \otimes \cdots \otimes V_r(z_r)$.

There is an affine R-matrix, initially due to Jimbo~\cite{Jimbo}, that intertwines between tensor products of evaluation representations. We denote it by $R_q(\z^{\alpha_k}):V_r(z_k) \otimes V_r(z_{k+1}) \to V_r(z_{k+1}) \otimes V_r(z_k)$ and it is given by the following formula:
\begin{equation}\label{JimboRmatrix}
\begin{split}
& R_q(\z^{\alpha_k})= \frac{1}{1-q^2\z^{\alpha_k}} \sum_{1 \leqslant i \leqslant r} (q-\z^{\alpha_k}q^{-1}) e_{ii}\otimes e_{ii} +{} \\
& \frac{1}{1-q^2\z^{\alpha_k}} \sum_{i>j} \bigl( (-q^{-1})(1-\z^{\alpha_k})e_{ij}\otimes e_{ji} +(-q)(1-\z^{\alpha_k})e_{ji}\otimes e_{ij} \bigr) +{} \\ 
& \frac{1}{1-q^2\z^{\alpha_k}} \sum_{i>j} \bigl( (q-q^{-1})e_{jj}\otimes e_{ii} + \z^{\alpha_k}(q-q^{-1}) e_{ii} \otimes e_{jj} \bigr).
\end{split} 
\end{equation}
In the above, $e_{ij}$ stands for the $r \times r$ matrix with a $1$ in the $(i,j)$ entry and all other entries equal to $0$. It is a map $V_r(z_k) \to V_r(z_{k+1})$ if it is on the left of the tensor product and $V_r(z_{k+1}) \to V_r(z_k)$ if it is on the right of the tensor product. 

\begin{remark}
This is not exactly the R-matrix in \cite{Jimbo}; it is a Drinfeld
twist by $-q$. See \cite{BBBF} for a definition of the Drinfeld twist and
details on how it modifies an R-matrix. This particular Drinfeld twist
appears very often when one deals with $U_q(\widehat{\mathfrak{gl}}(r))$
lattice models.  Let us consider the weights in our
Figure~\ref{allcoloredrmat} in which we restrict to configurations with
all edges colored. This is the $U_q(\widehat{\mathfrak{gl}}(r))$ portion
of the larger $U_q(\widehat{\mathfrak{gl}}(r|1))$ R-matrix. These are
the same as the weights in Figure~2.1.8 of \cite{BorodinWheelerColored}
(up to a factor, and their $q$ is our $q^2$). Both R-matrices come from
the same Drinfeld twist of $U_q(\widehat{\mathfrak{gl}}(r))$. Throughout
this section, when we write $U_q(\widehat{\mathfrak{gl}}(r))$, we will in
fact refer to a Drinfeld twist of the usual affine quantum group that
produces the R-matrix $R_q(\z^{\alpha_k})$. 
\end{remark}

It is a standard fact in the theory of quantum groups that
$R_q(\z^{\alpha_k})$ is a $U_q(\widehat{\mathfrak{gl}}(r))$-module homomorphism. We will also denote by
$(R_q(\z^{\alpha_k}))_{k,k+1}:V_r(\z) \to V_r(s_k \z)$ the map that acts as
$R_q(\z^{\alpha_k})$ on the $k$ and $k+1$ tensor components of $V_r(\z)$ and the identity elsewhere.
 
Consider the R-matrix in Figure~\ref{coloredrmat} restricted to vertices where all edges are colored. It is preferable to normalize the weights of $R(z_{k+1}, z_{k})$ by dividing by $z_{k+1}$ so that they may be expressed in terms of $\z^{\alpha_k} = z_{k} / z_{k+1}$. We further normalize by a factor of $1-v\z^{\alpha_k}$. Denote the resulting restricted R-matrix by $R_\text{col}(\z^{\alpha_k})$, and similarly the restricted version of $R(z_{k}, z_{k+1})$ is then $R_\text{col}(\z^{-\alpha_k})$. The vertices and weights are pictured in Figure~\ref{allcoloredrmat}.

\begin{figure}[htb]
  \[\begin{array}{|c|c|c|}
  \hline
  \begin{tikzpicture}[scale=0.6]
      \path[fill=white] (0,0) circle (.4);\node at (0,0) {${\scriptstyle \z^{\alpha_k}}$};
      \spokea{red}{\scriptstyle c}\spokeb{red}{\scriptstyle c}\spokec{red}{\scriptstyle c}\spoked{red}{\scriptstyle c}
  \end{tikzpicture} &
 \begin{tikzpicture}[scale=0.6]
      \path[fill=white] (0,0) circle (.4);\node at (0,0) {${\scriptstyle \z^{\alpha_k}}$};
      \spokea{blue}{\scriptstyle c}\spokeb{red}{\scriptstyle d}\spokec{red}{\scriptstyle d}\spoked{blue}{\scriptstyle c}
  \end{tikzpicture} &
  \begin{tikzpicture}[scale=0.6]
      \path[fill=white] (0,0) circle (.4);\node at (0,0) {${\scriptstyle \z^{\alpha_k}}$};
      \spokea{blue}{\scriptstyle c}\spokeb{red}{\scriptstyle d}\spokec{blue}{\scriptstyle c}\spoked{red}{\scriptstyle d}
  \end{tikzpicture}\\
  \hline
  1 & \begin{array}{ll}\frac{(1-v)}{1-v\z^{\alpha_k}} &\text{if $c<d$}\\ \rule{0pt}{3ex}\frac{(1-v)\z^{\alpha_k}}{1-v\z^{\alpha_k}} &\text{if
      $c>d$}\end{array} &
  \begin{array}{ll} \frac{1-\z^{\alpha_k}}{1-v\z^{\alpha_k}} &\text{if $c>d$}\\ \rule{0pt}{3ex}\frac{v(1-\z^{\alpha_k})}{1-v\z^{\alpha_k}} &\text{if $c<d$}\end{array} \\
  \hline
  \end{array}\]
  \caption{The colored R-matrix 
  $R_\text{col}(\z^{\alpha_k})$. The colors $c,d \in \mathfrak{P}$ which we will represent by $\{1,2,\cdots,r \}$ are always distinct.}
  \label{allcoloredrmat}
\end{figure}

Let $\ii_0:=(1,2,\cdots,r)$ and let $\mathbb{I}$ be the set of all permutations of $\ii_0$.
Given $w\in S_r$ and $i \in \{1,\cdots,r\}$, let $w_i = w^{-1}(i)$ which equals the $i$-th component of the $r$-tuple $w(\ii_0)$. For example, if $s_1 s_2 \in S_3$, then $(s_1s_2)(\ii_0)=(312)$ and therefore $(s_1s_2)_3=2$. The following lemma is well-known:

\begin{lemma}\label{lemmaw(i)}
$\ell(s_k w) > \ell(w)$ if and only if $w_{k+1}>w_k$. 
\end{lemma}

In the notation of Section~\ref{Iwahoripreliminaries}, let $\bar{\mathcal{A}}^\z_{s_k} := \frac{(1-\z^{\alpha_k})}{1-v\z^{\alpha_k}} \mathcal{A}^\z_{s_k}$ where $\mathcal{A}^\z_{w} : I(\z) \to I(w\z)$ is the standard intertwining integral~\eqref{eq:intertwiner}.

\begin{proposition}\label{porposition:intertwiningwithweights}
Equation~\eqref{intertwiningactionexplicit} can be rewritten as 
\begin{equation}\label{intertwiningwithweights}
\bar{\mathcal{A}}_{s_k}^\z (\Phi^\z_{w}) =     
\wt\left(\begin{tikzpicture}[scale=0.5,baseline=-1.5mm]
      \path[fill=white] (0,0) circle (.4);\node at (0,0) {${\scriptstyle \z^{\alpha_k}}$};
      \bspokea{blue}{w_k}\bspokeb{red}{\scriptstyle
      w_{k+1}}\bspokec{red}{\scriptstyle w_{k+1}}\bspoked{blue}{w_k}
  \end{tikzpicture}\right)\Phi_w^{s_k \z} + 
  \wt\left(\begin{tikzpicture}[scale=0.5,baseline=-1.5mm]
      \path[fill=white] (0,0) circle (.4);\node at (0,0) {${\scriptstyle \z^{\alpha_k}}$};
      \bspokea{blue}{\scriptstyle w_{k+1}}\bspokeb{red}{w_k}\bspokec{blue}{\scriptstyle w_{k+1}}\bspoked{red}{w_k}
  \end{tikzpicture}\right)\Phi^{s_k \z}_{s_k w}.
\end{equation}
\end{proposition}
\begin{proof}
Let $s_k, w \in S_r$ and let us prove the statement for the case $\ell(s_k w) > \ell(w)$; the opposite case is similar. Lemma~\ref{lemmaw(i)} implies that $w_{k+1} > w_k$. By consulting the weights in Figure~\ref{allcoloredrmat}, we see that 
\[\wt\left(\begin{tikzpicture}[scale=0.5,baseline=-1.5mm]
      \path[fill=white] (0,0) circle (.4);\node at (0,0) {${\scriptstyle \z^{\alpha_k}}$};
      \bspokea{blue}{w_k}\bspokeb{red}{\scriptstyle w_{k+1}}\bspokec{red}{\scriptstyle w_{k+1}}\bspoked{blue}{w_{k}}
  \end{tikzpicture}\right) = \frac{1-v}{1-v\z^{\alpha_k}}, \qquad \wt\left(\begin{tikzpicture}[scale=0.5,baseline=-1.5mm]
      \path[fill=white] (0,0) circle (.4);\node at (0,0) {${\scriptstyle \z^{\alpha_k}}$};
      \bspokea{blue}{\scriptstyle w_{k+1}}\bspokeb{red}{w_k}\bspokec{blue}{\scriptstyle w_{k+1}}\bspoked{red}{w_k}
  \end{tikzpicture}\right) = \frac{1-\z^{\alpha_k}}{1-v\z^{\alpha_k}}.   \]
and the equivalence follows immediately by comparison with equation~\eqref{intertwiningactionexplicit}. 
\end{proof}

\begin{remark}
\label{ordercaveat}
We will represent the colors of the palette $\mathfrak{P}$ by the integers
$\{1, \ldots, r\}$. In this section, our color ordering will be opposite to
the ordering we used in previous sections, so $1>2>\cdots>r$. We use this
ordering because we want to match the largest color (in this case $1$) with
the highest weight vector of a quantum group representation which is
customarily denoted by $v_1$.
\end{remark}

Since the edges of the R-matrix $R_\text{col}(\z^{\alpha_k})$ are specified by these colors,
we can then think of it as a map $R_\text{col}(\z^{\alpha_k}) :U_r(z_k) \otimes U_r(z_{k+1}) \to U_r(z_{k+1}) \otimes U_r(z_k)$, where $U_r(z)$ is a vector space with formal basis elements $u_{i}(z)$ associated to the colors $i$ for $1 \leqslant i \leqslant r$. 
One can write the R-matrix in Figure~\ref{allcoloredrmat} in matrix form as follows:
\begin{multline}\label{eq:coloredRmatrix} 
R_\text{col}(\z^{\alpha_k})=
\sum_{1 \leq i \leq r} \wt\left(\begin{tikzpicture}[scale=0.5,baseline=-1.5mm]
      \path[fill=white] (0,0) circle (.4);\node at (0,0) {${\scriptstyle \z^{\alpha_k}}$};
      \bspokea{red}{i}\bspokeb{red}{i}\bspokec{red}{i}\bspoked{red}{i}
  \end{tikzpicture}\right) e_{ii}\otimes e_{ii} + \sum_{i>j} \wt\left(\begin{tikzpicture}[scale=0.5,baseline=-1.5mm]
      \path[fill=white] (0,0) circle (.4);\node at (0,0) {${\scriptstyle \z^{\alpha_k}}$};
      \bspokea{red}{i}\bspokeb{blue}{j}\bspokec{red}{i}\bspoked{blue}{j}
  \end{tikzpicture}\right) e_{ij}\otimes e_{ji} \\
   + \sum_{i>j} \wt\left(\begin{tikzpicture}[scale=0.5,baseline=-1.5mm]
      \path[fill=white] (0,0) circle (.4);\node at (0,0) {${\scriptstyle \z^{\alpha_k}}$};
      \bspokea{blue}{j}\bspokeb{red}{i}\bspokec{blue}{j}\bspoked{red}{i}
  \end{tikzpicture}\right) e_{ji}\otimes e_{ij}  
+ \sum_{i>j} \wt\left(\begin{tikzpicture}[scale=0.5,baseline=-1.5mm]
      \path[fill=white] (0,0) circle (.4);\node at (0,0) {${\scriptstyle \z^{\alpha_k}}$};
      \bspokea{blue}{j}\bspokeb{red}{i}\bspokec{red}{i}\bspoked{blue}{j}
  \end{tikzpicture}\right) e_{jj}\otimes e_{ii} + \sum_{i>j} \wt\left(\begin{tikzpicture}[scale=0.5,baseline=-1.5mm]
      \path[fill=white] (0,0) circle (.4);\node at (0,0) {${\scriptstyle \z^{\alpha_k}}$};
      \bspokea{red}{i}\bspokeb{blue}{j}\bspokec{blue}{j}\bspoked{red}{i}
  \end{tikzpicture}\right) e_{ii} \otimes e_{jj}.
\end{multline} 
Define $U_r(\z):=U_r(z_1) \otimes \cdots \otimes U_r(z_r)$ and, for $\ii =(i_1,\cdots,i_r)\in\mathbb{I}$, let
\begin{equation}
\label{distinctu}
u_\ii(\z):= u_{i_1}(z_1) \otimes \cdots \otimes u_{i_r}(z_r) \in U_r(\z).
\end{equation}
Let $U^\mathrm{alt}_r(\z)$ the subspace of $U_r(\z)$ with basis $\{u_{\ii}(\z), \ii \in \mathbb{I}\}$. The R-matrix $R_\text{col}(\z^{\alpha_k})_{k,k+1}: U_r(\z) \to U_r(s_k\z)$ restricts to $R_\text{col}(\z^{\alpha_k})_{k,k+1}: U^\mathrm{alt}_r(\z) \to U^\mathrm{alt}_r(s_k\z)$ by removing the first term of~\eqref{eq:coloredRmatrix}.
We similarly define
\begin{equation}
\label{distinctv}
v_\ii(\z):= v_{i_1}(z_1) \otimes \cdots \otimes v_{i_r}(z_r) \in V_r(\z),
\end{equation}
and denote by $V^\mathrm{alt}_r(\z)$ the subspace of $V_r(\z)$ with basis $\{v_{\ii}(\z), \ii \in \mathbb{I}\}$. Note that this is $\textit{not}$ a $U_q(\widehat{\mathfrak{gl}}(r))$ submodule of $V_r(\z)$. Even so, the restriction of $R_q(\z^{\alpha_k})_{k,k+1}$ to $V^\mathrm{alt}_r(\z)$ maps into $V^\mathrm{alt}_r(s_k \z)$ because $R_q(\z^{\alpha_1})$ maps $v_1(z_1) \otimes v_{2}(z_2)$ to a linear combination of $v_1(z_2) \otimes v_{2}(z_1)$ and $v_2(z_2) \otimes v_{1}(z_1)$ as seen in~\eqref{JimboRmatrix}.

Consider the following isomorphisms of vector spaces $\theta_{\z}:I(\z)^J \to U^\mathrm{alt}_r(\z)$ and $\xi_\z:U^\mathrm{alt}_r(\z) \to V^\mathrm{alt}_r(\z)$ defined by
\begin{equation}\label{vsiso}
\theta_\z(\Phi_{w}^\z):= u_{w (\ii_0)}(\z), \quad \xi_\z(u_{w (\ii_0)}(\z)):= v_{w (\ii_0)}(\z).
\end{equation}
In this basis we have that 
\begin{equation}
  \label{eq:theta-phi-sk}
  \theta_{s_k\z}(\Phi_{w}^{s_k\z}) = (e_{w_k,w_k} \otimes e_{w_{k+1},w_{k+1}}) \theta_\z(\Phi_{w}^{\z}), \quad\;
  \theta_{s_k\z}(\Phi_{s_kw}^{s_k\z}) = (e_{w_{k+1},w_{k}} \otimes e_{w_{k},w_{k+1}}) \theta_\z(\Phi_{w}^{\z}),
\end{equation}
where we have suppressed the notation that the operators $(e_{ab} \otimes e_{cd})$ here act on the $k$ and $k+1$ factors of the tensor product of $U_r^\mathrm{alt}(\z)$.
We may use this to rewrite the action of $\bar{\mathcal{A}}_{s_k}^\z (\Phi^\z_{w})$ in~\eqref{intertwiningwithweights} in terms of the action of $R_\text{col}(\z^{\alpha_k})$ in~\eqref{eq:coloredRmatrix}, or similarly the action of $R_q(\z^{\alpha_k})$ in~\eqref{JimboRmatrix}, to obtain the following result.

\begin{theorem}\label{RmatrixintertwiningIwahori}
The following diagram commutes assuming $v=q^2$:

\begin{center}
\begin{tikzcd}
I(\z)^J \arrow[r, "\theta_\z"] \arrow[d, "\bar{\mathcal{A}}^{\z}_{s_k}"]
& U^\mathrm{alt}_r(\z) \arrow[d, "R_{\operatorname{col}}(\z^{\alpha_k})_{k,k+1}"] \arrow[r, "\xi_\z"]
& V^\mathrm{alt}_r(\z) \arrow[d, "-q R_q(\z^{\alpha_k})_{k,k+1}"]
\\
I(s_k \z)^J  \arrow[r, "\theta_{s_k \z}"]
& U^\mathrm{alt}_r(s_k \z) \arrow[r, "\xi_{s_k \z}"]
& V^\mathrm{alt}_r(s_k \z)
\end{tikzcd}
\end{center}
\end{theorem} 
\begin{proof}
  The commutativity of the left side of the diagram follows from Proposition~\ref{porposition:intertwiningwithweights} noting that after the identification between $I(\z)^J$ and $U^\mathrm{alt}_r(\z)$ via $\theta_\z$, the action of $\bar{\mathcal{A}}^{\z}_{s_k}$ and $R_{\operatorname{col}}(\z^{\alpha_k})_{k,k+1}$ are the same as seen from equations~\eqref{intertwiningwithweights}, \eqref{eq:coloredRmatrix} and~\eqref{eq:theta-phi-sk}.  
The right half of the diagram commutes because the restriction of the two R-matrices are equal.
Indeed we have matching entries of the R-matrices in equations~\eqref{JimboRmatrix} and \eqref{eq:coloredRmatrix} \textit{except} for the first term which does not occur in the restriction.
For example, the last entry in equation~\eqref{eq:coloredRmatrix} has weight $\frac{(1-v)\z^{\alpha_k}}{1-v\z^{\alpha_k}}$, while the last entry in equation~\eqref{JimboRmatrix} multiplied by $-q$ has weight $\frac{(1-q^2)\z^{\alpha_k}}{1-q^2\z^{\alpha_k}}$. 
\end{proof}

\begin{remark}
The theorem above can also be proved if we set $q^2=v^{-1}$ (as opposed to $q^2 = v$). In that setting the quantum group needs to be Drinfeld twisted by $q^{-1}$ (as opposed to by $(-q)$). We made this choice to be in agreement with \cite[Theorem 1]{BBB}, but both choices might be useful when considering representation theoretic applications of this result.    
\end{remark}

\begin{remark}\label{remark:parahoric}
One can generalize Theorem~\ref{RmatrixintertwiningIwahori} to the parahoric setting by choosing elements of the form in equation~\eqref{distinctv} with repetitions determined by the chosen Levi subgroup or by the possible right boundary conditions of the lattice model associated to the parahoric Whittaker function. For example, if the Levi subgroup is 
$\GL_2 \times \GL_1$, the space of $K_\J$-fixed vectors in $I(\z)$ has a basis
indexed by elements in $W^\J =\{1,s_2,s_1s_2\}$ which is in bijection with
$\{v_{112}, v_{121}, v_{211}\}$. For each basis element we also have a
corresponding right boundary condition for the colored lattice model in
Section~\ref{sec:parahoric_systems}. Assume
$\textcolor{red}{R} > \textcolor{blue}{B}$. If (in accord with Remark~\ref{ordercaveat})
we let the integers $1$ and $2$ be identified with
the colors $\textcolor{red}{R}$ and $\textcolor{blue}{B}$ respectively,
$v_{112}$ corresponds to right boundary condition
$(\textcolor{red}{R},\textcolor{red}{R},\textcolor{blue}{B})$, $v_{121}$ to
$(\textcolor{red}{R},\textcolor{blue}{B}, \textcolor{red}{R})$, and $v_{211}$
to $(\textcolor{blue}{B},\textcolor{red}{R},\textcolor{red}{R})$.
\end{remark}

\begin{remark}\label{intWRmatrix}
Note that Proposition \ref{csinterident} can be rewritten as 
\[ 
\Omega_{s_k \z} \circ \mathcal{\bar A}_{s_k}^\z = \frac{1- v \z^{-\alpha_k}}{1-v\z^{\alpha_k}} \Omega_\z. 
\]
The factor above agrees up to a scalar with the fully uncolored Boltzmann weight for $R(z_{k+1}, z_k)$ in Figure~\ref{coloredrmat} and should be thought of as an entry of the R-matrix for the evaluation module of $U_q(\widehat{\mathfrak{gl}}(r|1))$. This remark and the previous theorems realize the action of intertwiners on both the space of Whittaker functionals and Iwahori (or parahoric) fixed vectors as entries in the $R$-matrix of $U_q(\widehat{\mathfrak{gl}}(r|1))$. 
\end{remark}

The purpose of this section was to give a dictionary between objects related to different areas of mathematics: intertwiners for $p$-adic groups, $R$-matrices for quantum groups, and $R$-matrices for lattice models. Let us now upgrade this dictionary by matching two techniques used in the theory of $p$-adic groups and lattice models. For lattice models we used the train argument to derive functional equations for the partition function. This argument was first used by Baxter to prove commutativity of transfer matrices in the early 80's and in our setting can be summarized as follows:
\begin{equation}\label{pictorial}
\begin{split}
  & \wt\left(\begin{tikzpicture}[scale=0.5,baseline=-1.5mm]
      \draw[line width=0.5mm] (-.75,-.75) node[circle,draw,fill=white,inner sep=0pt] {$\scriptstyle +$} -- (.75,.75) node[circle,draw,fill=white,inner sep=0pt] {$\scriptstyle +$};
      \draw[line width=0.5mm] (-.75,.75) node[circle,draw,fill=white,inner sep=0pt] {$\scriptstyle +$} -- (.75,-.75) node[circle,draw,fill=white,inner sep=0pt] {$\scriptstyle +$};
      \path[fill=white] (0,0) circle (.4);
      \node at (0,0) {${\scriptstyle \z^{\alpha_k}}$};
  \end{tikzpicture} \right)  
   Z (\mathfrak{S}_{\mathbf{z}, \lambda, w_1,w_2}) = \text{intermediate states appearing in the train argument} \\
    & =\wt\left(\begin{tikzpicture}[scale=0.5,baseline=-1.5mm]
      \path[fill=white] (0,0) circle (.4);\node at (0,0) {${\scriptstyle \z^{\alpha_k}}$};
      \bspokea{blue}{w_k}\bspokeb{red}{\scriptstyle
      w_{k+1}}\bspokec{red}{\scriptstyle w_{k+1}}\bspoked{blue}{w_k}
  \end{tikzpicture}\right)
  \  Z (\mathfrak{S}_{s_k \mathbf{z}, \lambda, w_1,w_2}) +   
  \wt\left(\begin{tikzpicture}[scale=0.5,baseline=-1.5mm]
      \path[fill=white] (0,0) circle (.4);\node at (0,0) {${\scriptstyle \z^{\alpha_k}}$};
      \bspokea{blue}{w_k}\bspokeb{red}{\scriptstyle
      w_{k+1}}\bspokec{red}{\scriptstyle w_{k+1}}\bspoked{blue}{w_k}
  \end{tikzpicture}\right) 
   Z (\mathfrak{S}_{s_k \z, \lambda, s_k w_1,w_2}).
  \end{split}
\end{equation}

For $p$-adic groups we used the intertwiner to derive functional equations for Iwahori Whittaker functions based on results and ideas of Casselman and Shalika. This process can be reimagined using Proposition~\ref{porposition:intertwiningwithweights} and Remark~\ref{intWRmatrix} as follows (we denote $\varpi^{-\lambda} w_2$ by $g$):  
\begin{equation}\label{pictorial2}
\begin{split}
  & \wt\left(\begin{tikzpicture}[scale=0.5,baseline=-1.5mm]
      \draw[line width=0.5mm] (-.75,-.75) node[circle,draw,fill=white,inner sep=0pt] {$\scriptstyle +$} -- (.75,.75) node[circle,draw,fill=white,inner sep=0pt] {$\scriptstyle +$};
      \draw[line width=0.5mm] (-.75,.75) node[circle,draw,fill=white,inner sep=0pt] {$\scriptstyle +$} -- (.75,-.75) node[circle,draw,fill=white,inner sep=0pt] {$\scriptstyle +$};
      \path[fill=white] (0,0) circle (.4);
      \node at (0,0) {${\scriptstyle \z^{\alpha_k}}$};
  \end{tikzpicture} \right)   \Omega_{\z}( \pi(g) \Phi_{w_1}^{\z} ) = \Omega_{s_k \z} \circ \bar{\mathcal{A}}_{s_k}^{\z} (  \pi(g) \Phi_{w_1}^{\z} )
    = \Omega_{s_k \z} (  \pi(g)  \bar{\mathcal{A}}_{s_k}^{\z} \Phi_{w_1}^{\z} )   \\
 &=  \Omega_{s_k \z} \left(  \pi(g) \left(
\wt\left(\begin{tikzpicture}[scale=0.5,baseline=-1.5mm]
      \path[fill=white] (0,0) circle (.4);\node at (0,0) {${\scriptstyle \z^{\alpha_k}}$};
      \bspokea{blue}{w_k}\bspokeb{red}{\scriptstyle
      w_{k+1}}\bspokec{red}{\scriptstyle w_{k+1}}\bspoked{blue}{w_k}
  \end{tikzpicture}\right)\Phi_{w_1}^{s_k \z} + 
  \wt\left(\begin{tikzpicture}[scale=0.5,baseline=-1.5mm]
      \path[fill=white] (0,0) circle (.4);\node at (0,0) {${\scriptstyle \z^{\alpha_k}}$};
      \bspokea{blue}{\scriptstyle w_{k+1}}\bspokeb{red}{w_k}\bspokec{blue}{\scriptstyle w_{k+1}}\bspoked{red}{w_k}
  \end{tikzpicture}\right)\Phi^{s_k \z}_{s_k w_1} \right) \right) \\
&=   
 \wt\left(\begin{tikzpicture}[scale=0.5,baseline=-1.5mm]
      \path[fill=white] (0,0) circle (.4);\node at (0,0) {${\scriptstyle \z^{\alpha_k}}$};
      \bspokea{blue}{w_k}\bspokeb{red}{\scriptstyle
      w_{k+1}}\bspokec{red}{\scriptstyle w_{k+1}}\bspoked{blue}{w_k}
  \end{tikzpicture}\right)
 \Omega_{s_k \z} (  \pi(g) \Phi_{w_1}^{s_k \z} ) +
\wt\left(\begin{tikzpicture}[scale=0.5,baseline=-1.5mm]
      \path[fill=white] (0,0) circle (.4);\node at (0,0) {${\scriptstyle \z^{\alpha_k}}$};
      \bspokea{blue}{w_k}\bspokeb{red}{\scriptstyle
      w_{k+1}}\bspokec{red}{\scriptstyle w_{k+1}}\bspoked{blue}{w_k}
  \end{tikzpicture}\right) 
  \Omega_{s_k \z} (  \pi(g)
\Phi_{s_k w_1}^{s_k \z} ). 
  \end{split}
\end{equation}
We see that the beginning and the end of the equations~\eqref{pictorial},~\eqref{pictorial2} correspond to each other by use of Theorem~\ref{coloredwhittaker} and Proposition~\ref{porposition:intertwiningwithweights}. The idea of both arguments is also the same. In the $p$-adic setting one moves the intertwiner from the `Whittaker side' to the `Iwahori side', while in the lattice model setting one moves the R-matrix from the left side (which corresponds to the Whittaker functionals as we argued before) to the right side (corresponding to the space of Iwahori fixed vectors).

This phenomena also appears in the theory of metaplectic spherical Whittaker functions for an $n$-fold metaplectic cover of $\GL_r(F)$, which can also be realized as partition functions of a solvable lattice model~\cite{BBB}. In that case the action of the intertwining integral on the space of Whittaker functionals is the Kazhdan-Patterson scattering matrix, which has been interpreted (up to a Drinfeld twist) as the $U_q(\widehat{\mathfrak{gl}}(n))$ R-matrix in \cite[Theorem 1]{BBB}, while the action of the intertwining integral on the spherical vector is the Gindinkin-Karpelevich factor which can be interpreted as the spin \plus part of a larger R-matrix. The train argument and the $p$-adic argument for producing functional equations work in the same way. This compelling connection between two a~priori different methods of argument should be useful in further relating the representation theories of $p$-adic and quantum groups.


\bibliographystyle{habbrv}
\bibliography{iwahori}

\end{document}